\pdfoutput=1
\documentclass[10pt]{amsart}

\usepackage{amsmath,amsthm,amssymb, bm}
\usepackage{mathtools}
\usepackage{mathrsfs}
\usepackage{microtype}
\usepackage{lmodern}
\usepackage[T1]{fontenc}
\usepackage{dsfont}
\usepackage[english]{babel}
\usepackage{pbox}
\usepackage{csquotes}
\usepackage{multirow}
\usepackage{enumerate}
\usepackage{tikz,pgffor}
\usepackage{tikz-cd}
\usepackage{footnote}
\usepackage[draft=false,hidelinks]{hyperref}

\def\dd{\,{\mathrm d}}

\newcommand{\SL}{\mathrm{SL}}
\newcommand{\Qd}{\mathbb{Q}}
\newcommand{\QQ}{\Qd}
\newcommand{\Qbar}{\overline{\Qd}}
\newcommand{\Vd}{\mathbb{V}}
\newcommand{\Zd}{\mathbb{Z}}
\newcommand{\Fd}{\mathbb{F}}
\newcommand{\Rd}{\mathbb{R}}
\newcommand{\Gd}{\mathbb{G}}
\newcommand{\Pd}{\mathbb{P}}
\newcommand{\Ad}{\mathbb{A}}
\newcommand{\xv}{\mathbf{x}}
\newcommand{\yv}{\mathbf{y}}
\newcommand{\zv}{\mathbf{z}}
\newcommand{\Vm}{\mathscr{V}}
\newcommand{\Rm}{\mathscr{R}}
\newcommand{\Pm}{\mathscr{P}}
\newcommand{\Nm}{\mathscr{N}}
\newcommand{\Mm}{\mathscr{M}}
\newcommand{\Om}{\mathscr{O}}
\newcommand{\Tm}{\mathscr{T}}
\newcommand{\rleft}{\mathopen{}\mathclose\bgroup\left}
\newcommand{\rright}{\aftergroup\egroup\right}
\newcommand{\Xf}{\mathfrak{X}}
\newcommand{\tX}{\widetilde{X}}
\newcommand{\tY}{\widetilde{Y}}
\newcommand{\tZ}{\widetilde{Z}}
\newcommand{\tpi}{\widetilde{\pi}}
\newcommand\Sigmamax{\Sigma_\mathrm{max}}

\newcommand{\Z}{\Zd}

\DeclareMathOperator{\Pic}{Pic}
\DeclareMathOperator{\Div}{div}
\DeclareMathOperator*{\Hom}{Hom}
\DeclareMathOperator*{\Supp}{supp}
\DeclareMathOperator*{\Spec}{Spec}
\DeclareMathOperator*{\Cl}{Cl}
\DeclareMathOperator{\Res}{Res}
\DeclareMathOperator{\rank}{rk}

\theoremstyle{plain}
\newtheorem{theorem}{Theorem}
\numberwithin{theorem}{section}
\newtheorem{lemma}[theorem]{Lemma}
\newtheorem{cor}[theorem]{Corollary}
\newtheorem{prop}[theorem]{Proposition}

\theoremstyle{definition}
\newtheorem{remark}[theorem]{Remark}

\numberwithin{equation}{section}

\begin{document}

\author[Blomer]{Valentin Blomer}

\address{Universit\"at Bonn, Mathematisches Institut, Endenicher Allee 60, 53115 Bonn, Germany}

\email{blomer@math.uni-bonn.de}

\author[Br\"udern]{J\"org Br\"udern}

\address{Universit\"at G\"ottingen, Mathematisches Institut, Bunsenstra{\ss}e 3--5, 37073 G\"ottingen, Germany}

\email{jbruede@gwdg.de}

\author[Derenthal]{Ulrich Derenthal}

\address{Leibniz Universit\"at Hannover, Institut f\"ur Algebra,
  Zahlentheorie und Diskrete Mathematik, Welfengarten 1, 30167
  Hannover, Germany}

\address{School of Mathematics, Institute for Advanced Study,
  1 Einstein Drive, Princeton, New Jersey, 08540, USA}

\email{derenthal@math.uni-hannover.de}

\author[Gagliardi]{Giuliano Gagliardi}

\address{Leibniz Universit\"at Hannover, Institut f\"ur Algebra,
  Zahlentheorie und Diskrete Mathematik, Welfengarten 1, 30167
  Hannover, Germany}

\email{gagliardi@math.uni-hannover.de}

\title[The Manin--Peyre conjecture for smooth spherical Fano threefolds]
{The Manin--Peyre conjecture\\for smooth spherical Fano threefolds}

\thanks{The first author was supported by the Deutsche
  Forschungsgemeinschaft (DFG, German Research Foundation) through
  Germany's Excellence Strategy, project number 390685813. The first
  two authors were supported by the DFG, project number 462335009. The
  third author was supported by the DFG, project number 255083470, and
  by the Charles Simonyi Endowment at the Institute for Advanced
  Study.}
 

\begin{abstract}
  The Manin--Peyre conjecture is established for smooth spherical Fano
  threefolds of semisimple rank one and type $N$. Together with the previously
  solved case $T$ and the toric cases, this covers all types of smooth spherical Fano
  threefolds. The case $N$ features a number of structural novelties; most
  notably, one may lose regularity of the ambient toric variety, the height
  conditions may contain fractional exponents, and it may be necessary to
  exclude a thin subset with exceptionally many rational points from the
  count, as otherwise Manin's conjecture in its original form would turn out
  to be incorrect.
\end{abstract}


\setcounter{tocdepth}{1}

\maketitle

\tableofcontents 

\section{Introduction}\label{sec:intro}

Let $G$ be a connected reductive group over $\Qbar$.  A normal
$G$-variety $X$ is called spherical if a Borel subgroup of $G$ has a
dense orbit in $X$. Spherical varieties are a very large and
interesting class of varieties that admit a combinatorial description
by spherical systems and colored fans \cite{lun01,bp15b, lv83}
generalizing the combinatorial description of toric varieties. Indeed,
if the acting group $G$ has semi-simple rank 0, then $G$ is a torus.
  
  If $G$ has semi-simple rank 1,  we may assume that $G = \SL_2\times \Gd_\mathrm{m}^r$ for some $r \geq 0$. 
 Let $G/H = (\SL_2 \times \Gd_\mathrm{m}^r)/H$ be the open
  orbit in $X$ and define  $H'$ by 
  $H'\times \Gd_\mathrm{m}^r = H\cdot \Gd_\mathrm{m}^r \subseteq \SL_2
  \times \Gd_\mathrm{m}^r$. Then the homogeneous space $\SL_2/H'$ is
  spherical, and there are three possible cases:
  \begin{itemize}
  \item  either $H'$ is a maximal torus (\emph{the case
    $T$}), 
    \item or $H'$ is  the normalizer of a maximal torus in $\SL_2$ (\emph{the
    case $N$}),
   \item or the homogeneous space $\SL_2/H'$ is horospherical, in which
  case $X$ is isomorphic (as an abstract variety) to a toric variety.
 \end{itemize} 
 In a recent paper \cite{BBDG}, the authors initiated a program to
 establish Manin's conjecture for (split models over $\Qd$ of) smooth
 spherical Fano varieties, based on (a) the combinatorial description
 of spherical varieties, (b) the universal torsor method, and (c)
 techniques from analytic number theory including the Hardy-Littlewood
 circle method and multiple Dirichlet series. In particular, in case
 $T$, we confirmed Manin's conjecture for threefolds as well as for
 some higher-dimensional varieties. Here we relied on Hofscheier's
 classification \cite{hofscheier} of smooth spherical Fano threefolds
 over $\overline{\Qd}$, which identifies four such cases having
 natural split models over $\mathbb{Q}$.
  
 In this paper we fully resolve the harder case $N$. This is achieved by a
 further development of the methods employed in \cite{BBDG}, and we proceed to
 describe the major new ingredients.

\medskip

 The first point concerns a general reformulation of Manin's conjecture as a counting problem. Let $X$ be a smooth split projective variety over $\Qd$ with big and
  semiample anticanonical class $\omega_X^\vee$ whose Picard group is free
  of finite rank. Assume that its Cox ring $\Rm(X)$ is finitely generated with
  precisely one relation with integral coefficients.  This  defines a canonical embedding of $X$ into a (not
necessarily complete) ambient toric variety $Y^\circ$; it  can be completed to a projective
toric variety $Y$ such that the natural map $\Cl Y \to \Cl X=\Pic X$ is an
isomorphism and $-K_X$ is big and semiample on $Y$. Under the assumption that
$Y$ is regular, Sections 2 -- 4 in \cite{BBDG} (culminating in
\cite[Propositions 3.7 \&  4.11]{BBDG}) provide a general scheme to
parametrize the rational points on $X$ in terms of the universal torsor and to
express the Manin--Peyre constant in terms of the Cox coordinates. The
corresponding counting problem is in many cases amenable to techniques from
analytic number theory. The assumption that $Y$ is regular holds for smooth
Fano threefolds of type $T$, but may fail in the case of type $N$. Part \ref{part1} of the present paper generalizes the passage to the universal torsor to varieties $X$ for which $Y$ is not necessarily regular. This result is independent of the theory of spherical varieties and should therefore have applications elsewhere. 

\medskip

The second point is of analytic nature. The universal torsor of spherical varieties of semi-simple rank 1 and type $T$ has a defining equation of the form
$$x_{11}x_{12} - x_{21} x_{22} - \text{ some monomial }  = 0,$$
which needs to be analyzed subject to rather complicated height conditions. The fact that we have a decoupled bilinear form in four variables is crucial for the method and allows in particular an auxiliary soft argument based on lattice considerations. For type $N$, the equation takes the form
$$x_{11}x_{12} - x_{21}^2  - \text{ some monomial }  = 0.$$
From an analytic perspective, this may be very delicate. A considerable portion of this paper is devoted to the investigation of the particular equation
\begin{equation}\label{new-eq}
x_{11}x_{12} - x_{21}^2 - x_{31}x_{32} x_{33}^2 = 0
\end{equation}
with variables constrained to dyadic boxes
$$|x_{11}| \asymp |x_{12}| \asymp |x_{21}| \asymp X, \quad |x_{31}| \asymp |x_{32}| \asymp Y, \quad |x_{33}|  \asymp X/Y.$$
Here $X$ is large, and $1 \leq Y \leq X$, and we need an asymptotic
formula for the number of integer solutions where the error term saves a fixed power of $\min(Y, X/Y)$. It is conceivable
that a modern variant of the circle method (like \cite{HB}) can handle this,
but this is not straightforward. There are considerable uniformity issues,
since we need to deal simultaneously with the cases when $Y$ is small, say
$Y = \exp((\log \log X)^2)$ (in which case the equation looks roughly like a
sum of two squares and a product), and $Y$ is large, say
$Y = X/\exp((\log \log X)^2)$ (in which case the equation looks roughly like a
sum of a square and two products). To keep track of uniformity, we will not
use the circle method directly but instead apply Poisson summation to
selected variables depending on the ranges of parameters. This is ultimately
more or less equivalent to the delta-symbol method of
Duke--Friedlander--Iwaniec \cite{DFI}, but it is in this case a more
convenient packaging.

The shape of equation \eqref{new-eq} offers a new feature that was not
present in \cite{BBDG}, and for which in fact very few examples are
known. In the special case where
$-x_{31}x_{32}$ is a square, the equation \eqref{new-eq} describes a
split quaternary quadratic form over $\mathbb{Q}$ of signature $(2, 2)$,
i.e., the sum of two hyperbolic planes. In this case it is well-known
(see e.g.\ \cite{HB}) that the asymptotic formula contains an
additional logarithm. In particular, Manin's conjecture in its
original form turns out to be wrong, and instead we need to prove a
``thin subset version'' of Manin's conjecture: we first remove a
certain portion from the variety with exceptionally
many points and then confirm the conjecture for the remaining
set.
More precisely, there should be a thin subset $T$ of the set of
rational points $X(\QQ)$ such that, for an anticanonical height $H$,
\begin{equation*}
  N_{X(\QQ) \setminus T, H}(B) := \#\{x \in X(\QQ) \setminus T \mid H(x) \le B\} = (1+o(1)) c B(\log B)^{\rank \Pic X-1}
\end{equation*}
as in Manin's conjecture with Peyre's constant $c$.
We refer the reader to the general discussion of this phenomenon in \cite{LST},
and to the (to our knowledge) only example of a smooth Fano variety \cite{BHB}
for which a thin version of Manin's conjecture has appeared in print. Soon after this work has been circulated in manuscript form,  Fano threefolds of Mori--Mukai type II.25 have been discussed in the preprint \cite{BBH}, providing yet another example where  the thin set version of the Manin--Peyre conjecture is true.

We now describe our results in more detail. As explained in \cite[\S 11]{BBDG}, there are three
smooth spherical Fano threefolds over $N$ over $\Qbar$ that are neither
horospherical nor equivariant compactifications of $\Gd_\mathrm{a}^3$; we
construct split models $X_1,X_2,X_3$ over $\Qd$ as in Table~\ref{tab:classification_spherical1}.

\begin{table}[ht]
  \centering
  \begin{tabular}[ht]{ccccccc}
    \hline
    rk Pic & Hofscheier & Mori--Mukai & torsor equation & label\\
    \hline\hline
     2 & $N_1 8$ & II.29 &  {$x_{11}x_{12}-x_{21}^2-x_{31}x_{32}x_{33}^2$} & variety $X_1$\\
      3 & $N_0 3$ & III.22  &  $x_{11}x_{12}-x_{21}^2-x_{31}x_{32}$ & variety $X_2$ \\
    3 & $N_1 9$ & III.19 &  $x_{11}x_{12}-x_{21}^2-x_{31}x_{32}$ & variety $X_3$\\
    \hline
    \end{tabular}
  \caption{Smooth Fano threefolds of type $N$ that are spherical, but not horospherical}
  \label{tab:classification_spherical1}
\end{table}

More precisely, let $X_1$ the blow-up of the quadric
$Q = \Vd(z_{11}z_{12} - z_{21}^2 - z_{31}z_{32}) \subset \Pd^4_{\Qd}$ in the
conic $C_{33} = \Vd(z_{31}, z_{32})$. This is a smooth Fano threefold of type
II.29 in the Mori--Mukai classification. We have a fibration $X_1 \to \Pd^1$
that is defined in Cox coordinates by
$(x_{11}:\dots:x_{33})\mapsto (x_{31}:x_{32})$. As explained above, we
must remove the thin subset
\begin{equation*}
  T_1 = \{(x_{11}:\dots:x_{33}) \in X_1(\Qd) \mid x_{31}x_{32} = -\square\ \ \text{or}\ \ x_{11}x_{12}x_{21}x_{31}x_{32} x_{33} = 0\}.
\end{equation*}

Let $W_2=\Pd^1_\Qd \times \Pd^2_\Qd$ with coordinates $(z_{01}:z_{02})$ and
$(z_{11}:z_{12}:z_{21})$.  Let $C_{32}$ be the curve
$\Vd(z_{02},z_{11}z_{12}-z_{21}^2)$ on $W_2$ and let $X_2$ be the blow-up of
$W_2$ in $C_{32}$. This is a smooth Fano threefold of type III.22. We will see
later that the height conditions (cf.\ \eqref{height1} below) contain fractional exponents
$\alpha^{\nu}_{ij}$; see also \cite{BT}.
Let $T_2 \subset X_2(\QQ)$ be the set of rational points where at least one
Cox coordinate is zero.

Let $X_3$ be the blow-up of the quadric
$Q = \Vd(z_{11}z_{12} - z_{21}^2 - z_{31}z_{32}) \subset \Pd^4_{\Qd}$ in the
points $P_{01} = \Vd(z_{11}, z_{12}, z_{21}, z_{31})$ and
$P_{02} = \Vd(z_{11}, z_{12}, z_{21}, z_{32})$. Its type is III.19.  As
before, let $T_3 \subset X_3(\QQ)$ be the set of rational points where at
least one Cox coordinate is zero.

In Sections~\ref{sec:heights} and \ref{sec:counting_problems}, we will
define natural anticanonical height functions
$H_j : X_j (\mathbb{Q}) \rightarrow \mathbb{R}$, $j=1, 2, 3$, using the
anticanonical monomials in their Cox rings.  We write
$N_j (B) = N_{X_j(\QQ) \setminus T_j, H_j}(B)$.
  
\begin{theorem}\label{thm1}
  The Manin--Peyre conjecture holds for the smooth spherical Fano threefolds
  $X_2,X_3$ of semisimple rank one and type $N$, and a thin version of the
  Manin--Peyre conjecture holds for $X_1$.  More precisely, there exist
  explicit constants $C_1,C_2,C_3$ such that
  $$N_j(B) = (1 + o(1))C_j B (\log B)^{\rank \Pic X_j -1} $$
  for $1 \leq j \leq 3$. 
  The values of $C_j$ are the ones predicted by Peyre. 
\end{theorem}
 
Together with previous results, this covers all types of smooth spherical Fano
threefolds.

\subsection*{Acknowledgements}

The authors thank the anonymous referee
for useful remarks and suggestions.

\part{Metrics and heights via Cox rings and universal torsors}\label{part1}

Universal torsors and Cox rings were introduced and studied by
Colliot-Th\'el\`ene and Sansuc \cite{CTS,CTS2} and Cox \cite{Cox}.

If a variety $X$ has a finitely generated Cox ring with one relation, this
gives a description of $X$ as a hypersurface in a toric variety $Y$. The
description of height and Tamagawa measures in \cite[Part 1]{BBDG} relies on
the assumption \cite[(2.3)]{BBDG} that $Y$ can be chosen to be regular, which
does not hold in our examples $X_2,X_3$. Here, we describe
one approach how to circumvent
this problem; see also \cite{BT}.

Our constructions of metrizations, heights, and Tamagawa measures on
universal torsors follow the work of Salberger \cite{MR1679841}; see
also \cite{BBS1} and \cite{BBDG}. This should be compared to the
closely related work of Peyre on universal torsors for Manin's
conjecture \cite{Pey2, Pey3, Pey4}.

\section{Heights and parametrization}\label{sec:desing}

We start with a situation similar to (but in several respects more general
than) \cite[Section~2]{BBDG}. Let $Y^\circ$ be a smooth split toric variety
over $\Qd$.  We do not assume that $Y^\circ$ is complete, but we still assume
the weaker property that $Y^\circ$ has only constant regular functions.  Let
\begin{equation}\label{eq:coxamb}
  \Rm(Y^\circ) \cong \Qd[x_1,\dots,x_J]
\end{equation}
be its Cox ring, where $x_1, \dots, x_J$ correspond to the torus
invariant prime divisors in $Y^\circ$. Let $0 \ne \Phi \in \Rm(Y^\circ)$ be
a homogeneous equation, and let $X \subseteq Y^\circ$ be the
corresponding subvariety.

We assume that $X$ is smooth and projective, with big and semiample
anticanonical class $-K_X$. Moreover, we assume that every torus orbit in
$Y^\circ$ meets $X$. We also assume that the pullback map
$\Pic Y^\circ \to \Pic X$ sends a big and semiample divisor class $L^\circ$ to
$-lK_X$ for some $l \in \Zd_{\ge 0}$.  Finally, we assume
$\Phi \in \Zd[x_1,\dots,x_J]$.

\begin{remark}
  \label{rem:smfano1}
  A situation as above can be naturally obtained starting from a
  smooth split projective variety over $\Qd$ with big and semiample
  anticanonical class $-K_X$ and finitely generated Cox ring
\begin{equation*}
  \Rm(X) \cong \Qd[x_1,\dots,x_J]/(\Phi),
\end{equation*}
where $x_1,\dots,x_J$ is a system of pairwise nonassociated $\Pic
X$-prime generators and $\Phi \ne 0$. By \cite[3.2.5]{adhl15}, there
exists a canonical embedding into an ambient toric variety $Y^\circ$,
and it satisfies all the above assumptions. Moreover, the pullback map
$\Pic Y^\circ \to \Pic X$ is an isomorphism, and we may take
$L^\circ = -K_X$ under this identification.
\end{remark}

If $\Sigma$ is the fan of any toric variety, we write $\Sigmamax$ for
the set of maximal cones and $\Sigma(1)$ for the set of rays.

Let $\Sigma^\circ$ be the fan of $Y^\circ$. The generators
$x_1,\dots,x_J \in \Rm(Y^\circ)$ are in bijection to the rays $\rho
\in \Sigma^\circ(1)$; we also write $x_\rho$ for $x_i$ corresponding
to $\rho$.

Let $Y$ be a completion of $Y^\circ$ such that the pullback map
$\Cl Y \to \Cl Y^\circ=\Pic Y^\circ$ is an isomorphism and $L^\circ$ is big
and semiample on $Y$; let $\Sigma$ be the fan of this toric variety $Y$. We
have $\Sigma(1) = \Sigma^\circ(1)$ and $\Rm(Y) = \Rm(Y^\circ)$. For example,
we may choose the unique $Y$ such that $L^\circ$ is ample on $Y$, which exists
by \cite[Proposition~2.4.2.6]{adhl15}; we call this $Y$ the \emph{standard
  small completion} of~$Y^\circ$.

We do not assume that $Y$ is regular (in contrast to
\cite[Section~2]{BBDG}). Let
\begin{equation*}
\rho: Y'' \to Y
\end{equation*}
be a toric desingularization of $Y$ that does not change the smooth locus of
$Y$.  Such a desingularization can be obtained by suitably subdividing the
singular cones in the fan $\Sigma$ of $Y$ into a smooth fan $\Sigma''$.

Let $Y' \subset Y''$ be a toric subvariety with $Y^\circ \subset Y'$.
This means that the fan $\Sigma'$ of $Y'$ is a subfan of $\Sigma''$
and contains $\Sigma^\circ$. If we write $\rho_{J+1}, \dots,
\rho_{J'}$ for the rays in $\Sigma'(1) \setminus \Sigma(1)$, we can
write
\begin{equation*}
\Rm(Y') = \Qd[y_1,\dots,y_J,y_{J+1},\dots,y_{J'}]
\end{equation*}
for the Cox ring of $Y'$, again with the correspondence between rays $\rho_i$
and generators $y_i$.

Since $X$ is smooth, we can identify it with its strict transform under
$\rho: Y'' \to Y$; it is a hypersurface in $Y'$ defined by a homogeneous
equation $\Phi'$ that is obtained from $\Phi$ by homogenizing (therefore,
$y_i \mapsto x_i$ for $i \le J$ and $y_i \mapsto 1$ for $i>J$ turns $\Phi'$
into $\Phi$). We obtain the following commutative diagram:
\[
\begin{tikzcd}
X\ar[r, hook]\ar[d,equal] & Y'\ar[r, hook] & Y''\ar[d, two heads]\\
X\ar[r, hook] & Y^\circ\ar[u, hook]\ar[r, hook] & Y\text{.}
\end{tikzcd}
\]
Below, we will regard $X$ mostly as embedded into $Y'$.

For simplicity, we assume that every cone in $\Sigma'$ is the face of
a maximal cone.

Let $U$ be the open torus in $Y'$. For each $\rho \in \Sigma'(1)$,
we have a $U$-invariant Weil divisor $D_\rho$ defined by $y_\rho$ of class
$[D_\rho]=\deg(y_\rho) \in \Pic Y'$.  Let
$D_0:=\sum_{\rho \in \Sigma'(1)} D_\rho$, which is an effective divisor of
class $[D_0]=-K_{Y'}$.

Let $S \subset \Sigma'(1)$ be such that $\{\deg(y_\rho) \mid \rho \notin S\}$
is a basis of $\Pic Y'$. This is true if and only if $S$ is a basis of $N$.
Therefore, we can write
$\deg(y_{\rho'}) = \sum_{\rho \notin S} a^S_{\rho',\rho} \deg(y_{\rho})$
for each $\rho' \in \Sigma'(1)$ with certain
$a^S_{\rho',\rho} \in \Zd$.
We define the rational section
\[z^S_{\rho'}:=y_{\rho'}/\prod_{\rho \notin S}
y_\rho^{a^S_{\rho',\rho}}\] of degree
$0 \in \Pic Y'$, with $z^S_\rho=1$ for $\rho \notin S$. The
$z^S_\rho$ for $\rho \in S$ define
a chart
\begin{equation*}
  f^S : U^S \to \Ad^{S}_\QQ,\quad P \mapsto (z^S_\rho(P))_{\rho \in S}
\end{equation*}
where $U^S$ is the open subset of $Y'$ where $y_\rho \ne 0$ for all
$\rho \notin S$ (i.e., the complement of $\bigcup_{\rho \notin S} D_\rho$ in
$Y'$). Note that $f^S$ is an isomorphism if and only if $S = \sigma(1)$ for
some $\sigma \in \Sigmamax'$.
For the open subset $X^S:=X \cap U^S$ of $X$, the image
$f^S(X^S) \subset \Ad^S_\QQ$ is defined by
\begin{equation*}
  \Phi^S:=\Phi'(z^S_\rho)=
  \Phi'(y_\rho)/\prod_{\rho\notin S}y_\rho^{b^S_\rho}
\end{equation*}
(with $b^S_\rho \in \Zd$ satisfying
$\deg\Phi'=\sum_{\rho\notin S} b^S_\rho\deg(y_\rho)$)
since $y_\rho \ne 0$ on $U^S$ for $\rho \notin S$.

\subsection{Universal torsors and models}

For universal torsors and Cox rings of toric varieties, see \cite[\S
4]{CTS}, \cite{Cox}, \cite[\S 8]{MR1679841}.

Let $\pi'\colon Y_0' \to Y'$ be the universal torsor as in \cite[\S 8]{MR1679841}.
Then the restriction $\pi'\colon X_0\to X$ to the preimage of $X \subset Y'$
is a torsor, but not necessarily universal since
the acting torus (dual to $\Pic Y'$) may be too large.
The toric varieties
$\Ad_\Qd^{J'} = \Ad_\Qd^{\Sigma'(1)} =: Y'_1 \supset Y'_0 \to Y'$ have the fans
$\Sigma'_1 \supset \Sigma'_0 \to \Sigma'$. Here, the sets of rays
$\Sigma'_1(1)=\Sigma'_0(1)$ are in natural bijection to $\Sigma'(1)$.
The $r'$ irreducible components of $Z'_Y = Y'_1 \setminus Y'_0$ are defined by the vanishing of
$x_\rho$ for all $\rho \in S'_j$, where the \emph{primitive collections}
\begin{equation}\label{eq:pcoll}
  S'_1,\dots,S'_{r'} \subseteq \Sigma'(1)
\end{equation}
are all sets with the following property: $S'_j \not\subseteq \sigma(1)$ for all
$\sigma \in \Sigma'$, but for every proper subset $S_j''$ of $S'_j$, there is
a $\sigma \in \Sigma'$ with $S_j'' \subseteq \sigma(1)$.

From the fans and their maps, we may construct $\Zd$-models
$\tpi'\colon \tY'_1 \setminus \tZ'_Y = \tY'_0 \to \tY'$, again as in \cite[\S 8]{MR1679841}.
By our assumption that $\Phi$ has integral coefficients, we
obtain $\Zd$-models $\tpi' \colon \tX_1
\setminus \tZ_X = \tX_0 \to \tX$ of $\pi'\colon X_1
\setminus Z_X = X_0 \to X$ by restricting everything to
$\Phi' = 0$.

We assume:
\begin{equation}\label{eq:proper_z}
  \text{The toric variety $Y'$ is chosen such that $\tX$ is proper over $\Spec \Zd$.}
\end{equation}

This is always possible since for $Y' = Y''$ the scheme $\tY'$ is projective over $\Spec \Zd$.

\begin{prop}\label{prop:lift_to_torsor}
  We have
  \begin{align*}
    \tX_0(\Zd) &= \{\xv=(x_\rho)_{\rho \in \Sigma'(1)} \in \Zd^{\Sigma'(1)} :
                 \Phi(\xv)=0,\ \gcd\{x_\rho : \rho \in S'_j\}=1  \text{ {\rm
                 for all  }} j=1,\dots,r'\},\\
    \tX_0(\Zd_p) &= \{\xv=(x_\rho)_{\rho \in \Sigma'(1)} \in \Zd_p^{\Sigma'(1)}
                   : \Phi(\xv)=0,\ p \nmid \gcd\{x_\rho : \rho \in S'_j\}
                   \text{ {\rm  for all }}j=1,\dots,r'\}.
  \end{align*}
  The map $\tpi'$ induces a $2^{\rank \Pic Y'} : 1$-map $\tX_0(\Zd) \to \tX(\Zd)=X(\Qd)$.
\end{prop}
\begin{proof}
  The proof is as in \cite[Proposition~2.2]{BBDG} by
  \eqref{eq:proper_z}. The referee kindly pointed out that an argument
  is also contained in the work of Peyre \cite{Pey2,Pey3,Pey4}.
\end{proof}

\subsection{Metrization of $\omega_X^{-1}$ via Poincar\'e residues} \label{subsec:metr}

Let $L$ be any big and semiample divisor class on $Y'$ such that
$L|_{X} = -lK_X$. Then there exists a uniquely determined divisor $E$
with support $\Supp E$ in the exceptional locus of $\rho$ such that
\begin{align*}
 L = -lK_{Y'} - l[X] + [E].
\end{align*}

The following lemma shows that (after possibly enlarging $l$) such an
$L$ can always be found.
\begin{lemma}
  \label{lemma:exL}
  After suitably enlarging $l$, there exists a big and semiample
  divisor class on $Y'$ such that $L|_{X} = -lK_X$.
\end{lemma}
\begin{proof}
Since $L^\circ$ is ample, it is $\Qd$-Cartier on $Y$. Hence, after
replacing $l$ by a positive multiple, we may assume that $L^\circ$ is
Cartier on $Y$. Let $L''$ be the pullback $\rho^*(L^\circ)$. Then
$L''$ is big and semiample on $Y''$, and moreover $L''|_{X} = -lK_X$.
The same is then true for $L = L''|_{Y'}$.
\end{proof}

As before, let $S \subset \Sigma'(1)$ be such that $\{\deg(y_\rho)
\mid \rho \notin S\}$ is a basis of $\Pic Y'$. Hence there is a
unique (not necessarily effective) Weil divisor
$D(S)=\sum_{\rho \notin S} a^S_\rho D_\rho$ of class $-K_{Y'}-[X]$ whose
support is contained in $\bigcup_{\rho \notin S} D_\rho$. Let
$y^{D(S)}$ be the corresponding rational monomial in $\mathcal{R}(Y')$
of degree $-K_{Y'}-[X]$. The characters defined by $z^S_\rho$ for $\rho \in
S$ form a basis of $M=\Hom(U,\mathbb{G}_\mathrm{m})$.

By \cite[Proposition~8.2.3]{MR2810322}, we have a global nowhere vanishing
section $s_{Y'}$ of $\omega_{Y'}(D_0)$ (defined up to sign, independent of $S$; we have
$s_{Y'} = \Omega_0 / \prod_{\rho \in \Sigma'(1)} y_\rho$ for $\Omega_0$
from \cite[(8.2.3)]{MR2810322}) whose restriction to the open subset $U^S$ is
$\pm \bigwedge_{\rho \in S} \frac{\dd z^S_\rho}{z^S_\rho}$

With $y^{D_0} = \prod_{\rho \in \Sigma'(1)} y_\rho$, for each $S$ as above, 
\begin{equation*}
  \varpi^S := \frac{y^{D_0}}{y^{D(S)}\Phi} s_{Y'} \in \Gamma(Y',\omega_{Y'}(D(S)+X))
\end{equation*}
defines a nowhere vanishing global section. On $U^S$, we have
\begin{equation*}
  \varpi^S  = \frac{\pm 1}{\Phi^S} \bigwedge_{\rho \in S} \dd z^S_\rho \in
  \Gamma(U^S, \omega_{Y'}(X)).
\end{equation*}

Let $\Pm^l$ be a set of polynomials $F
\in \QQ[y_1,\dots,y_{J'}]$ of degree $L$. For each polynomial $F \in
\Pm^l$ of degree $L$, let $D(F)$ be the effective divisor on
$Y'$ of class $\deg F = L$ defined by $F$ (in Cox coordinates). If $X
\subset \Supp D(F)$, then we remove $F$ from our set $\Pm^l$;
clearly, this does not change the results that we want to prove. For
$X \not\subset \Supp D(F)$, we define
\begin{equation*}
  \varpi_F := \frac{y^{lD_0+E}}{F\cdot\Phi^l} s_{Y'}^l \in \Gamma(Y',\omega_{Y'}^l(D(F)+lX-E)).
\end{equation*}

We have the Poincar\'e residue map
$\Res : \omega_{Y'}(X) \to \iota'_*\omega_X$ of $\Om_{Y'}$-modules (where
$\iota': X \to Y'$ is the inclusion). On the smooth open subset $U^S$ of
$Y'$, it maps $\varpi^S \in \Gamma(U^S,\omega_{Y'}(X))$ to
$\Res(\varpi^S) \in \Gamma(U^S,\iota'_*\omega_X) = \Gamma(X^S,\omega_X)$,
which is
\begin{equation}\label{eq:residue}
  \Res(\varpi^S) = \frac{\pm 1}{\partial \Phi^S/\partial
    z^S_{\rho_0}} \bigwedge_{\rho \in S\setminus\{\rho_0\}} \dd z^S_\rho
\end{equation}
on the open subset of $X^S$ where
$\partial \Phi_\rho/\partial z^S_{\rho_0} \ne 0$, for any
$\rho_0 \in S$. Furthermore,
$\Res^l : \omega_{Y'}^l(lX) \to \iota'_*\omega_X^l$ sends $\varpi_F \in
\Gamma(U_F,\omega_{Y'}^l(lX - E))$ to $\Res^l(\varpi_F) \in
\Gamma(U_F,\iota'_*\omega_X^l) = \Gamma(X_F,\omega_X^l)$, where $U_F$ is the
complement of $D(F)$ in $Y'$, and $X_F = X \cap U_F$.

\begin{lemma}
  The section $\Res(\varpi^S)$, $\Res^l(\varpi_F)$ extends uniquely to a nowhere
  vanishing global section of $\omega_X(D(S)\cap X)$, $\omega_X^l(D(F)\cap X)$, respectively.
\end{lemma}

\begin{proof}
  For $\Res(\varpi^S)$, this is as in \cite[Lemma~2.3]{BBDG}, i.e., similar
  to \cite[Lemma~13]{BBS1}. The computation for $\Res^l(\varpi_F)$
  is analogous, using $X \not\subset \Supp D(F)$ and $E \cap X = \emptyset$.
\end{proof}

Therefore, $\tau^S:=\Res(\varpi^S)^{-1}$,
$\tau_F:=\Res^l(\varpi_F)^{-1}$ are global nowhere vanishing sections of
$\omega_X^{-1}(-D(S)\cap X)$, $\omega_X^{-l}(-D(F)\cap X)$, which we can also view as
a global section of $\omega_X^{-1}$, $\omega_X^{-l}$, respectively.

\begin{lemma}\label{lem:tau}
  The sections $\tau^S \in \Gamma(X,\omega_X^{-1})$,
  $\tau_F \in \Gamma(X,\omega_X^{-l})$ do not vanish anywhere on $X^S$,
  $X_F$, respectively.
\end{lemma}

\begin{proof}
  The support of $D(S) \cap X$ is contained in $X \cap \bigcup_{\rho
    \notin S} D_\rho$, which is the complement of $X^s$. Moreover,
  $D_F \cap X$ is the complement of $X_F$.
\end{proof}

From now on, we assume:
\begin{equation}\label{eq:monicbpf}
  \parbox{35em}{$\Pm^l$ only contains monic
    monomials (of degree $L$), and for each
    $\sigma \in \Sigmamax'$ there exists $F \in \Pm^l$ such that
    $\Supp \Div F$ does not meet $U_\sigma$.}
\end{equation}
Then the set $\Pm^l$ is in particular basepoint-free.

We define a $v$-adic norm/metric on $\omega_X^{-1}$ by
\begin{equation*}
  \|\tau(P)\|_v := \min_{F \in \Pm^l : P \notin D(F)} \left|\frac{\tau^l}{\tau_F}(P)\right|_v^{1/l}
\end{equation*}
for a local section $\tau$ of $\omega_X^{-1}$ not vanishing in
$P \in X(\QQ_v)$.

\begin{lemma}
  Let $p$ be a prime such that $\tX$ is smooth over $\Zd_p$. On
  $\omega_X^{-l}$, the $p$-adic norm $\|\cdot\|_p$ defined by
  \begin{equation*}
    \|\tau(P)\|_p := \min_{F \in \Pm^l : P \notin D(F)} \left|\frac{\tau}{\tau_F}(P)\right|_p
  \end{equation*}
  for a local section $\tau$ of $\omega_X^{-l}$ not vanishing in
  $P \in X(\QQ_p)$ coincides with the model norm $\|\cdot\|_p^*$
  determined by $\tX$ over $\Zd_p$ as in
  \cite[Definition~2.9]{MR1679841}.
\end{lemma}

\begin{proof}
  See \cite[Lemma~3.3]{BBDG}. For such a $\tau$ not vanishing in $P$, let $Q
  \in \Pm^l$ be such that
  $|(\tau^Q/\tau)(P)|_p = \max_{F \in \Pm^l}
  |(\tau^F/\tau)(P)|_p$, which is positive by Lemma~\ref{lem:tau} and the
  fact that the set $\Pm^l$ is basepoint-free. Hence $\tau^Q$ does not
  vanish in $P$, and
  \begin{equation*}
    \|\tau^Q(P)\|_p^{-1} = \max_{F \in \Pm^l}
    \left|\frac{\tau^F}{\tau^Q}(P)\right|_p
    = \max_{F \in \Pm^l}
    \frac{|(\tau^F/\tau)(P)|_p}{|(\tau^Q/\tau)(P)|_p} = 1.
  \end{equation*}

  For each $F \in \Pm^l$, the section
  $\tau^F$ extends to a global section $\widetilde\tau^F$ of
  $\omega_{\tX/\Zd_p}^{-l}$, and $\omega_{\tX/\Zd_p}^{-l}$ is
  generated by the set of all these $\widetilde\tau^F$ as an
  $\Om_{\tX}$-module. The reason is that everything required for the
  definition of $\tau^F$ above can also be defined over $\Zd_p$.
  For the existence of the Poincar\'e residue map in this case, we
  refer to \cite[Definition~4.1]{kk}.

  For every $F \in \Pm^l$, we have $\left|\frac{\tau^F}{\tau^Q}(P)\right|_p \le 1$
  as in the computation above. This implies $\tau^F(P) = a_F \tau^\xi(P)$ for some
  $a_F \in \Zd_p$ in the $\Qd_p$-module $\omega_X^{-1}(P)$, and hence also
  $\widetilde\tau^\sigma(P) = a_F \widetilde\tau^Q(P)$ in the $\Zd_p$-module
  $P^*(\omega_{\tX/\Zd_p}^{-1})$, which shows that $P^*(\omega_{\tX/\Zd_p}^{-1})$
  is generated by $\widetilde\tau^Q(P)$. Hence $\|\tau^Q(P)\|_p^*=1$ by
  definition of the model norm. We conclude
  \begin{equation*}
    \|\tau(P)\|_p = |(\tau/\tau^Q)(P)|_p \cdot \|\tau^Q(P)\|_p =
    |(\tau/\tau^Q)(P)|_p \cdot \|\tau^Q(P)\|_p^* = \|\tau(P)\|_p^*. \qedhere
  \end{equation*}
\end{proof}

\subsection{Height functions}\label{sec:heights}

For $P \in X(\QQ)$, we define
\begin{equation}\label{eq:height_definition}
  H(P):=\prod_v \|\tau(P)\|_v^{-1}
\end{equation}
for a local section $\tau$ of $\omega_X^{-1}$ not vanishing in
$P$.

\begin{remark}\label{rem:rational_functions}
  Let $F$, $F_0$ be homogeneous
  elements of the same degree in the Cox ring of $Y'$. If
  $F_0$ does not vanish in $P$, then $F/F_0$ can be regarded as a
  rational function on $X$ that can be evaluated in $P \in X(\QQ)$.
\end{remark}

\begin{lemma}\label{lem:height}
  We have
  \begin{equation*}
    H(P) = \left(\prod_v \max_{F \in \Pm^l} \left|\frac{F}{F_0}(P)\right|_v\right)^{1/l}
  \end{equation*}
  for any polynomial $F_0$ of degree $L$ not vanishing in $P$.
\end{lemma}

\begin{proof}
  We have $P \in X^S(\QQ)$ for some $S$ as above. We can compute $H(P)$ with
  $\tau:=\tau^S$ by Lemma~\ref{lem:tau}. Applying the $\Om_{Y'}$-module
  homomorphism $\Res$ to $\varpi_F = F^{-1}y^{lD(S)+E}(\varpi^S)^l$ shows
  $\tau_F = Fy^{-lD(S)-E}(\tau^S)^l$, hence
  \begin{equation}\label{eq:norm_max}
    \|\tau^S(P)\|_v^{-1}
    = \max_{F \in \Pm^l} \left|\frac{\tau_F}{(\tau^S)^l}(P)\right|_v^{1/l}
    = \max_{F \in \Pm^l}\left|\frac{F}{y^{lD(S)+E}}(P)\right|_v^{1/l},
  \end{equation}
  which is our claim in the case $F_0:=y^{lD(S)+E}$. The general case follows
  using the product formula.
\end{proof}

We lift the height function $H$ to $X_0$ by composing it with $\pi\colon
X_0(\QQ) \to X(\QQ)$, giving
\begin{equation*}
  H_0 : X_0(\QQ) \to \Rd_{>0}.
\end{equation*}

\begin{lemma}\label{lem:height_torsor}
  For $P_0 \in X_0(\QQ)$, we have
  \begin{equation*}
    H_0(P_0) = \left(\prod_v \max_{F \in \Pm^l} |F(P_0)|_v\right)^{1/l}.
  \end{equation*}
\end{lemma}

\begin{proof}
  Let $P = \pi(P_0) \in X(\QQ)$. As in the proof of \cite[Lemma~3.5]{BBDG},
  for $F_0$ of degree $L$ not vanishing in $P$ and $F \in \Pm^l$, we have
  $(F/F_0)(P) = F(P_0)/F_0(P_0)$ if we compute $(F/F_0)(P)$ as in
  Remark~\ref{rem:rational_functions} and also regard $F,F_0$ as regular
  functions on $X_0(\QQ)$ that can be evaluated in $P_0$.
  We apply this to Lemma~\ref{lem:height} to obtain
  \begin{equation*}
    H_0(P_0)=H(P) = \left(\prod_v \max_{F \in \Pm^l} \left|\frac{F}{F_0}(P)\right|_v\right)^{1/l}
    = \left(\prod_v \max_{F \in \Pm^l} \left|\frac{F(P_0)}{F_0(P_0)}\right|\right)^{1/l}.
  \end{equation*}
  Then we use the product formula.
\end{proof}

In its integral model, this simplifies as follows.

\begin{cor}\label{cor:height_torsor_integral}
  For $P_0 \in \tX_0(\Zd)$, we have
  \begin{equation*}
    H_0(P_0) = \max_{F \in \Pm^l} |F(P_0)|_\infty^{1/l}.
  \end{equation*}
\end{cor}

\begin{proof}
  This is analogous to \cite[Proposition~11.3]{MR1679841} and
  \cite[Corollary~3.6]{BBDG}. For a prime $p$, we have $P_0 \pmod p$ in
  $\tX_0(\Fd_p)$. There is a $\sigma \in \Sigmamax'$ such that
  $y_\rho(P_0 \pmod p) \ne 0 \in \Fd_p$ for all $\rho \notin \sigma(1)$ since
  $\tX_0$ is defined by the irrelevant ideal in $\tX_1$. Choose $Q \in \Pm^l$
  such that $\Supp \Div Q$ does not meet $U_\sigma$. Then we have
  $Q(P_0 \pmod p) \ne 0 \in \Fd_p$, and hence $|Q(P_0)|_p=1$.  Therefore, we
  have $\max_{F \in \Pm^l} |F(P_0)|_p^{1/l}=1$ and only the archimedean factor
  in Lemma~\ref{lem:height_torsor} remains.
\end{proof}

\subsection{Counting problem}\label{sec:abstract_counting_problems}

The following result parametrizes the set
$N_{X(\QQ) \setminus T,H}(B)$ in terms of integral points on the
universal torsor of the ambient toric variety $Y'$ (which is given by
its Cox ring \eqref{eq:coxamb} and the primitive collections
\eqref{eq:pcoll}), the equation $\Phi'$, and the monomials in
$\Pm^l$. The resulting counting problem is amenable to methods
of analytic number theory.

\begin{prop}\label{prop:countingproblem_abstract}
  Let $X$ be a variety as in Section~\ref{sec:desing}, let $Y'$ be a
  toric variety satisfying \eqref{eq:proper_z}, let $L$ be a divisor
  class as in Section~\ref{subsec:metr}, and let $\Pm^l$ be
  a set of monomials satisfying \eqref{eq:monicbpf}.

  Let $T$ be an arbitrary subset of $X(\QQ)$. Then
  \begin{equation*}
    2^{\rank\Pic Y'} N_{X(\QQ) \setminus T,H}(B) =
     \#\left\{\yv \in \Zd^{\Sigma'(1)} :
      \begin{aligned}
        &\Phi'(\yv)=0,\,  \max_{F \in \Pm^l}|F(\yv)|_\infty^{1/l} \le B,\, \pi'(\yv) \notin T, \\
        &\gcd\{y_\rho : \rho \in S'_j\} = 1 \text{ for every $j=1,\dots,r'$}
      \end{aligned}
    \right\}\text{,}
  \end{equation*}
  using the notation \eqref{eq:coxamb} and \eqref{eq:pcoll}.
\end{prop}

\begin{proof}
  This follows from Proposition~\ref{prop:lift_to_torsor} and Corollary~\ref{cor:height_torsor_integral}.
\end{proof}

\section{Peyre's constant}

We keep the notation and assumptions of Section~\ref{sec:desing}. In
addition, we assume from now on that we are in the situation of
Remark~\ref{rem:smfano1}. In particular, the pullback map
$\Pic Y^\circ \to \Pic X$ is an isomorphism and $X$ is split.

\subsection{Tamagawa measures}

By \cite[(2.2.1)]{MR1340296} and \cite[Theorem~1.10]{MR1679841}, the $v$-adic norm
$\|\cdot\|_v$ on $\omega_X^{-1}$ defined above induces a Tamagawa measure $\mu_v$ on
$X(\QQ_v)$.

\begin{lemma}
  Let $S=\sigma(1)$ for some $\sigma \in \Sigmamax'$. For a Borel subset
  $N_v$ of $X^S(\QQ_v)$, we have
  \begin{equation}\label{eq:local_measure_abstract}
    \mu_v(N_v)
    =\int_{N_v} \frac{|\Res(\varpi^S)|_v}{\max_{F \in \Pm^l} |\tau_F\Res(\varpi^S)^l|_v^{1/l}}
    =\int_{N_v} \frac{|\Res(\varpi^S)|_v}{\max_{F \in \Pm^l} |F/y^{lD(S)+E}|_v^{1/l}},
  \end{equation}
  where $|\Res(\varpi^S)|_v$ is the $v$-adic density on
  $X^S(\QQ_v)$ of the volume form $\Res(\varpi^S)$
  on $X^S$.

  If $N_v$ is contained in a sufficiently small neighborhood of $P$ in
  $X^S(\QQ_v)$ with $\partial \Phi^S/\partial z^S_{\rho_0}(P) \ne 0$, then
  \begin{equation}\label{eq:local_measure_explicit}
    \mu_v(N_v)=\int_{\pi^{S}_{\rho_0}(N_v)} \frac{\bigwedge_{\rho \in S \setminus \{\rho_0\}} \dd z^S_\rho}
    {|\partial \Phi^S/\partial z^S_{\rho_0}(\zv^S)|_v \max_{F
        \in \Pm^l}|F(\zv^S)|_v^{1/l}}
  \end{equation}
  in the affine coordinates
  $\zv^S = (z^S_{\rho})_{\rho \in S}$, where
  $\pi^S_{\rho_0} : U^S(\QQ_v)=\QQ_v^{S} \to
  \QQ_v^{S\setminus\{\rho_0\}}$ is the natural projection and
  $z^S_{\rho_0}$ is expressed in terms of the other coordinates using the
  implicit function for $\Phi^S$.
\end{lemma}

\begin{proof}
  This is analogous to the proof of \cite[Proposition~4.1]{BBDG}.
  However, we work with $\|\tau_{F_0}(P)\|_v$ for $F_0=x^{lD(S)+E}$
  and use $F_0(\zv^S)=1$ in our affine coordinates on $X^S(\QQ_v)$. At
  the end, comparing the definitions of $\varpi^S$ and $\varpi_F$
  shows $\varpi_F/(\varpi^S)^l = y^{lD(S)+E}/F$, hence
  $\tau_F \Res(\varpi^S)^l = \tau_F/(\tau^S)^l = F/y^{lD(S)+E}$, and
  hence the integrals in \eqref{eq:local_measure_abstract} are equal.
\end{proof}

\begin{remark}
  Since we have assumed that every cone in $\Sigma'$ is the face of a
  maximal cone, the open subvarieties $X^S$ for $S = \sigma(1)$ with $\sigma \in \Sigmamax'$
  cover $X$.
\end{remark}

\begin{remark}
  If we are in the special case where $X$ is covered by open subvarieties
  $X^S$ with $S \subset \Sigma(1)$ for $S = \sigma(1)$ with
  $\sigma \in \Sigmamax'$,
  then \eqref{eq:local_measure_explicit} gives the same formula for the
  $v$-adic density as we would have obtained by working with $\Rm(X)$ directly
  (since the additional coordinates $z_\rho^S$ are $1$ for all
  $\rho \in \Sigma'(1) \setminus \Sigma(1)$; up to the description of the
  height function defined via the monomials
  $F \in \Pm^l \subset \Qd[y_1,\dots, y_{J'}]$, in which the additional
  variables $y_{J+1},\dots,y_{J'}$ corresponding to
  $\rho \in \Sigma'(1) \setminus \Sigma(1)$ are set to $z_\rho^S=1$). For
  example, this is the case if $X$ is covered by the open subvarieties $X^S$
  for $S \in \Sigmamax^\circ$.
\end{remark}

\subsection{Measures on the torsor}

Analogously to \cite[(4.4)]{BBDG}, we obtain a $v$-adic measure $m_v$
on $X_0(\Qd_v)$, which is explicitly (for sufficiently small subsets $M_v$)
\begin{equation*}
  m_v(M_v) = \int_{\pi_{\rho_0}(M_v)} \frac{\bigwedge_{\rho \in \Sigma_0'(1) \setminus \{\rho_0\}} \dd y_\rho}
  {|\partial \Phi'/\partial x_{\rho_0}(\yv)|_v \max_{F \in \Pm^l} |F(\yv)|_v^{1/l}}
\end{equation*}
in the coordinates $\yv=(y_\rho)_{\rho \in \Sigma_0'(1)}$, where $\pi_{\rho_0}$ is the
projection to all coordinates $y_\rho$ with $\rho\ne\rho_0$ and where
$y_{\rho_0}$ is expressed in terms of these coordinates using the
implicit function theorem.

\begin{lemma}\label{lem:measure_variety_torsor_p-adic}
  For any prime $p$, we have $m_p(\tX_0(\Zd_p)) = (1-p^{-1})^{\rank \Pic Y'} \mu_p(X(\Qd_p))$.
\end{lemma}

\begin{proof}
  This is analogous to \cite[Lemma~4.2, Lemma~4.3]{BBDG}.
\end{proof}

\subsection{Comparison to the number of points modulo $p^\ell$}

As in \cite[\S 4.4]{BBDG}, for any prime $p$ and $l \in \Zd_{>0}$, we
have
\begin{equation*}
  \tX_0(\Zd/p^{\ell}\Zd) = \{\xv \in (\Zd/p^\ell\Zd)^{\Sigma'(1)} : \Phi'(\xv) = 0 \in \Zd/p^\ell\Zd,\ p \nmid \gcd\{x_\rho: \rho \in S'_j\}\text{ for all }j = 1,\dots,r'\}.
\end{equation*}
In particular, the additional variables $x_\rho$ indexed by $\rho \in
\Sigma'(1)\setminus\Sigma(1)$ (obtained via the desingularization of the
ambient toric variety) appear here.

\begin{prop}\label{prop:measure_torsor_mod_p^l}
  For every prime $p$, there is an $\ell_0 \in \Zd_{>0}$ such that for all
  $\ell \ge \ell_0$ we have
  \begin{equation*}
    m_p(\tX_0(\Zd_p)) = \frac{\# \tX_0(\Zd/p^{\ell}\Zd)}{(p^\ell)^{\Sigma'(1)-1}}.
  \end{equation*}
  Furthermore,
  \begin{equation*}
    (1-p^{-1})^{\rank \Pic X} \mu_p(X(\Qd_p)) = (1-p^{-1})^{\#\Sigma(1)-\#\Sigma'(1)} \lim_{\ell \to \infty} \frac{\# \tX_0(\Zd/p^{\ell}\Zd)}{(p^\ell)^{\dim X_0}}.
  \end{equation*}
\end{prop}

\begin{proof}
  This relies on the regularity of $Y'$ and is otherwise analogous to
  \cite[Lemma~4.4, Proposition~4.5, Proposition~4.6]{BBDG}. The
  referee kindly pointed out that an argument is also contained in the
  work of Peyre \cite{Pey2,Pey3,Pey4}.
\end{proof}

\subsection{The real density}

In this section, we compute the real density and Peyre's
$\alpha$-constant as in \cite[\S 4.5]{BBDG}. Those results can be
applied with minor modifications, which we now discuss.

We choose $\sigma \in \Sigmamax'$, $\rho_0 \in \sigma(1)$, and
$\rho_1 \in \Sigma(1) \setminus \sigma(1)$ subject to the following
conditions analogous to \cite[(4.7)]{BBDG}:
\begin{equation}\label{eq:assumption_real_density}
  \begin{aligned}
    &\text{$\sigma \in \Sigmamax'$ also appears in $\Sigmamax$.}\\
    &\text{Every variable $x_\rho$ for $\rho \in
      \Sigma(1)$ appears in at most one monomial of $\Phi$.}\\
    &\text{Writing  $-K_X = \sum_{\rho \in \Sigma(1) \setminus \sigma(1)} \alpha^\sigma_\rho
      \deg(x_\rho)$ in $\Pic X$, we have $\alpha^\sigma_{\rho_1} \ne 0$.}\\
    &\text{The variable $x_{\rho_0}$ appears with exponent $1$ in $\Phi$.}\\
    &\text{No $x_\rho$ with $\rho \in \sigma(1) \cup \{\rho_1\} \setminus
    \{\rho_0\}$ appears in the same monomial of $\Phi$ as $x_{\rho_0}$.}
  \end{aligned}
\end{equation}

We define the numbers $b_{\rho,\rho'}$ and $b_{\rho'}$ as in \cite[\S
4.5]{BBDG}, computed in $\Pic X$, for
$\rho' \in \sigma(1)' = \sigma(1) \cup \{\rho_1\}$ and
$\rho \in \Sigma(1) \setminus \sigma(1)'$. Then we define $c^*$ as in
\cite[(4.9)]{BBDG}.

We now define $c_{\infty}$ as in \cite[(4.11)]{BBDG}. We work without
the additional coordinates indexed by
$\Sigma'(1) \setminus \Sigma(1)$. Therefore, we can use the results
from \cite[(4.11)]{BBDG}, considering $X$ to be embedded into the
possibly singular toric variety $Y$. Since we are working with a more
general height function, we change the definition of $H_\infty$ to
  \[
    H_\infty(\mathbf{x}) = \max_{F \in \Pm^l}|F(\mathbf{x})|^{1/l},
  \]
where the additional coordinates indexed by
$\Sigma'(1) \setminus \Sigma(1)$ are set to $1$ on the right-hand
side.

In order to proceed as in \cite[(4.11)]{BBDG}, it remains to show
that the monomials in $\Pm^l$ with the additional coordinates set to
$1$ are homogeneous of degree $-lK_X \in \Pic(X) = \Cl(Y')$.

\begin{lemma}
  \label{lemma:pic-alpha-b}
  If we write
  \[l^{-1}\cdot L = \sum_{\rho \in \Sigma'(1) \setminus \sigma(1)} (\alpha')^\sigma_\rho
    \deg(y_\rho)\] in $\Pic Y'$, then we have
  $(\alpha')^\sigma_\rho = \alpha^\sigma_\rho$ for all
  $\rho \in \Sigma(1)$. In other words, it does not matter whether we
  compute the numbers $\alpha^\sigma_\rho$ in $\Pic X$ or $\Pic Y'$.
  The same is true for the numbers $b_{\rho,\rho'}$ and $b_{\rho'}$.
\end{lemma}
\begin{proof}
  For every $\rho\in \Sigma'(1)$, we have
  $D_\rho = \deg y_{\rho} \in \Pic Y'$ for some prime divisor $D_\rho$
  in $Y'$ with $X \nsubseteq D_\rho$. We can therefore directly
  compute the pullbacks to $X$ of all $D_\rho$.

  For every $\rho \in \Sigma(1)$, the pullback of $\deg(y_\rho)$ is
  $\deg(x_\rho) \in \Pic X$. On the other hand, for every
  $\rho \in \Sigma'(1) \setminus \Sigma(1)$ we have
  $D_\rho|_X = \emptyset$; therefore the pullback to $X$ is
  $0 \in \Pic X$.

  Finally, the pullback of $l^{-1}\cdot L$ is $-K_X$. By pulling back
  the defining equations for $(\alpha')^\sigma$, the result follows.
  The argument for $b_{\rho,\rho'}$ and $b_{\rho'}$ is the same.
\end{proof}

\begin{lemma}
  \label{lemma:monomialsPicX}
  Let $F \in \Pm^l$, and let $F_Y$ be the corresponding monomial where
  the additional coordinates indexed by
  $\Sigma'(1) \setminus \Sigma(1)$ are set to $1$. Then $F_Y$ is
  homogeneous of degree $-lK_X \in \Pic(X) = \Cl(Y')$.
\end{lemma}
\begin{proof}
  We write
  \[
    F = \prod_{\rho \in \Sigma'(1)} y_\rho^{k_\rho}
  \]
  and, since the monomial $F$ is of degree $L$, we obtain
  \[
    L = \sum_{\rho \in \Sigma'(1)} k_\rho\deg(y_\rho).
  \]
  The proof now proceeds exactly as in the proof of
  Lemma~\ref{lemma:pic-alpha-b}.
\end{proof}

\begin{prop}\label{prop:real_density}
  In the notation above and assuming \eqref{eq:assumption_real_density}, we have
  \begin{equation*}
    \alpha(X)\mu_\infty(X(\Rd)) = \frac{1}{2^{\rank\Pic X}} c^*c_\infty.
  \end{equation*}
\end{prop}

\begin{proof}
  First, we note that \cite[Lemma~4.8]{BBDG} is still valid. Moreover,
  we observe that the additional variables $z_\rho^S$ for
  $\rho \in \Sigma'(1)\setminus\Sigma(1)$ are $1$ in the expression
  \eqref{eq:local_measure_explicit} for $\mu_\infty(X(\Rd))$ for
  $S = \sigma(1)$. Therefore, the expected real density
  $\omega_\infty = \mu_\infty(X(\Rd))$ has the same description as in
  \cite[(4.10)]{BBDG}.
  Taking into account Lemma~\ref{lemma:monomialsPicX}, the statement
  and proof of \cite[Proposition~4.10]{BBDG} stay the same.
\end{proof}

\part{Analysis of a diophantine equation}

The counting problem in Corollary~\ref{cor:counting_problems}(a) corresponding to the variety $X_1$ is rather delicate and not covered by the general method of \cite{BBDG}. This part of the paper is devoted to a detailed investigation. 

\section{Elementary bounds}\label{sec-elem}

For $\xi \in \mathbb{Z}\setminus \{0\}$ we consider the equation
\begin{equation}\label{eq0}
 x_{11} x_{12} + x_{21}^2 + \xi^2x_{31}x_{32}x_{33}^2 = 0, \quad x_{31}x_{32} \not= -\square
 \end{equation}
 where all variables are non-zero integers. For $\textbf{X} = (X_{ij})$ with
 $X_{ij} \geq 1$ let $N_{\xi}(\textbf{X})$ be the number of solutions to
 \eqref{eq0} in boxes $\frac{1}{2}X_{ij} \leq |x_{ij}| \leq X_{ij}$.  In many
 cases it will improve the readability considerably to relabel the variables,
 and it will be convenient to refer to \eqref{eq0} in the form
 \begin{equation*}
 ab+ c^2 + \xi^2ywz^2 = 0, \quad yw \not= \square. 
\end{equation*}
 Consequently, we will write $\textbf{X} = (A, B, C, Y, W, Z)$.   We generally write
 $$\| \textbf{X} \| = \max(A, B, C, Y, W, Z)$$
 Using this notation, we start with some elementary bounds. 
 \begin{prop}\label{elem}
 We have
 \begin{displaymath}
   \begin{split}
     N_{\xi}(\emph{\textbf{X}}) &\ll  \| \emph{\textbf{X}} \|^{\varepsilon}
     (ABC)^{1/2}(YW)^{3/4} Z^{1/2}.  
   \end{split}
 \end{displaymath}
\end{prop}
   
\begin{proof} The bound 
\begin{equation}\label{div1}
   N_{\xi}(\textbf{X}) \ll \|\textbf{X} \|^{\varepsilon}\min(  ABC, CYWZ)
   \end{equation}
    follows from a simple divisor argument. Alternatively, we can fix $y, w, z$ so that
    $$c^2 = -ywz^2\xi^2  + O(AB),$$
which then determines $a, b$ up to a divisor function. The   equation for $c$ defines an interval of length $$\ll \frac{AB}{\sqrt{YW} Z|\xi|}$$
 so that
\begin{equation}\label{div2}
N_{\xi}(\textbf{X}) \ll  \|\textbf{X} \|^{\varepsilon}YWZ \Big(\frac{AB}{Z\sqrt{YW}|\xi|} + 1\Big) =  \|\textbf{X} \|^{\varepsilon} \Big( \frac{AB\sqrt{YW}}{|\xi|} + YWZ\Big).
\end{equation}
The claim follows now from  \eqref{div1} and \eqref{div2} after taking suitable geometric means, namely 
$$\min( AB\sqrt{YW}, CYWZ) \leq  (ABC)^{1/2}(YW)^{3/4} Z^{1/2}, \quad \min(ABC, YWZ) \leq \sqrt{ABCYWZ}.$$
\end{proof}

We refine this argument a bit as follows. For $0 < \Delta \leq 1$ and $\textbf{X} = (X_{ij})$ let $N^{\ast}_{\xi}(\textbf{X}, \Delta)$ be the set of solutions to \eqref{eq0} with the same size conditions as $N_{\xi}(\textbf{X})$ except that for one index $$(ij) \in \{(21), (31), (32), (33)\}$$ the condition $\frac{1}{2}X_{ij} \leq |x_{ij}| \leq  X_{ij}$ is replaced with $X_{ij} \leq |x_{ij}| \leq  X_{ij}(1 + \Delta)$.  We also  define $N_{\xi}(\textbf{X}, \Delta)$ to the be the same size conditions as $N_{\xi}(\textbf{X})$ except that  $\frac{1}{2}X_{11} \leq |x_{11}| \leq  X_{11}$ is replaced with $X_{11} \leq |x_{11}| \leq  X_{11}(1 + \Delta)$. The following proposition investigates $N^{\ast}_{\xi}(\textbf{X}, \Delta) $, while $N_{\xi}(\textbf{X}, \Delta)$ comes up in Proposition \ref{prop3}.

\begin{prop}\label{elem1}
We have
 \begin{displaymath}
 \begin{split}
 N^{\ast}_{\xi}(\emph{\textbf{X}}, \Delta) \ll \|\emph{\textbf{X}} \|^{\varepsilon} \Delta^{1/2} (ABC)^{1/2}(YW)^{3/4} Z^{1/2}.
 \end{split}
 \end{displaymath}
 \end{prop}
 
 \begin{proof} We distinguish two cases. If it is not the $c$-variable that is restricted, then by the same argument as above we have
 \begin{displaymath}
 \begin{split}
 N^{\ast}_{\xi}(\textbf{X}, \Delta) &\ll \|\textbf{X} \|^{\varepsilon}\min\Big(ABC, 
 \Delta CYWZ, \Delta(AB\sqrt{YW} + YWZ)\Big)\\
 &\ll \|\textbf{X} \|^{\varepsilon}\big(\Delta (ABC)^{1/2}(YW)^{3/4}Z^{1/2} + \Delta^{1/2}(ABCYWZ)^{1/2}\big) \end{split}
 \end{displaymath}
 and the claim follows. 
 
 If the restricted variable is $c$, then we have similarly
 \begin{displaymath}
 \begin{split}
 N^{\ast}_{\xi}(\textbf{X}, \Delta) &\ll \|\textbf{X} \|^{\varepsilon}\min\Big(\Delta ABC, 
 \Delta CYWZ,  AB\sqrt{YW} + YWZ\Big)\\
 &\ll \|\textbf{X} \|^{\varepsilon}\big(\Delta^{1/2} (ABC)^{1/2}(YW)^{3/4}Z^{1/2} + \Delta^{1/2}(ABCYWZ)^{1/2}\big).  
   \end{split}
 \end{displaymath}
 This completes the proof.  \end{proof}
 
\emph{Remarks:} 1) It will be convenient to introduce the following short-hand notation for expressions like those in Proposition \ref{elem1}.  For ${\bm \zeta} = (\zeta_1, \zeta_2, \zeta_3) \in \mathbb{R}^3$ and $\textbf{X} = (A, B, C, Y, W, Z)$ we define
\begin{equation}\label{exp-not}
\textbf{X}^{(\zeta)} = (AB)^{1 - \zeta_1} C^{1 - 2\zeta_2} (YW)^{1-\zeta_3} Z^{1-2\zeta_3}.
\end{equation}
With this notation, the bounds in Propositions  \ref{elem} and \ref{elem1}  involve $\textbf{X}^{(\frac{1}{2}, \frac{1}{4}, \frac{1}{4})}$. 

2)  In order for $N_{\xi}(\textbf{X})$ to be non-zero, we must have
\begin{equation}\label{trick}
C \ll (AB)^{1/2} + |\xi|Z\sqrt{YW}.
\end{equation}
 
\section{Character sums }

We consider  the following two character sums. For $a, c, z, \xi \in \mathbb{N}$, $h_1, h_2\in \mathbb{Z}$ let 
\begin{equation}\label{defS}
S_{\xi}(h_1, h_2, a, c, z) = \sum_{\substack{y, w\, (\text{mod } a)\\a \mid c^2 + \xi^2 ywz^2}} e\Big( \frac{h_1y + h_2w}{a}\Big).
\end{equation}
For $x, a \in \mathbb{N}$, $k_1, k_2 \in \mathbb{Z}$ let
 \begin{equation}\label{defT}
  T(k_1, k_2, x, a)= \sum_{\gamma^2 + x\xi^2 \equiv 0 \, (\text{mod } a)}e\left(\frac{h_1\gamma + h_2\xi}{a}\right). 
  \end{equation}
  Let $\tau$ denote the divisor function.

\begin{lemma}\label{lem1} We have
$$S_{\xi}(0, 0, a, c, z)  = \sum_{\substack{a_1a_2a_3 = a\\ a_3 \mid c^2}} a_1(\xi^2z^2, a_2a_3)a_3 \mu(a_2)$$
and
$$|S_{\xi}(h_1, h_2, a, c, z)|\leq \tau(a)  (a, h_1, h_2)a^{1/2} (a, c^2)^{1/2} .$$
\end{lemma}

\begin{proof} We have
\begin{displaymath}
\begin{split}
S_{\xi}(h_1, h_2, a, c, z)& = \frac{1}{a} \sum_{\alpha,y, w\, (\text{mod } a) } e\Big( \frac{\alpha(c^2 + yw z^2\xi^2) + h_1y + h_2w}{a}\Big)\\
& = \sum_{\substack{\alpha, w\, (\text{mod }a)\\ \alpha w z^2 \xi^2+ h_1 \equiv 0 \, (\text{mod }a)}} e\Big( \frac{\alpha c^2 + h_2w}{a}\Big)\\
& = \sum_{\substack{a_1a_2 = a\\ a_1 \mid h_1}}\underset{\alpha \, (\text{mod }a_2)} {\left.\sum\right.^{\ast}}  \sum_{\substack{w\, (\text{mod }a)\\   w z^2\xi^2 \equiv - \bar{\alpha} \frac{h_1}{a}  \, (\text{mod }a_2)}}  e\Big( \frac{\alpha a_1 c^2 + h_2w}{a}\Big)\\
& =  \sum_{\substack{a_1a_2 =| a|\\ a_1(z^2\xi^2, a_2) \mid h_1}}\underset{\alpha \, (\text{mod }a_2)} {\left.\sum\right.^{\ast}} \hspace{-0.3cm} \sum_{\substack{w\, (\text{mod }a)\\   w\equiv - \overline{\alpha} \overline{\frac{z^2\xi^2}{(z^2, a_2)} }\frac{h_1}{a_1(z^2\xi^2, a_2)}  \, (\text{mod } \frac{a_2}{(z^2\xi^2, a_2)})}}  \hspace{-1cm}e\Big( \frac{\alpha a_1 c^2 + h_2w}{a}\Big). 
\end{split}
\end{displaymath}
The $w$-sum vanishes unless $\frac{a(z^2\xi^2, a_2)}{a_2} \mid h_2$, so we obtain
\begin{displaymath}
\begin{split}
  &  \sum_{\substack{a_1a_2 = a\\ a_1(z^2\xi^2, a_2) \mid (h_1, h_2)}}\underset{\alpha \, (\text{mod }a_2)} {\left.\sum\right.^{\ast}}   e\Big(\frac{ bc^2}{a_2}\Big)a_1(z^2\xi^2, a_2)e\Big(- \frac{h_2\overline{\alpha} \overline{\frac{z^2\xi^2}{(z^2\xi^2, a_2)} }\frac{h_1}{a_1(z^2\xi^2, a_2)} }{a}\Big)\\
   & =  \sum_{\substack{a_1a_2 = a\\ a_1(z^2\xi^2, a_2) \mid (h_1, h_2)}} a_1(z^2\xi^2, a_2) S\Big(c^2, -\frac{h_2}{a_1}\frac{h_1}{a_1(z^2\xi^2, a_2)}\overline{\frac{z^2\xi^2}{(z^2\xi^2, a_2)} }\ , a_2\Big)
\end{split}
\end{displaymath}
where $S(., ., .)$ is the  Kloosterman sum. If $h_1 = h_2 = 0$, then the claim follows by the formula $$S(c^2, 0, a_2) = \sum_{d \mid (a_2, c^2)}d\mu\Big( \frac{a_2}{d}\Big). $$ 
In general we use Weil's bound $|S(c^2, \ast, a_2)|  \leq \tau(a_2) a_2^{1/2 }(c^2, a_2)^{1/2}$ to complete the proof of the lemma. \end{proof}

A number $D \in \mathbb{Z} \setminus \{0\}$ is a discriminant if $D \equiv 0, 1$ (mod 4). For each discriminant we denote by $\chi_D = (D/.)$ the corresponding quadratic character. It is primitive if and only if $D$ is a fundamental discriminant. If $d\in \mathbb{N}$ is odd we write $d^{\ast}$ for the unique discriminant with $|d^{\ast}| = d$. 
For an odd number $d$ we write $\epsilon_d = \sqrt{\chi_{-4}(d)} \in \{1, i\}$. 

\begin{lemma}\label{lem3} We have
$$T(0, 0, x, a)  = \sum_{ \substack{d_1|d_2| = a\\ (x, d_2) = \square} } d_1 \phi(d_2)  \chi_{\frac{d_2}{(x, d_2)}}\Big(\frac{-x}{(x, d_2)}\Big) (x, d_2)^{1/2}$$
where $d_2$ runs over all discriminants (positive or negative). If $a$ is odd, then
$$|T(k_1, k_2, x, a)| \leq \tau(a)(a, k_1^2 x + k_2^2) (a, x)^{1/2}. $$
\end{lemma}

\begin{proof} We have 
\begin{displaymath}
\begin{split}
T(k_1, k_2, x, a) &= \frac{1}{a} \sum_{\alpha, c, z\, (\text{mod }a)} e\Big( \frac{\alpha(c^2 + xz^2) + k_1c + k_2 z}{a}\Big) \\
&= \frac{1}{a} \sum_{\substack{d_1d_2 = a\\ d_1\mid (k_1, k_2)}} d_1^2\underset{\alpha \,(\text{mod } d_2)}{\left.\sum\right.^{\ast} }  \sum_{c, z\, (\text{mod }d_2)}e\Big( \frac{\alpha(c^2 + xz^2) + \frac{k_1}{d_1}c + \frac{k_2}{d_1} z}{d_2}\Big).
  \end{split}\end{displaymath}
Let us first consider the case $k_1 = k_2 = 0$. We split the modulus $d_2 = u2^{\rho}$ into an odd part and a power of two. The $\alpha, c, z$-sum becomes 
\begin{displaymath}
\begin{split}
&   \underset{\alpha \,(\text{mod } u)}{\left.\sum\right.^{\ast} }  \sum_{c, z\, (\text{mod }u)} e\Big( \frac{\overline{2^{\rho}}\alpha(  c^2 +    x z^2)}{u  }\Big)  \underset{\alpha \,(\text{mod } 2^{\rho})}{\left.\sum\right.^{\ast} }  \sum_{c, z\, (\text{mod }2^{\rho})}  e\Big( \frac{\overline{u}\alpha(  c^2 +    x z^2)}{2^{\rho}  }\Big). \\
 \end{split}
\end{displaymath}
By the well-known evaluation of the Gau{\ss} sum, the first $c, z$-sum equals
$$d_2\epsilon_{u}\epsilon_{\frac{u}{(x, u)}} (x, u)^{1/2} \chi_{u^{\ast}}(\alpha)\chi_{(\frac{u}{(x, u)})^{\ast}}\Big(\frac{x}{(x, u)}\alpha\Big).$$
Summing this over $\alpha$, we see that only the contribution of $(x, u) = \square$ survives, and the first $\alpha, c, z$-sum equals
$$\delta_{(x, u) = \square} u \phi(u) (x, u)^{1/2} \chi_{\frac{u^{\ast}}{(x, u)}}\Big(- \frac{x}{(x, u)}\Big).$$
For the second $\alpha, c, z$-sum modulo powers of 2 we argue similarly, but we have to distinguish a few cases. Recall first that for odd $\alpha$ we have 
$$\sum_{d\, (\text{mod } 2^{\rho})} e\Big( \frac{\alpha d^2}{2^{\rho}}\Big) = \begin{cases} 1, & \rho = 0,\\ 0, & \rho = 1,\\ 2^{\rho/2}(\chi_{2^{\rho}}(\alpha) + i\chi_{-2^{\rho}}(\alpha)), &   \rho \geq 2.
\end{cases}$$
If $4\mid x/(2^{\rho}, x)$ and $4 \mid 2^{\rho}$, we obtain
$$ \underset{\alpha \,(\text{mod } 2^{\rho})}{\left.\sum\right.^{\ast} }  2^{\rho}(2^{\rho}, x)^{1/2} (\chi_{2^{\rho}}(\alpha) + i\chi_{-2^{\rho}}(\alpha))\Big(\chi_{\frac{2^{\rho}}{(x, 2^{\rho})}}\Big(\frac{x}{(x, 2^{\rho})}\alpha\Big) + i\chi_{-\frac{2^{\rho}}{(x, 2^{\rho})}}\Big(\frac{x}{(x, 2^{\rho})}\alpha\Big)\Big).$$
This vanishes, unless $(x, 2^{\rho}) = \square$ in which case it equals
$$ 2^{\rho} \phi(2^{\rho}) (x, 2^{\rho})^{1/2} \sum_{\pm} \chi_{\pm \frac{2^{\rho}}{(x, 2^{\rho})}} \Big( - \frac{x}{(x, 2^{\rho})}\Big).$$
The remaining cases are simple: if $\rho = 1$, the sum vanishes and the case $\rho = 0$ is trivial. 
If $2^{\rho} \mid x$, the $z$-sum is equals $2^{\rho} = 2^{\rho/2}(x, 2^{\rho})^{1/2}$, and evaluating the $c$-sum, we see that the $\alpha$-sum vanishes unless $\rho$ is even, i.e.\ $2^{\rho} = (x, 2^{\rho}) = \square$. If $x/(2^{\alpha}, x) = 2$, the $z$-sum vanishes. In this way we confirm in all cases that the second $\alpha, c, z$-sum equals
$$ \delta_{(x, 2^{\rho}) = \square} 2^{\rho} \phi(2^{\rho}) (x, 2^{\rho})^{1/2} \Big(\delta_{4(x, 2^{\rho}) \mid 2^{\rho}}\sum_{\pm} \chi_{\pm \frac{2^{\rho}}{(x, 2^{\rho})}} \Big( - \frac{x}{(x, 2^{\rho})}\Big) +\delta_{2^{\rho} \mid x} \Big).$$
Combining this with the evaluation for odd moduli, we confirm that $T(0, 0, x, a)$ equals
\begin{displaymath}
\begin{split}
  \frac{1}{a}& \sum_{\substack{d_1d_2 = a\\  (d_2, x) = \square}} d_1^2 d_2\phi(d_2) (x, d_2)^{1/2} \Bigg(\delta_{4(x, d_2) \mid d_2}\sum_{\pm} \chi_{\pm \frac{d_2}{(x, d_2)}} \Big( - \frac{x}{(x, d_2)}\Big) +\delta_{2(x, d_2) \nmid d_2 } \chi_{\frac{d_2^{\ast}}{(x, d_2)}} \Big( - \frac{x}{(x, d_2)}\Big)\Bigg).
\end{split}
\end{displaymath}
This is equivalent to the formula given in the lemma.  

We now turn to the estimation for general $k_1, k_2$ and for simplicity restrict ourselves to odd $a$, as in the statement of the lemma.  In this case we can evaluate the two Gau{\ss} sums in $c, z$ simply by completing the square, and we obtain
$$\frac{1}{a}  \sum_{\substack{d_1d_2 = a\\ d_1\mid k_1\\ d_1(x, d_2) \mid k_2}}  \underset{\alpha \,(\text{mod } d_2)}{\left.\sum\right.^{\ast} }  d_1^2d_2   \epsilon_{d_2}\chi_{(x, d_2)^{\ast}}(\alpha) \chi_{(\frac{d_2}{(x, d_2)})^{\ast}}\Big(\frac{x}{(x, d_2)}\Big)  e\Big(\frac{-\overline{4\alpha}(\frac{k_1^2 x + k_2^2}{d_1^2 (x, d_2)})}{d_2}\Big) (x, d_2)^{1/2} \epsilon_{\frac{d_2}{(x, d_2)}}. $$
 Let $\delta$ denote the conductor of $\chi_{(x, d_2)^{\ast}}$ and write $d_2 = \delta\delta_1\delta_2$ with $(\delta_2, \delta) = 1$, $\delta_1 \mid \delta^{\infty}$. Then by the well-known evaluation of quadratic character sums, the $\alpha$-sum is bounded by
$$\delta_1\delta^{1/2}\Big(\delta_2, \frac{k_1^2x+ k_2^2}{d_1^2 (x, d_2)}\Big).$$
We decompose uniquely $(x, d_2)= t_1t_2$ such that $t_1$ is the largest square  coprime to $t_2$. Then $\delta = \text{rad}(t_2)$, $\delta_1 = t_2/\text{rad}(t_2)$, $\delta_2 = t_1$, so that we obtain the upper bound
\begin{displaymath}
\begin{split}
 & \sum_{\substack{d_1d_2 = a\\ d_1\mid k_1\\ d_1(x, d_2) \mid k_2 }}   d_1     (x, d_2)^{1/2} \frac{t_2}{\text{rad}(t_2)^{1/2}}\Big(t_1, \frac{k_1^2x+ k_2^2}{d_1^2 (x, d_2)}\Big) \\
 & = \sum_{\substack{d_1d_2 = a\\ d_1\mid k_1\\ d_1(x, d_2) \mid k_2 }}      \frac{ (x, d_2)^{1/2} }{\text{rad}(t_2)^{1/2}}\Big(t_1d_1t_2, \frac{k_1^2x+ k_2^2}{d_1 t_1}\Big) \leq \sum_{d_1d_2 = a} (x, d_2)^{1/2} (a, k_1^2x + k_2^2), \\
  \end{split}
  \end{displaymath}
and the claim follows. \end{proof}

We also recall the following standard estimate.
\begin{lemma}\label{lem5} Let $V$ be a smooth function with compact support in $[-2, -1/3] \cup[1/3, 2]$ such that $V^{(j)} \ll_j \Omega^j$ for some $\Omega \geq 1$.  If $\Delta \not= \square$ is a discriminant, $g \in \mathbb{N}$ and $N \geq 1$, then
$$\sum_{(n, g) = 1} \chi_{\Delta}(n) V\Big( \frac{n}{N}\Big) \ll \tau(g) N^{1/2} (\Omega |\Delta|)^{1/4+\varepsilon}. $$
In addition, if $n \in \mathbb{Z} \setminus \{0\}$ is not a square and $D \geq 1$, then
$$\sum_{(\Delta, g) = 1} \chi_{\Delta}(n) V\Big( \frac{\Delta}{D}\Big) \ll \tau(g) D^{1/2} (\Omega |n|)^{1/4+\varepsilon} $$
where the sum runs over all discriminants. The implied constants depend only on $\varepsilon$. 
\end{lemma}

\begin{proof} This is standard by Mellin inversion and the convexity bound for Dirichlet $L$-functions (with Euler factors at primes dividing $g$ removed) $$L^{(g)}(s, \chi_{\Delta}) \ll \tau(g) (|\Delta| (1 + |s|))^{1/4 + \varepsilon}, \quad \Re s = 1/2.$$ Suffice it to say that the (two-sided) Mellin transform of $V$ is entire and satisfies $$\widehat{V}(s) \ll_{\Re s, A} \Big(1 + \frac{ |\Im s|}{\Omega}\Big)^{-A}$$
for all  $A > 0$. The second bound follows from quadratic reciprocity as follows: By quadratic reciprocity we have $\chi_{\Delta}(n) = \chi_{\tilde{n}}(\Delta)$ for $n > 0$  where $\tilde{n}$ is the discriminant computed as follows: write $n = 2^a y$ with $y$ odd, and let $y^{\ast}$ denote the discriminant satisfying $|y^{\ast}| = y$. Let $a' = a+2$ if $a=2$ and $a' = a$ otherwise. Then $\tilde{n} = 2^{a'} y^{\ast}$. If $n < 0$ then $\chi_{\Delta}(n) = \chi_{\tilde{n}}(\Delta) \chi_{\Delta}(-1)$. In this way, the second bound follows from the first by detecting the condition $\Delta \equiv 0, 1$ (mod 4) by characters. \end{proof}

\emph{Remark:} Using the strongest available uniform subconvexity bounds, the exponents $1/4$ can be replaced with $1/6$.

\section{Upper bound estimates}

For $\textbf{X} = (A, B, C, Y, W, Z)$ and $\xi \in \mathbb{Z} \setminus \{0\}$ we recall the definition of $N_{\xi}(\textbf{X})$ and $N_{\xi}(\textbf{X}, \Delta)$ from Section \ref{sec-elem}, and we now define a smooth version of the former.   Let $\Omega \geq 3$ be a parameter. We choose an even,  smooth, non-negative  test function $V$ with support in $[-1-\Omega^{-1}, -1/2 + \Omega^{-1}] \cup [1/2-\Omega^{-1}, 1+\Omega]$ and $V(x) = 1$ on $[-1, -1/2] \cup [1/2, 1]$ satisfying $V^{(j)}(x) \ll_j \Omega^j$ for  all $j \in \mathbb{N}_0$. We define 

\begin{equation*}
\tilde{N}_{\xi}(\textbf{X}, \Omega) = \sum_{\substack{ab + c^2 + \xi^2 ywz^2 = 0\\ yw \not= -\square}} V\Big(\frac{a}{A}\Big) V\Big(\frac{b}{B} \Big) V\Big(\frac{c}{C}\Big) V\Big(\frac{y}{Y}\Big) V\Big(\frac{w}{W}\Big)V\Big(\frac{z}{Z}\Big). 
\end{equation*}
(Strictly speaking, $\tilde{N}_{\xi}(\textbf{X}, \Omega)$ depends on $V$ and not only on $\Omega$, but this small abuse of notation is convenient.)  Fix a small constant  
$\lambda \leq 10^{-6}$ and consider  $\textbf{X}, \xi$ satisfying
\begin{equation}\label{good}
\begin{split}
&\min(A, B, C) \geq \max(A, B, C)^{1-\lambda}, \\
& \min(Y, W) \geq \max(Y, W) \|\textbf{X}\|^{-\lambda}, \\
&\min(AB, C^2, YWZ^2) \geq \max(AB, C^2, YWZ^2)^{1-\lambda}, \\
& |\xi| \leq \| \textbf{X}\|^{\lambda}. 
\end{split}
  \end{equation}
We will see later that these are the critical size conditions. \\

 Our aim in this section is to establish the following three estimates that we will prove simultaneously. Recall the notation \eqref{exp-not}. 
  
\begin{prop}\label{prop2} For $\emph{\textbf{X}}, \xi$ satisfying \eqref{good} we have
$$N_{\xi}(\emph{\textbf{X}}) \ll  \emph{\textbf{X}}^{(\frac{3}{4}, 0, \frac{1}{4})} +  \emph{\textbf{X}}^{(\frac{1}{2}, \frac{1}{4}, \frac{1}{4})}.$$
\end{prop}

\begin{prop}\label{prop3} For $\emph{\textbf{X}}, \xi$ satisfying \eqref{good} and $0 < \Delta < 1$ we have
$$N_{\xi}(\emph{\textbf{X}}, \Delta) \ll  \big(\emph{\textbf{X}}^{(\frac{3}{4}, \frac{1}{8}, \frac{1}{8})} + \emph{\textbf{X}}^{(\frac{1}{2}, \frac{1}{4}, \frac{1}{4})}\big) 
\big(\Delta \|\emph{\textbf{X}} \|^{\varepsilon} + \|\emph{\textbf{X}} \|^{-1/5}\big) .  $$
\end{prop}

\begin{prop}\label{prop4} For $\emph{\textbf{X}}, \xi$ satisfying \eqref{good} and $\Omega \geq 3$ we have
$$\tilde{N}_{\xi}(\emph{\textbf{X}}, \Omega) = M_2 + O\Big( \big(\emph{\textbf{X}}^{(\frac{5}{8}, \frac{1}{8}, \frac{1}{4})} + \emph{\textbf{X}}^{(\frac{3}{4}, \frac{1}{8}, \frac{1}{8})}\big)
\Omega \| \emph{\textbf{X}}\|^{11\lambda}  \frac{ A^{1/2}}{Y}   \Big)$$
where $M_2$ is given by \eqref{diag} below. We also have 
$$\tilde{N}_{\xi}(\emph{\textbf{X}}, \Omega) =\tilde{M}_3 + O\Big( \emph{\textbf{X}}^{(\frac{1}{2}, \frac{1}{4}, \frac{1}{4})}
 \Omega^4  
\Big(  \frac{\|\emph{\textbf{X}}\|^{10\lambda}}{Z} + \frac{1}{(YW)^{1/21}}\Big)\Big)    $$
where $\tilde{M}_3$ is defined in \eqref{main} below. 
\end{prop}

The rest of this section is devoted to the proof. We will occasionally use the following notation: for a positive integer $n$  let $\sqrt{n}^+$ denote the smallest integer whose square is a multiple of $n$. 

We have trivially $N_{\xi}(\textbf{X})  \leq \tilde{N}_{\xi}(\textbf{X}, \Omega)$  for every $\Omega \geq 3$ and 
\begin{equation}\label{start}
\tilde{N}_{\xi}(\textbf{X}, \Omega) = \sum_{\substack{a, c, y, w, z\\ a \mid c^2 + \xi^2 ywz^2\\ yw \not= -\square}} V\Big(\frac{a}{A}\Big) V\Big(\frac{-\xi^2ywz^2 - c^2}{aB} \Big) V\Big(\frac{c}{C}\Big) V\Big(\frac{y}{Y}\Big) V\Big(\frac{w}{W}\Big)V\Big(\frac{z}{Z}\Big). 
\end{equation}
We apply two different strategies. We first apply Poisson summation in $y, w$ and then Poisson summation in $c, z$. In order to obtain Proposition \ref{prop3}, we will occasionally replace $V(a/A)$ with the characteristic function on $A \leq |a| \leq A(1+\Delta)$. In the interest of a reasonably compact presentation, we will not introduce extra notation for this.

\subsection{Poisson summation in $y, w$}

We first add the contribution of $wy = -\square$ to \eqref{start}. As in the proof of Proposition \ref{elem}, we see by a divisor argument that this infers an error of at most $O(C(YW)^{1/2}Z \| \textbf{X}\|^{\varepsilon})$ with an implied constant depending only on $\varepsilon$. Then we apply Poisson summation $y, w$ to obtain
\begin{equation}\label{poi1}
\begin{split}
\tilde{N}_{\xi}(\textbf{X}, \Omega)  =  &\sum_{a, c, z} V\Big(\frac{a}{A}\Big) V\Big(\frac{c}{C}\Big) V\Big(\frac{z}{Z}\Big) \frac{1}{a^2} \sum_{h_1, h_2} S_{\xi}(h_1, h_2, a, c, z)\Phi_{\xi}(h_1, h_2, a, c, z) \\
&+ O\big(C(YW)^{1/2}Z \| \textbf{X}\|^{\varepsilon}\big)
\end{split}
\end{equation}
with $S_{\xi}$ as in \eqref{defS} and 
$$\Phi_{\xi}(h_1, h_2, a, c, z)= \int_{\mathbb{R}}\int_{\mathbb{R}} V\Big(\frac{y}{Y}\Big) V\Big(\frac{w}{W}\Big)V\Big(\frac{-\xi^2ywz^2 - c^2}{aB} \Big) e\Big(- \frac{h_1y + h_2w}{|a|}\Big) dy\, dw.$$
Let
\begin{equation*}
H_1 = \Omega\Big( \frac{A}{Y} + \frac{\xi^2WZ^2}{B}\Big), \quad H_2 =\Omega\Big( \frac{A}{W} + \frac{\xi^2YZ^2}{B}\Big).
\end{equation*}

\begin{lemma}\label{lem2} For $|a|\asymp A$, $|z| \asymp Z$ we have 
$$\Phi_{\xi}(h_1, h_2, a, c, z)  \ll_N 
\frac{(AB)^{1/4}(YW)^{3/4}}{(Z|\xi|)^{1/2}}
\Big(1 + \frac{|h_1|}{H_1} + \frac{|h_2|}{H_2}\Big)^{-N}$$
for arbitrary $N > 0$ and 
$$(a \partial_a)^{j_1} (c\partial_c)^{j_2} (z\partial_z)^{j_3}   \Phi_{\xi}(0, 0, a, c, z)  \ll_{\textbf{j}}
\frac{(AB)^{1/4}(YW)^{3/4}}{(Z|\xi|)^{1/2}} \Big( 1+\frac{C^2}{AB}\Big)^{j_2} \Big( 1+\frac{\xi^2YWZ^2}{AB}\Big)^{j_1} \Omega^{j_1+j_2+j_3}$$
for arbitrary $\textbf{j} \in \mathbb{N}_0^3$. 
\end{lemma}

\begin{proof}
We observe that the  volume of the $(y, w)$-region defined by $$|y| \asymp Y, \quad |w| \asymp W, \quad |c^2 + \xi^2ywz^2| \asymp AB$$ is trivially bounded by $O(YW)$, but also by $O(AB/z^2\xi^2)$ and so by $O((AB)^{1/4}(YW)^{3/4}/|z\xi|^{1/2})$.  The first claim  follows now  by repeated  integration by parts, the second by differentiating under the integral sign.   There is an important subtlety: we combine each application of the operator $z\partial_z$ with a partial integration in $y$, i.e.\
\begin{displaymath}
\begin{split}
& \int_{\mathbb{R}} z\partial_z \Big[ V\Big(\frac{y}{Y}\Big) V\Big(\frac{-\xi^2ywz^2 - c^2}{aB} \Big)\Big] dy
 = - 2 \int_{\mathbb{R}}\partial_y  \Big(y V\Big(\frac{y}{Y}\Big) \Big)V\Big(\frac{-\xi^2ywz^2 - c^2}{aB} \Big)   dy. 
\end{split}
\end{displaymath}
This completes the proof. \end{proof}

We now write  the right hand side of \eqref{poi1} as $M_1+M_2$ 
where $M_1$ is the off-diagonal  contribution $(h_1, h_2) \not= (0, 0)$ and $M_2$ is the diagonal contribution $h_1 = h_2 = 0$. By Lemma \ref{lem1} we have 
\begin{equation}\label{diag}
M_2 = \sum_{a, c, z} V\Big(\frac{a}{A}\Big) V\Big(\frac{c}{C}\Big) V\Big(\frac{z}{Z}\Big) \frac{1}{a^2} \sum_{\substack{a_1a_2a_3 = |a|\\ a_3 \mid c^2}} a_1(\xi^2z^2, a_2a_3)a_3 \mu(a_2) \Phi_{\xi}(0, 0, a, c, z). 
\end{equation}
We postpone the analysis of $M_2$ and investigate first $M_1$. By Lemma \ref{lem1}  and Lemma \ref{lem2} we have 
\begin{equation*}
\begin{split}
M_1 &\ll \sum_{a, c, z} V\Big(\frac{a}{A}\Big) V\Big(\frac{c}{C}\Big) V\Big(\frac{z}{Z}\Big) \frac{(AB)^{1/4}(YW)^{3/4}}{A^2(Z|\xi|)^{1/2}}\\
& \qquad\qquad \times\sum_{(h_1, h_2)\not= (0, 0)}(a, h_1, h_2) \tau(a) |a|^{1/2}(a, c) \Big(1 + \frac{|h_1|}{H_1} + \frac{|h_2|}{H_2}\Big)^{-10}\\
& \ll \sum_{a } V\Big(\frac{a}{A}\Big)   \frac{CZ(AB)^{1/4}(YW)^{3/4}}{A^{3/2}(Z|\xi|)^{1/2}}
\sum_{(h_1, h_2)\not= (0, 0)}(a, h_1, h_2) \tau(a)^2   \Big(1 + \frac{|h_1|}{H_1} + \frac{|h_2|}{H_2}\Big)^{-10}\\
&  \ll \sum_{a } V\Big(\frac{a}{A}\Big)  \frac{CZ(AB)^{1/4}(YW)^{3/4}}{A^{3/2}(Z|\xi|)^{1/2}}
\tau(a)^3(1 + H_1 + H_2) .
 \end{split}
\end{equation*}
By several applications of \eqref{good}, we see that
$$\frac{1+ H_1 + H_2}{A^{1/2 - \varepsilon}} \ll   \Omega \| \textbf{X}\|^{10\lambda} \frac{A^{1/2 }}{Y},$$
so that 
\begin{equation}\label{off}
\begin{split}
M_1 &  \ll  \frac{C(AB)^{1/4}(YW)^{3/4} Z^{1/2}}{|\xi|^{1/2}}  \Omega  \|\textbf{X}\|^{10\lambda}  \frac{A^{1/2 }}{Y} .
 \end{split}
\end{equation}
Note that up until now we have not used any property of the weight $V(a/A)$ except that it restricts $a\ll A$. In particular, \eqref{off} continues to hold with any $\Omega$ if $V(a/A)$ is replaced by the characteristic function on $A \leq |a| \leq A(1+\Delta)$.

We can now complete the proof of the first half of Proposition \ref{prop4} by recalling \eqref{poi1} and noting that \eqref{good} implies
\begin{displaymath}
\begin{split}
&CZ\sqrt{YW}(ABYW)^{\varepsilon} \ll (AB)^{1/4}C (YW)^{3/4} Z^{1/2}  \cdot \frac{\|\textbf{X}\|^{10\lambda}}{Y},
\end{split}
\end{displaymath}
so that 
$$N_{\xi}(\textbf{X}, \Omega) = M_2 + O\Big((AB)^{1/4}C (YW)^{3/4} Z^{1/2}  \cdot \Omega \frac{\|\textbf{X}\|^{10\lambda}}{Y}\Big).$$
As in the proof of Proposition \ref{elem1} we reduce the power of $C$ using \eqref{trick}, which completes the proof of the first half of Proposition \ref{prop4}. 

We now turn to $M_2$ as given in \eqref{diag}. We will evaluate this asymptotically later, but for now content ourselves with an upper bound given by 
\begin{equation}\label{m2start}
\begin{split}
M_2&\ll  \frac{C}{A^2} \cdot \frac{(AB)^{1/4}(YW)^{3/4}}{(Z|\xi|)^{1/2}} \sum_{a_1, a_2, a_3, z} V\Big( \frac{a_1a_2a_3}{A}\Big) V\Big(\frac{z}{Z}\Big) a_1a_3^{1/2} (z^2\xi^2, a_2a_3).
\end{split}
\end{equation}
If we replace $V(a/A)$ with the characteristic function on $A\leq |a| \leq A(1+\Delta)$, a very soft bound is given by 
\begin{displaymath}
\begin{split}
&\ll  \frac{C}{A^2} \cdot \frac{(AB)^{1/4}(YW)^{3/4}}{(Z|\xi|)^{1/2}} \sum_{z} \sum_{A\leq a_1a_2a_3 \leq A(1+\Delta)}  V\Big(\frac{z}{Z}\Big) a_1a_2a_3  (z\xi,  a_3)\\
& \ll \frac{C}{A^2} \cdot \frac{(AB)^{1/4}(YW)^{3/4}}{(Z|\xi|)^{1/2}} \Delta A^{2+\varepsilon} Z \tau(\xi) \ll \Delta (AB)^{1/4} C(YW)^{3/4}Z^{1/2} A^{\varepsilon}.
\end{split}
\end{displaymath}
Together with \eqref{off} for $\Omega = 3$ we obtain
\begin{equation}\label{delta1}
N_{\xi}(\textbf{X}, \Delta) \ll   (AB)^{1/4} C(YW)^{3/4}Z^{1/2} \Big(\Delta \|\textbf{X}\|^{\varepsilon} + \|\textbf{X}\|^{10\lambda} \frac{A^{1/2 }}{Y}\Big).
\end{equation}
After this interlude we now return to \eqref{m2start} and estimate the right hand side by
\begin{displaymath}
\begin{split}
&\leq \frac{C}{A^2} \frac{(AB)^{1/4}(YW)^{3/4}}{(Z|\xi|)^{1/2}} \sum_d \sum_{\substack{a_1, a_2, a_3\\ d \mid a_2a_3}}\sum_{\frac{d}{(d, \xi^2)}\mid z^2} V\Big( \frac{a_1a_2a_3}{A}\Big) V\Big(\frac{z}{Z}\Big) a_1a_3^{1/2} d.
\end{split}
\end{displaymath}
For notational simplicity let us write $d_{\xi} = d/(d, \xi^2)$.
Then we can continue to estimate
\begin{displaymath}
\begin{split}
M_2& \ll \frac{C}{A^2} \frac{(AB)^{1/4}(YW)^{3/4}}{(Z|\xi|)^{1/2}}  \sum_d \sum_{d_2 d_3 = d} \sum_{\substack{a_1, a_2, a_3\\ d_2 \mid a_2, d_3\mid a_3}} V\Big( \frac{a_1a_2a_3}{A}\Big) \frac{Z}{\sqrt{d_{\xi}}^+}a_1a_3^{1/2} d\\
&\ll \frac{C}{A^2}  \frac{(AB)^{1/4}(YW)^{3/4}}{(Z|\xi|)^{1/2}} \sum_d d\sum_{d_2 d_3 = d}   \frac{A^2}{d_2^2d_3^{3/2}} \frac{Z}{\sqrt{d_{\xi}}^+} .
\end{split}
\end{displaymath}
The $d$-sum is 
\begin{equation*}
\leq \sum_d \frac{\tau(d)}{d^{1/2} \sqrt{d_{\xi}}^+} \leq \sum_{\eta \mid \xi^2} \frac{\tau(\eta)}{\eta^{1/2}} \sum_d \frac{\tau(d)}{d^{1/2} \sqrt{d}^+} \ll \tau(\xi^2). 
\end{equation*}
Combining this with the off-diagonal contribution \eqref{off} and choosing $\Omega = 3$, we have shown
our first important bound
\begin{equation}\label{first}
N_{\xi}(\textbf{X}) \leq \tilde{N}_{\xi}(\textbf{X}, 3)  \ll  \frac{C(AB)^{1/4}(YW)^{3/4} Z^{1/2}}{|\xi|^{1/2}} \Big(\tau(\xi^2) +    \|\textbf{X}\|^{10\lambda} \frac{A^{1/2 }}{Y}\Big). 
\end{equation}

\subsection{Poisson summation in $c, z$}

We now return to \eqref{start}  and apply Poisson summation in $c, z$ getting \begin{equation*}
N_{\xi}(\textbf{X}, \Omega) = \sum_{wy \not= -\square} V\Big(\frac{y}{Y}\Big) V\Big(\frac{w}{W}\Big) \sum_a V\left(\frac{a}{A}\right) \frac{CZ}{a^2} \sum_{k_1, k_2} T(k_1, k_2, yw\xi^2, a)   \Psi(k_1, k_2, yw\xi^2, a)
\end{equation*}
 with $T$ as in \eqref{defT} and
 $$\Psi(k_1, k_2, x, a) = \int_{\mathbb{R}} \int_{\mathbb{R}} V\Big(\frac{c}{C}\Big)V\Big(\frac{z}{Z}\Big)V\Big(-\frac{c^2 + xz^2}{aB}\Big)  e\left(\frac{-k_1c - k_2z}{|a|}\right) dc\, dz.$$

 Let 
 \begin{equation*}
   K_1 = \Omega\Big(\frac{A}{C} + \frac{C}{B}\Big), \quad K_2 = \Omega\Big(\frac{A}{Z} + \frac{\xi^2 YWZ}{B}\Big). 
 \end{equation*}
 
 \begin{lemma}\label{lem4} For $|a| \asymp A$  we have
 $$\Psi(k_1, k_2, x, a)\ll_N 
  \frac{(CZAB)^{1/2}}{|x|^{1/4}}
 \Big(1 + \frac{|k_1|}{K_1} + \frac{|k_2|}{K_2}\Big)^{-N}$$
 and
 $$ (x \partial_x)^{j_1} (a\partial_a)^{j_2} \Psi(0,0, x, a)\ll_{j_1, j_2} 
 \frac{(CZAB)^{1/2}}{|x|^{1/4}}
  \Omega^{j_1+j_2}  $$
 for arbitrary $N, j_1, j_2 \in \mathbb{N}_0$. 
 \end{lemma}
 
 \begin{proof}
   As in Lemma \ref{lem2} we observe that the volume of the $(c, z)$-region
   defined by $|c| \asymp C$, $|z| \asymp Z$, $|c^2 + xz^2| \asymp AB$ is
   trivially bounded by $O(CZ)$, but also by $O(AB/|x|^{1/2})$.  Indeed, if
   $x > 0$, this is the volume of the ellipse $c^2 + xz^2 \ll AB$, while for
   $x < 0$ we have $c = \sqrt{|x|z^2 + O(AB)}$ for fixed $|z| \asymp Z$, which
   has volume $\ll AB/(|x|^{1/2}Z)$. Taking the geometric mean, we bound the
   $(c, z)$-volume by $O((CZAB)^{1/2}/|x|^{1/4})$.  The claims follow now by
   repeated partial integration and differentiation under the integral sign. As
   in the proof of Lemma \ref{lem2}, each application of $x\partial_x$ is
   coupled with an integration by parts in $z$.\end{proof}
 
As before we decompose
$$N_{\xi}(\textbf{X}, \Omega)  = \tilde{M}_1 + \tilde{M}_2$$
where $\tilde{M}_1$ is the off-diagonal contribution $(k_1, k_2) \not= 0$ and
$\tilde{M}_2$ is the diagonal contribution $k_1 = k_2 = 0$.  For notational
simplicity let 
$$ \Xi(k_1, k_2) = 
\frac{\sqrt{CZAB}}{|\xi|^{1/2} (YW)^{1/4}} \Big(1 + \frac{|k_1|}{K_1} +
\frac{|k_2|}{K_2}\Big)^{-10}.$$ 
In the following we write $(a, b^{\infty}) := \max_n (a, b^n)$ and $[d_1, d_2]$ for the least common multiple of $d_1$ and $d_2$. 
Using Lemma \ref{lem3} and Lemma \ref{lem4} we
obtain
 \begin{displaymath}
 \begin{split}
\tilde{M}_1\ll &  \sum_{wy \not= -\square} V\Big(\frac{y}{Y}\Big) V\Big(\frac{w}{W}\Big) \sum_a V\left(\frac{a}{A}\right) \frac{1}{A^2} \sum_{(k_1, k_2) \not= (0, 0)}(a, \xi^2yw)^{1/2}(a, 2^{\infty}(k_1^2 \xi^2 yw + k_2^2)) \Xi(k_1, k_2)\\
  \ll &  \sum_{d_1, d_2 \ll A}   \sum_{\substack{wy \not= -\square\\ d_1 \mid \xi^2 yw}} V\Big(\frac{y}{Y}\Big) V\Big(\frac{w}{W}\Big) \sum_{[d_1, d_2] \mid a} V\left(\frac{a}{A}\right) \frac{1}{A^2} \sum_{\substack{(k_1, k_2) \not= (0, 0)\\ d_2 \mid 2^{\infty}(k_1^2 \xi^2 yw + k_2^2)}} d_1^{1/2} d_2  \Xi(k_1, k_2)\\
   \ll & \sum_{d_1, d_2 \ll A}   \sum_{\substack{wy \not= -\square\\ d_1 \mid \xi^2 yw}} V\Big(\frac{y}{Y}\Big) V\Big(\frac{w}{W}\Big)    \frac{  \log A}{A} \sum_{\substack{(k_1, k_2) \not= (0, 0)\\ d_2 \mid  k_1^2 \xi^2 yw + k_2^2}} \frac{d_1^{1/2} d_2}{[d_1, d_2]} \Xi(k_1, k_2).
 \end{split}
 \end{displaymath} 
We write $k_1^2\xi^2 yw = -k_2^2 + \alpha d_2$. Since $yw \not= -\square$ and $(k_1, k_2) \not= (0, 0)$, we have $\alpha\not= 0$, and moreover $\alpha \equiv - k_2^2 \, (\text{mod } d_1)$. Once $\alpha$ and $k_2$ are chosen, the variables $y, w, k_1$ are determined up to a divisor function. We conclude the upper bound
 \begin{displaymath}
 \begin{split}
\tilde{M}_1  \ll &  \sum_{d_1, d_2 \ll A }      \frac{  \log A}{A} \Big(1 + \frac{K_2}{\sqrt{d_1}^+}\Big)   \frac{K_2^2+ K_1^2 \xi^2 YW}{d_2 }  \frac{d_1^{1/2} d_2}{[d_1, d_2]}
\frac{\sqrt{CZAB}}{|\xi|^{1/2} (YW)^{1/4}}
\|\textbf{X}\|^{\varepsilon}\\
  \end{split}
 \end{displaymath} 
and so
\begin{equation*}
\tilde{M}_1 \ll   \frac{ 1}{A}  (A^{1/2} + K_2)   ( K_2^2+ K_1^2 \xi^2 YW) \frac{\sqrt{CZAB}}{|\xi|^{1/2} (YW)^{1/4}}
\|\textbf{X}\|^{\varepsilon}. 
\end{equation*}
We now invoke \eqref{good} several times to conclude that
 $$ \frac{  (A^{1/2} + K_2) }{A}   \Big( \frac{K_2^2}{YW}+ K_1^2 \xi^2 \Big )  \|\textbf{X}\|^{\varepsilon} \ll \Omega^{4} \frac{\|\textbf{X}\|^{10\lambda}}{Z}.$$
 Thus we obtain the simplified bound
\begin{equation}\label{tildem1}
\tilde{M}_1 \ll \frac{\sqrt{CZAB}(YW)^{3/4}}{|\xi|^{1/2} }
 \Omega^{4} \frac{\|\textbf{X}\|^{10\lambda}}{Z}. 
\end{equation}
Note that in order to derive this bound we did not use any property of the weight $V(a/A)$, except that it bounds $a \ll A$. In particular, \eqref{tildem1} continues to hold for all $\Omega$, if $V(a/A)$ is replaced with the characteristic function on $A  \leq |a| \leq A(1+\Delta)$.

On the other hand, again by Lemma \ref{lem3} we have
\begin{equation}\label{again}
\begin{split}
\tilde{M}_2 = 2\sum_{y w\not= -\square}&V\Big( \frac{y}{Y}\Big) V\Big(\frac{w}{W}\Big)  \sum_{\substack{d_1, d_2\\ (\xi^2yw, d_2) = \square}} V\Big(\frac{  d_1d_2}{A}\Big)  \frac{\phi(d_2)}{d_1d_2^2}  \\
&\times \chi_{\frac{d_2}{(\xi^2yw, d_2)}}\Big( -\frac{\xi^2yw}{(\xi^2yw, d_2) }\Big)(\xi^2yw, d_2)^{1/2} \Psi(0, 0,  \xi^2yw, d_1d_2)
\end{split}
\end{equation}
where $d_2$ runs over all (positive or negative) discriminants and the factor 2 comes from the fact that $V$ is even and we have used the decomposition $|a| = d_1|d_2|$ from Lemma \ref{lem3}. 

Before we manipulate this further, we complete the proof of Proposition \ref{prop3}. Replacing $V(d_1d_2/A)$ in the previous display by the characteristic function on $A \leq |d_1d_2| \leq A(1+\Delta)$, we obtain by a simple divisor estimate the upper bound
$$\ll \tau(\xi^2) YW \Delta \|\textbf{X}\|^{\varepsilon}\frac{\sqrt{CZAB}}{|\xi|^{1/2} (YW)^{1/4}}$$
for the right hand side of \eqref{again}. 
By \eqref{good}, the factor $\tau(\xi^2)$ can be absorbed into $\|\textbf{X}\|^{\varepsilon}$. Combining this with \eqref{tildem1}, 
we obtain 
$$N_{\xi}(\textbf{X}, \Delta) \ll (ABC)^{1/2}(YW)^{3/4} Z^{1/2} \Big(\frac{\Delta \|\textbf{X}\|^{\varepsilon}}{|\xi|^{1/2}} +  \frac{\|\textbf{X}\|^{10\lambda}}{Z}\Big).$$
 Together with \eqref{delta1} we obtain
 $$N_{\xi}(\textbf{X}, \Delta) \ll (AB)^{1/4}C^{1/2}( (AB)^{1/4} + C^{1/2})(YW)^{3/4} Z^{1/2} \Big(\frac{\Delta \|\textbf{X}\|^{\varepsilon} }{|\xi|^{1/2}}+ \|\textbf{X}\|^{10\lambda} \min\Big(\frac{A^{1/2}}{Y}, \frac{1}{Z}\Big)\Big).$$
 Using \eqref{trick}, we replace the second appearance of $C^{1/2}$ with $(AB)^{1/4} + C^{1/4} (|\xi| Z)^{1/4}(YW)^{1/8}$.  By several applications of \eqref{good} we have
\begin{equation}\label{combinederror}
 \min\Big(   \frac{A^{1/2  }}{Y} , \frac{1 }{Z}\Big) \leq \frac{A^{1/4}}{(YZ)^{1/2}} \leq \frac{Y^{\lambda}}{B^{1/4}}\frac{(AB)^{1/4}}{(YWZ^2)^{1/4}} \leq \frac{(YAB)^{\lambda}}{B^{1/4}} \ll  \| \textbf{X} \|^{-1/5-15\lambda}
 \end{equation}
which completes the proof of Proposition \ref{prop3}.

After this interlude we return to \eqref{again}.  With $\delta^2 = (\xi^2yw, d_2)$ we rewrite this as
$$\tilde{M}_2 =2\sum_{\delta} \sum_{ \substack{\xi^2yw\equiv 0 \, (\delta^2)\\ yw \not= - \square}} V\Big( \frac{y}{Y}\Big) V\Big(\frac{w}{W}\Big)  \sum_{ d_1, d_2} V\Big(\frac{d_1d_2\delta^2}{A}\Big)  \frac{\phi(d_2 \delta^2)}{d_1d_2^2 \delta^3} \chi_{d_2} \Big( \frac{-\xi^2yw}{\delta^2} \Big) \Psi(0, 0,  \xi^2yw, d_1d_2 \delta^2)$$
where the condition $(\xi^2yw/\delta^2, d_2) = 1$ is automatic from the Jacobi symbol. Write $\delta' = \delta/(\delta, \xi)$, so $\delta'^2 \mid yw$.  We write $(\delta'^2, y) = f$, so that $g = \frac{\delta'^2}{f} \mid w$, and we obtain
$$  2\sum_{\delta} \sum_{fg =\delta'^2} \sum_{\substack{wy \not= -\square\\ (g, y) = 1}} V\Big( \frac{fy}{Y}\Big) V\Big(\frac{g w}{W}\Big)  \sum_{ d_1, d_2} V\Big(\frac{d_1d_2\delta^2}{A}\Big)   \frac{\phi(d_2 \delta^2)}{d_1d_2^2 \delta^3} \chi_{d_2}\Big(-\frac{\xi^2yw}{(\delta, \xi)^2}\Big) \Psi(0, 0,\xi^2\delta'^2yw,  d_1d_2 \delta^2 ).$$
We   write $\delta = \delta' \xi_1$ with $\xi_1\xi_2 =  \xi$ and $(\xi_2, \delta') = 1$ getting
\begin{displaymath}
\begin{split}
\tilde{M}_2 = 2\underset{(\xi_2, fg) = 1}{\sum_{\xi_1\xi_2 = \xi} \sum_{fg =\square}} &\sum_{\substack{wy \not= -\square\\ (g, y) = 1}} V\Big( \frac{fy}{Y}\Big) V\Big(\frac{g w}{W}\Big)  \sum_{ d_1, d_2} V\Big(\frac{d_1d_2fg\xi_1^2}{A}\Big)   \frac{\phi(d_2 fg \xi_1^2)}{d_1d_2^2 (fg)^{3/2}\xi_1^3} \\
&\times \chi_{d_2}( -\xi_2^2yw ) \Psi(0, 0, \xi^2fgyw, d_1d_2 fg\xi_1^2).
\end{split}
\end{displaymath}
We further decompose
$$\tilde{M}_2 = \tilde{M}_3 + \tilde{M}_4$$
where $\tilde{M}_3$ is the contribution $d_2 = \square$ and $\tilde{M}_4$ is the contribution $d_2\not= \square$. 
We have \begin{equation}\label{main}
\begin{split}
\tilde{M}_3 = 2\underset{(\xi_2, fg) = 1}{\sum_{\xi_1\xi_2 = \xi} \sum_{fg =\square}} &\sum_{\substack{wy \not= -\square\\ (g, y) = 1}} V\Big( \frac{fy}{Y}\Big) V\Big(\frac{g w}{W}\Big)  \sum_{ \substack{d_1, d_2 \\ (d_2,\xi_2yw) =1} } V\Big(\frac{d_1d_2^2fgF_1^2}{A}\Big)   \frac{\phi(d_2^2 fg \xi_1^2)}{d_1d_2^4 (fg)^{3/2}\xi_1^3} \\
&\times  \Psi(0, 0,  \xi^2fgyw, d_1d_2^2 fg\xi_1^2).
\end{split}
\end{equation}
We will later evaluate this asymptotically, but for now we content ourselves with the upper bound
\begin{equation}\label{tildem3}
\begin{split}
\tilde{M}_3& \ll \sum_{\xi_1\xi_2 = \xi} \sum_{fg =\square} \sum_{w, y} V\Big( \frac{fy}{Y}\Big) V\Big(\frac{g w}{W}\Big) \sum_{d_1, d_2 } V\Big(\frac{d_1d_2^2fg\xi_1^2}{A}\Big)   \frac{1}{d_1d_2^2 (fg)^{1/2}\xi_1} \frac{\sqrt{CZAB}}{|\xi|^{1/2} (YW)^{1/4}}\\
& \ll  \tau(\xi)\frac{\sqrt{CZAB}(YW)^{3/4}}{|\xi|^{1/2}}.
\end{split}
\end{equation}
To bound $\tilde{M}_4$, we insert  a smooth partition of unity localizing  $|d_2| \asymp D_2$, say, so that  
$$\tilde{M}_4 =  \sum_{ D_2} \tilde{M}_4(D_2)$$
where $D_2 \ll A$ runs over powers of 2 and
\begin{displaymath}
\begin{split}
\tilde{M}_4(D_2) =  \underset{(\xi_2, fg) = 1}{\sum_{\xi_1\xi_2 = \xi} \sum_{fg =\square}} &\sum_{\substack{wy \not= -\square\\ (g, y) = 1}} V\Big( \frac{fy}{Y}\Big) V\Big(\frac{g w}{W}\Big)  \sum_{ d_1} \sum_{\substack{d_2 \not= \square\\ (d_2, \xi) = 1}} V\Big(\frac{d_1d_2fg\xi_1^2}{A}\Big) W\Big(\frac{d_2}{D_2}\Big)  \\  &\times  \frac{\phi(d_2 fg \xi_1^2)}{d_1d_2^2 (fg)^{3/2}\xi_1^3} \chi_{d_2}( -yw ) \Psi(0, 0, \xi^2fgyw, d_1d_2 fg\xi_1^2)
\end{split}
\end{displaymath}
for  a suitable smooth compactly supported  function $W$. We estimate $\tilde{M}_4(D_2)$ in two ways. First we re-insert the contribution of $yw = -\square$ at the cost of an error 
$$\ll \tau(\xi)(YW)^{1/2} \frac{\sqrt{CZAB}}{|\xi|^{1/2} (YW)^{1/4}}$$
and then sum over $y, w$ with Lemma \ref{lem5} and \ref{lem4} getting a bound
\begin{equation}\label{e1}
\ll \tau(\xi) (YW)^{1/2} \frac{\sqrt{CZAB}}{|\xi|^{1/2} (YW)^{1/4}} 
 (\Omega D_2)^{1/2+\varepsilon}.
\end{equation}
Obviously this majorizes the previous error term. 

Alternatively, we can also re-insert the contribution of $d_2 = \square$ at the cost of an error
$$\ll \tau(\xi) \frac{YW}{D_2^{1/2}} \frac{\sqrt{CZAB}}{|\xi|^{1/2} (YW)^{1/4}}$$
and then sum over $d_2$ with Lemma \ref{lem5} and \ref{lem4} getting a bound
\begin{equation}\label{e2}
\ll \tau(\xi)^2 \frac{YW}{D_2^{1/2}}\frac{\sqrt{CZAB}}{|\xi|^{1/2} (YW)^{1/4}} 
 (\Omega YW)^{1/4+\varepsilon},
\end{equation}
which again majorizes the previous error. 

Combining \eqref{e1} and \eqref{e2}, we bound $\tilde{M}_4(D_2) $ by 
\begin{displaymath}
\begin{split}
& \ll \tau(\xi)^2 \frac{\sqrt{CZAB} (YW)^{3/4}}{|\xi|^{1/2}} \min\Big( \frac{(\Omega D_2)^{1/2 + \varepsilon}}{(YW)^{1/2}}, 
 \frac{(\Omega YW)^{1/4+\varepsilon}}{D_2^{1/2}}\Big)\\
&\ll \tau(\xi)^2  \frac{\sqrt{CZAB} (YW)^{3/4}}{|\xi|^{1/2}}  \frac{\Omega^{1/2+\varepsilon}}{(YW)^{1/20 - \varepsilon} D_2^{1/10- \varepsilon}}. 
\end{split}
\end{displaymath}
Summing over $D_2$, we finally obtain
\begin{equation*}
  \tilde{M}_4 \ll\sqrt{CZAB} (YW)^{3/4}\frac{\Omega^{1/2 + \varepsilon}}{(YW)^{1/20 - \varepsilon}}. 
\end{equation*}
Combining this with \eqref{tildem1} we have
\begin{displaymath}
\begin{split}
\tilde{M}_1 + \tilde{M}_4 & \ll  \sqrt{CZAB} (YW)^{3/4}\Big(  \frac{\Omega^{2/3}}{(YW)^{1/21}}
 + \Omega^4\frac{\|\textbf{X}\|^{10\lambda}}{Z}\Big)\\
 &\ll\sqrt{CZAB} (YW)^{3/4} \Omega^4  \Big(  \frac{\|\textbf{X}\|^{10\lambda}}{Z} + \frac{1}{(YW)^{1/21}}\Big). 
\end{split}
\end{displaymath}
Since $N_{\xi}(\textbf{X}, \Omega) = \tilde{M}_1 + \tilde{M}_3 + \tilde{M}_4$, this completes the proof of the second half of Proposition \ref{prop4}. 

On the other hand, choosing $\Omega  = 3$, we also invoke  \eqref{tildem3} to obtain our second important bound
\begin{equation*}
  N_{\xi}(\textbf{X}) \leq  N_{\xi}(\textbf{X}, 3) = \tilde{M}_1 + \tilde{M}_3 + \tilde{M}_4  \ll  \frac{(ABC)^{1/2}  (YW)^{3/4} Z^{1/2}}{|\xi|^{1/2}} \Big(\tau(\xi)^2 +\frac{\| \textbf{X}\|^{10\lambda}}{ Z} \Big) . \end{equation*}
We now combine this with \eqref{first}   to conclude
$$N_{\xi}(\textbf{X}) \ll \Big( (AB)^{1/4}C^{1/2}(C^{1/2} + (AB)^{1/4}))(YW)^{3/4} Z^{1/2}
\Big)\Big(1+\| \textbf{X}\|^{10\lambda}  \min\Big(   \frac{A^{1/2  }}{Y} , \frac{1 }{Z}\Big)\Big). 
$$
The proof of   Proposition \ref{prop2} is now completed by \eqref{combinederror}.  

\section{An asymptotic formula}

We now upgrade the previous upper bound to an asymptotic formula for $N_{\xi}(\textbf{X})$. The main term features the singular series and the singular integral that would follow from a formal application of the circle method. With this in mind we define
\begin{equation}\label{defsing}
\mathcal{E}_{\xi} =  \sum_{q} \frac{1}{q^6}\underset{d \, (\text{mod } q)}{\left.\sum\right.^{\ast} }  \sum_{a, b, c, y,w, z\, (\text{mod }q)} e\Big( \frac{d(ab + c^2 + \xi^2wyz^2)}{q}\Big).
\end{equation}
For later purposes we compute this as an Euler product. We have
\begin{displaymath}
\begin{split}
  \mathcal{E}_{\xi} & =   \sum_{q}\underset{d \, (\text{mod } q)}{\left.\sum\right.^{\ast} }  \frac{1}{q^5}\sum_{  c, y,w, z\, (\text{mod }q)} e\Big( \frac{d(  c^2 + \xi^2wyz^2}{q}\Big)  = \sum_{q} \frac{1}{q^4}\underset{d \, (\text{mod } q)}{\left.\sum\right.^{\ast} }  \sum_{  c \, (\text{mod } q) }\sum_{\xi^2yz^2 \equiv 0 \, (\text{mod }q)} e\Big( \frac{d c^2  }{q}\Big).\\
  \end{split}
\end{displaymath}
We have
$$\underset{d \, (\text{mod } q)}{\left.\sum\right.^{\ast} }  \sum_{  c \, (\text{mod } q) }e\Big( \frac{d c^2  }{q}\Big) = \begin{cases} q^{1/2}\phi(q), & q = \square,\\
0, &q \not= \square,\end{cases}$$
 so we   conclude 
  \begin{displaymath}
\begin{split}
  \mathcal{E}_{\xi}     = \sum_{q }  \frac{\phi(q^2)}{q^7} \sum_{\xi^2yz^2 \equiv 0 \, (\text{mod }q^2)} 1 =  \sum_{q} \frac{\phi(q) (\xi, q)^4}{q^6}   \sum_{ yz^2 \equiv 0 \, (\text{mod } (\frac{q}{(\xi, q)})^2)}1. 
  \end{split}
\end{displaymath}
If $p$ is a prime and $n\geq 1$, then a simple combinatorial argument shows that the number of pairs $y, z$ (mod $p^{2n}$) with $p^{2n} \mid yz^2$ equals
$$ p^{3n} + p^{3n-1} - p^{2n-1}.$$
For a prime $p$ let $r_p = v_p(\xi)$ denote the $p$-adic valuation of $\xi$. 
 Evaluating geometric sums, we finally conclude
   \begin{displaymath}
\begin{split}
  \mathcal{E}_{\xi}  &  =  \prod_p \Big(1 + \sum_{n=1}^{r_p} \frac{(p-1)p^{n-1+4n}}{p^{6n}} + \sum_{n=r_p+1}^{\infty}   \frac{(p-1)p^{n-1 + 4r_p}}{p^{6n}} \big(p^{3(n-r_p)} + p^{3(n-r_p)-1} - p^{2(n-r_p)-1}\big)\Big)\\
  & = \prod_p \Big(1 + \frac{1+p + p^2 - p^{2-r_p}}{p(1+p+p^2)}\Big). 
   \end{split}
\end{displaymath}
In particular, the Euler product is absolutely convergent and satisfies
\begin{equation}\label{sing-bound}
\mathscr{E}_{\xi} \ll \xi^{\varepsilon}. 
\end{equation}

We also define the singular integral for a tuple $\textbf{X} = (X_{11}, X_{12}, X_{22}, X_{31}, X_{32}, X_{33})$ with $X_{ij} \geq 1$ as
$$\mathcal{I}_{\xi}(\textbf{X}) = \int_{\mathbb{R}} \int_{\frac{1}{2}X_{ij} \leq |x_{ij}| \leq X_{ij}} e\big((x_{11}x_{12} + x_{21}^2 + \xi^2 x_{31}x_{32}x_{33}^2)\alpha\big) d\textbf{x}\, d\alpha .$$
That the $\alpha$-integral is absolutely convergent follows from the estimates
\begin{equation}\label{estimates}
\begin{split}
&  \iint_{\substack{\frac{1}{2}A \leq |a| \leq A\\ \frac{1}{2} B \leq |b| \leq B}}
e(ab\alpha)da\, db  \ll \min(AB, |\alpha|^{-1}) \ll  \frac{\|\textbf{X}\|^{\varepsilon} (AB)^{1/2}}{|\alpha|^{1/2}(|\alpha|^{\varepsilon} + |\alpha|^{-\varepsilon} )},\\
& \int_{\frac{1}{2}C \leq |c| \leq C}  e(c^2\alpha )dc \ll \min(C, |C\alpha|^{-1}) \ll \frac{\|\textbf{X}\|^{\varepsilon} C^{1/2}}{|\alpha|^{1/4}(|\alpha|^{\varepsilon} + |\alpha|^{-\varepsilon} )} ,\\
& \iiint_{\substack{\frac{1}{2}Y \leq |y| \leq Y\\ \frac{1}{2}W \leq |w| \leq W\\\frac{1}{2}Z \leq |z| \leq Z }} 
 e(\xi^2ywz^2\alpha )dy\, dw\, dz \ll \min(YWZ, |\xi^2Z\alpha|^{-1})
  \ll   \frac{\|\textbf{X}\|^{\varepsilon} (YW)^{3/4} X^{1/2}}{|\alpha|^{1/4}(|\alpha|^{\varepsilon} + |\alpha|^{-\varepsilon} )}. 
\end{split}
\end{equation}
For future reference we state the similar bounds
\begin{equation}\label{estimates1}
\begin{split}
& \iint_{\substack{\frac{1}{2}A \leq |a| \leq A\\ \frac{1}{2} B \leq |b| \leq B}} e(ab\alpha)da\, db-    \iint_{\mathbb{R}^2} V\Big(\frac{a}{A}\Big)V\Big(\frac{b}{B}\Big)
e(ab\alpha)da\, db \\
&\quad\quad\quad\quad \quad\quad 
\ll \min(\Omega^{-1}AB, |\alpha|^{-1})\ll \frac{\Omega^{-1/2} \|\textbf{X}\|^{\varepsilon} (AB)^{1/2}}{|\alpha|^{1/2}(|\alpha|^{\varepsilon} + |\alpha|^{-\varepsilon} )},\\
& \int_{\frac{1}{2}C \leq |c| \leq C}  e(c^2\alpha )dc -     \int_{\mathbb{R}} V\Big(\frac{c}{C}\Big) e(c^2\alpha )dc \ll \min(\Omega^{-1}C, |C\alpha|^{-1}) \ll \frac{\Omega^{-3/4}\|\textbf{X}\|^{\varepsilon} C^{1/2}}{|\alpha|^{1/4}(|\alpha|^{\varepsilon} + |\alpha|^{-\varepsilon} )} ,\\
& \iiint_{\substack{\frac{1}{2}Y \leq |y| \leq Y\\ \frac{1}{2}W \leq |w| \leq W\\\frac{1}{2}Z \leq |z| \leq Z }} 
 e(\xi^2ywz^2\alpha )dy\, dw\, dz  -   \iiint_{\mathbb{R}^3} V\Big(\frac{y}{Y}\Big)V\Big(\frac{w}{W}\Big)V\Big(\frac{z}{Z}\Big) e(\xi^2ywz^2\alpha )dy\, dw\, dz \\
 &\quad\quad \quad\quad\quad\quad \ll \min(\Omega^{-1}YWZ, |\xi^2Z\alpha|^{-1}) \ll  \frac{\Omega^{-3/4}\|\textbf{X}\|^{\varepsilon} (YW)^{3/4} Z^{1/2}}{|\alpha|^{1/4}(|\alpha|^{\varepsilon} + |\alpha|^{-\varepsilon} )}. 
\end{split}
\end{equation}

Our aim in this section is to prove the following asymptotic formula. Let 
\begin{equation}\label{zeta}
\begin{split}
& \bm \zeta^{(1)} = (1/2, 1/4, 1/4), \quad \bm \zeta^{(2)} = (3/4, 1/8, 1/8), \quad  \bm \zeta^{(3)} = (5/8, 1/8, 1/4). 
\end{split}
\end{equation}

\begin{prop}\label{asymp-prop} The exists $\delta > 0$ with the following property. For  $\emph{\textbf{X}} = (X_{ij}), \xi$ satisfying \eqref{good} and any $\varepsilon > 0$ we have
$$N_{\xi}(\emph{\textbf{X}}) - \mathcal{E}_{\xi} \mathcal{I}_{\xi}(\emph{\textbf{X}}) \ll_{\varepsilon} (\min_{ij} X_{ij})^{-\delta}\| \emph{\textbf{X}}\|^{\varepsilon}\sum_{i=1}^3\emph{\textbf{X}}^{(\bm\zeta^{(i)})}. $$
\end{prop}

The rest of this section is devoted to the proof. As a first step we estimate the effect of smoothing. To this end we recall the definition of $N_{\xi}(\textbf{X}, \Delta)$ and $N^{\ast}_{\xi}(\textbf{X}, \Delta)$: the former restricts one of the variables $c, y, w, z$ to a small interval, the latter the variable $a$. By symmetry and \eqref{good}, the bound in Proposition \ref{prop3} holds also when $b$ is restricted to a small interval.  Combining Propositions \ref{elem1} and \ref{prop3}, we obtain
\begin{equation}\label{desmooth}
N_{\xi}(\textbf{X})  = N_{\xi}(\textbf{X}, \Omega) + O\Big( (\| \textbf{X}\|^{\varepsilon} \Omega^{-1/2}  + \| \textbf{X}\|^{-1/5} )\sum_{i=1}^3\textbf{X}^{(\bm\zeta^{(i)})}\Big). 
\end{equation}

An evaluation of $N_{\xi}(\textbf{X}, \Omega)$ is given in Proposition \ref{prop4}, and we proceed to evaluate the main terms $M_2$ and $\tilde{M}_3$ defined in \eqref{diag} and \eqref{main}. 

\subsection{Computation of $M_2$}

Recall that 
$$M_2 = \sum_{a, c, z}V\Big(\frac{a}{A}\Big) V\Big(\frac{c}{C}\Big) V\Big(\frac{z}{Z}\Big)
 \frac{1}{a^2} \sum_{\substack{a_1a_2a_3 = |a|\\ a_3 \mid c^2}} a_1(\xi^2z^2, a_2a_3)a_3 \mu(a_2)  \Phi_{\xi}(0, 0, a, c, z) $$
where
$$\Phi_{\xi}(0, 0, a, c, z) =   \int_{\mathbb{R}^2}   V\Big(\frac{y}{Y}\Big) V\Big(\frac{w}{W}\Big) V\left(\frac{-(c^2 + \xi^2yw z^2)}{aB}\right) dy\, dw.$$
Let $$F_{\xi}(a, c, z) = V\Big( \frac{a}{A}\Big)V\Big( \frac{z}{Z}\Big)V\Big( \frac{c}{C}\Big) \Phi_{\xi}(0, 0, a, c, z).$$
By symmetry we can restrict $a, c, z$ to be positive at the cost of a factor $8$. We denote
by
$$\widehat{F}_{\xi}(s, u, v) =8 \int_{[0, \infty)^3}F_{\xi}(a, c, z) a^{s-1}  c^{u-1} z^{v-1} da\, dc\, dz$$
the   Mellin transform of  $F_{\xi}$. It is holomorphic in all three variables (since $V$ has compact support), and by
Lemma \ref{lem2} and partial integration we have
$$\widehat{F}_{\xi}(s, u, v)\ll_N  \frac{(AB)^{1/4}(YW)^{3/4}}{(Z|\xi|)^{1/2}}  A^{\Re s } C^{\Re u } Z^{\Re v} \Big(1 + \frac{|\Im s|}{S} + \frac{|\Im u|}{U} + \frac{|\Im  v|}{V}\Big) ^{-N}$$
 for all $N > 0$ where
 $$S =  \Omega\Big(1+ \frac{|\xi|^2 YWZ^2}{AB}\Big), \quad  U = \Omega\Big(1 + \frac{C^2}{AB}\Big), \quad V = \Omega. $$
 By \eqref{good} we have $S, U \ll \Omega  \| \textbf{X}\|^{10\lambda}$. Let us also define
$$L_{\xi}(s, u, v) =  \sum_{a_1, a_2, a_3}\sum_{z} \sum_{a_3 \mid c^2} \frac{ (\xi^2z^2, a_2a_3)  \mu(a_2)}{a_1^{1+s}a_2^{2+s} a^{1+s}_3c^u    z^v } , \quad \Re s > 0, \quad \Re u > 1, \quad \Re v > 1.    $$
By Mellin inversion we have
\begin{equation}\label{mellin}
M_2 =   \int_{(2)}\int_{(2)}\int_{(2)}  L_{\xi}(s, u, v)  \widehat{F}_{\xi}(s, u, v) \frac{ds\, du\, dv}{(2\pi i)^3}. 
\end{equation}
By a long, cumbersome and uninspiring, but completely straightforward computation based on geometric series we can compute $ L_{\xi}(s, u, v) $ as an Euler product. If $v_p(\xi) = r_p$, then $$L_{\xi}(s, u, v) = \zeta(s+1)\zeta(u) \zeta(v) \zeta(2+2s+u)\zeta(2s+u+v) H_{\xi}(s, u, v)$$ where the $p$-Euler factor of $H_{\xi}(s, u, v)$ is given by

{\tiny \begin{displaymath}
\begin{split}
&  \Big( 1 - p^{-3 - 6 s - 3 u - v} - p^{-2 - s - (2 s + u) r_p}+ p^{-1 - 
  s - (2 s + u) r_p} + p^{-2 - 3 s - 2 u - (2 s + u) r_p} - p^{-1 - 3 s - 
  2 u - (2 s + u) r_p} \\
  &- p^{-3 - 3 s - u - (2 s + u) r_p} + p^{-2 - 3 s - 
  u - (2 s + u) r_p} - p^{-3 - 2 s - u - (2 s + u) r_p} + p^{-2 - 2 s - 
  u - (2 s + u) r_p} + p^{-1 - 2 s - u - (2 s + u) r_p}\\
  & - p^{-2 s - 
  u - (2 s + u) r_p} + p^{-1 - s - u - (2 s + u) r_p} - p^{-s - 
  u - (2 s + u) r_p} + p^{-2 - s - v - (2 s + u) r_p} - p^{-1 - s - 
  v - (2 s + u) r_p} \\
  &- p^{-2 - 3 s - 2 u - v - (2 s + u) r_p}+ p^{-1 - 
  3 s - 2 u - v - (2 s + u) r_p} + p^{-3 - 3 s - u - 
  v - (2 s + u) r_p} - p^{-2 - 3 s - u - v - (2 s + u) r_p}\\
  & - p^{-2 - 
  2 s - u - v - (2 s + u) r_p} - p^{-1 - 2 s - u - 
  v - (2 s + u) r_p} + p^{-2 s - u - v - (2 s + u) r_p} - p^{-1 - s - u - 
  v - (2 s + u) r_p} + p^{-s - u - v - (2 s + u) r_p} \\
  &+ p^{-3 - 2 s - u - 
  v - (2 s + u) r_p} - p^{-2 - 3 s - 2 u } + p^{-2 - 5 s - 3 u - 
  v } - p^{-3 - 5 s - 2 u - v} - p^{-3 s - 2 u - v} + p^{-3 - 4 s - 
  2 u} + p^{-2 - 4 s - 2 u - v} \\
  &+ p^{-1 - 4 s - 2 u - v} - p^{-2 - 
  2 s - u} - p^{-1 - 2 s - u} - p^{-2 s - u - v} + p^{-3 - 3 s - 
  u} + p^{-1 - 3 s - u - v} + p^{-s - u} - p^{-1 - s}\Big)\Big(1 - \frac{1}{p^{2s+u}}\Big)^{-1}.
\end{split}
\end{displaymath}}
For $p \nmid \xi$ (i.e.\ $r_p = 0$) this simplifies considerably as
\begin{displaymath}
\begin{split}
&1 - p^{-2 - s} - p^{-3 - 2 s - u} + p^{-1 - s - u} + p^{-2 - s - 
  v} - p^{-1 - s - v} + p^{-3 - 4 s - 2 u - v} - p^{-2 - 3 s - 2 u - 
  v} \\
  &+ p^{-3 - 3 s - u - v}  + p^{-3 - 2 s - u - v} - p^{-2 - 2 s - 
  u - v} - p^{-1 - 2 s - u - v} - p^{-1 - s - u - v} + p^{-s - u - v}.
  \end{split}
\end{displaymath}
In particular, we see that $L_{\xi}(s, u, v)$ is holomorphic in 
\begin{equation}\label{region}
 \Re s \geq -1/5, \quad \Re v \geq 4/5, \quad \Re u \geq 4/5
 \end{equation}
except for polar divisors at $s=0$, $u=1$, $v=1$; away from the polar divisors it is (crudely) bounded by $\tau(\xi) ((1+|s|)(1 + |u|)(1 + |v|))^{1/8}$ in this region. Another computation shows that
$$\underset{s=0}{\text{res}}\, \underset{u=1}{\text{res}}\, \underset{v=1}{\text{res}} L_{\xi}(s, u, v) =  \prod_p \Big(1 + \frac{1+p + p^2 - p^{2-r_p}}{p(1+p+p^2)}\Big) = \mathcal{E}_{\xi}.$$
Shifting contours to the left in \eqref{mellin} we conclude that
$$M_2 =  \mathcal{E}_{\xi} \widehat{F}_{\xi}(0, 1, 1)  + O\Big(\Omega^{3}  |\xi|^{\varepsilon}(AB)^{1/4} C(YW)^{3/4} Z^{1/2}\Big( \| \textbf{X}\|^{30\lambda}(A^{-1/8} + C^{-1/8} )+ Z^{-1/8}\Big)\Big).$$
Using again  \eqref{good} and the now familiar device based on \eqref{trick}, we can write the error term as 
\begin{equation}\label{errorterm}
\| \textbf{X}\|^{\varepsilon} \Omega^{3}  \big( \textbf{X}^{(\frac{5}{8}, \frac{1}{8}, \frac{1}{4})} + \textbf{X}^{(\frac{3}{4}, \frac{1}{8}, \frac{1}{8})}\big)\min(A, C, Z)^{-1/10}.
\end{equation}
By definition and using symmetry again to remove the factor 8, we have 
\begin{equation*}
\begin{split}
&\widehat{F}_{\xi}(0, 1, 1) 
 =   \int_{\mathbb{R}^5}V\Big(\frac{a}{A}\Big)V\Big(\frac{z}{Z}\Big)V\Big(\frac{c}{C}\Big)    V\Big(\frac{y}{Y}\Big) V\Big(\frac{w}{W}\Big) V\left(\frac{-(c^2 + \xi^2ywz^2 )}{aB}\right) dy\, dw \, \frac{da}{|a|} \, dz\, dc. \\
\end{split}
\end{equation*}
By Fourier inversion we have
\begin{displaymath}
\begin{split}
V\left(\frac{-(c^2 + \xi^2ywz^2) }{a B}\right)& = \int_{\mathbb{R}}\int_{\mathbb{R}}V(b) e(b\alpha) db\, e\Big(  \alpha \frac{c^2 + \xi^2ywz^2 }{a B}\Big) d\alpha \\
&= |a|  \int_{\mathbb{R}}\int_{\mathbb{R}} V\Big(\frac{b}{B}\Big) e(ba \alpha) db\, e\big(  \alpha (c^2 + \xi^2ywz^2 )\big) d\alpha,
\end{split}
\end{displaymath}
so that $\widehat{F}_f(0, 1, 1)$ equals 
\begin{displaymath}
\begin{split}
  \int_{\mathbb{R} }\int_{\mathbb{R}^6} V\Big(\frac{a}{A}\Big)V\Big(\frac{z}{Z}\Big)V\Big(\frac{c}{C}\Big)    V\Big(\frac{y}{Y}\Big) V\Big(\frac{w}{W}\Big) V\Big(\frac{b}{B}\Big) e\big(  \alpha (ab+ c^2 + \xi^2ywz^2)\big)    db\,  dy\, dw \, da \, dz\, dc\, d\alpha.  \\
\end{split}
\end{displaymath}
It remains to remove the smoothing and quantify the error from replacing $V$ with the characteristic function on $[1/2, 1]$. By \eqref{estimates} and \eqref{estimates1}, we see that
\begin{equation}\label{final-desmooth}
\widehat{F}_f(0, 1, 1) - \mathcal{I}_{\xi}(\textbf{X}) \ll \Omega^{-1/2} \textbf{X}^{(\frac{1}{2}, \frac{1}{4}, \frac{1}{4})} \| \textbf{X}\|^{\varepsilon}.
\end{equation}
Combining this with \eqref{errorterm}, \eqref{desmooth}, the first part of Proposition \ref{prop4} and choosing $\Omega = \min(A,  C,  Z)^{1/50}$, we have shown
\begin{equation}\label{final1}
N_{\xi}(\textbf{X})- \mathcal{E}_{\xi} \mathcal{I}_{\xi} \ll \| \textbf{X} \|^{\varepsilon}\sum_{i=1}^3\textbf{X}^{(\bm\zeta^{(i)})}\Big( \min(\textbf{X}_{ij})^{-1/100} + \| \textbf{X} \|^{11\lambda} \frac{A^{1/2+1/50}}{Y}\Big) .
\end{equation}

\subsection{Computation of $\tilde{M}_3$}

The argument for $\tilde{M}_3$ is similar. Recall from \eqref{main} that
\begin{equation*}
\begin{split}
\tilde{M}_3 = 2\underset{(\xi_2, fg) = 1}{\sum_{\xi_1\xi_2 = \xi} \sum_{fg =\square}} &\sum_{\substack{wy \not= -\square\\ (g, y) = 1}} V\Big( \frac{fy}{Y}\Big) V\Big(\frac{g w}{W}\Big)  \sum_{ \substack{d_1, d_2 \\ (d_2,\xi_2yw) =1} } V\Big(\frac{d_1d_2^2fg\xi_1^2}{A}\Big)   \frac{\phi(d_2^2 fg \xi_1^2)}{d_1d_2^4 (fg)^{3/2}\xi_1^3} \\
&\times  \Psi(0, 0,  \xi^2fgyw, d_1d_2^2 fg\xi_1^2)
\end{split}
\end{equation*}
where $\Psi$ satisfies the bounds of Lemma \ref{lem4}. Recall that all variables run over positive integers, except for $y, w$ that run over positive and negative integers. 
We first add back the contribution $wy = -\square$ at the cost of an error
\begin{equation}\label{squares}
\ll \tau(\xi) \sqrt{YW}  \frac{\sqrt{CZAB} }{|\xi|^{1/2} (YW)^{1/4}}
\end{equation}
by estimating trivially the contribution of all variables. 

Since $V$ is even, we can rewrite $\tilde{M}_3$, up to the error \eqref{squares}, as 
\begin{equation*}
\begin{split}
4\sum_{\pm}\underset{(\xi_2, fg) = 1}{\sum_{\xi_1\xi_2 = \xi} \sum_{fg =\square}} &\sum_{\substack{wy \not= -\square\\ (g, y) = 1}} V\Big( \frac{fy}{Y}\Big) V\Big(\frac{g w}{W}\Big)  \sum_{ \substack{d_1, d_2 \\ (d_2,\xi_2yw) =1} } V\Big(\frac{d_1d_2^2fg\xi_1^2}{A}\Big)   \frac{\phi(d_2^2 fg \xi_1^2)}{d_1d_2^4 (fg)^{3/2}\xi_1^3} \\
& \times \Psi(0, 0, \pm \xi^2fgyw, d_1d_2^2 fg\xi_1^2)
\end{split}
\end{equation*}
where now all variables run over positive integers. 

Let
$$G_{\xi}(a, y, w) = 4\sum_{\pm} V\Big(\frac{a}{A}\Big)V\Big(\frac{y}{Y}\Big) V\Big(\frac{w}{W}\Big) \Psi (0, 0, \pm \xi^2yw, a).$$
Then
\begin{equation}\label{then}
\begin{split}
\tilde{M}_3 = &\underset{(\xi_2, fg) = 1}{\sum_{\xi_1\xi_2 = \xi} \sum_{fg =\square}} \sum_{  (g, y) = 1} \sum_{ \substack{d_1, d_2 \\ (d_2,\xi_2yw) =1} }    \frac{\phi(d_2^2 fg \xi_1^2)}{d_1d_2^4 (fg)^{3/2}\xi_1^3} G(fy, gw, d_1d_2^2 fg \xi_1^2) \\
&+ O\big((CZAB)^{1/2}(YW)^{1/4}\big). 
\end{split}
\end{equation}
 As before, we denote by $\widehat{G}_{\xi}(s, u, v)$ the Mellin transform of $G_{\xi}(a, y, w) $; it is entire in all three variables and by Lemma \ref{lem4} satisfies
 $$\widehat{G}_{\xi}(s, u, v)\ll_N \frac{\sqrt{CZAB} }{|\xi|^{1/2} (YW)^{1/4}}
 A^{\Re s-1} Y^{\Re u - 1} W^{\Re v - 1} (1 + |s| + |u| + |v|)^{-N}.$$
 Next we define
 $$\tilde{L}_{\xi}(s, u, v) = \underset{(\xi_2, fg) = 1}{\sum_{\xi_1\xi_2 = \xi} \sum_{fg =\square}} \sum_{  (g, y) = 1} \sum_{  (d_2,\xi_2yw) =1 } \frac{\phi(d_2^2 fg \xi_1^2)}{ d_2^4 (fg)^{3/2}\xi_1^3 (d_2^2fg\xi_1^2)^s (fy)^u(gw)^v} $$
which is absolutely convergent in $\Re u, \Re v > 1$, $\Re s > -1/2$. Then by Mellin inversion we have 
 \begin{equation*} 
\begin{split}
\tilde{M}_3 = \int_{(2)}  \int_{(2)}  \int_{(2)}  \tilde{L}_{\xi}(s, u, v) \zeta(s+1) \widehat{G}_{\xi}(s, u, v) \frac{ds\, du \, dv}{(2\pi i)^3} .    \end{split}
\end{equation*}
As before, we analyze $\tilde{L}_{\xi}(s, u, v) $ by computing its Euler product expansion. A similarly long and cumbersome computation yields
$$\tilde{L}_{\xi}(s, u, v) = \zeta(u)\zeta(v) H_{\xi}(s, u, v)$$
where $H_{\xi}(s, u, v) \ll \tau(\xi)$ is holomorphic in the same region \eqref{region} and $H_{\xi}(0, 1, 1) = \mathcal{E}_{\xi}$. Shifting contours, we conclude as before
\begin{displaymath}
\begin{split}
\tilde{M}_3 - \mathcal{E}_{\xi} \widehat{G}_{\xi}(0, 1, 1) \ll 
\| \textbf{X}\|^{\varepsilon} \Omega^3 (CZAB)^{1/2}(YW)^{3/4} \min(Y, W, A)^{-1/10} 
\end{split}
\end{displaymath}
which also contains the error term from \eqref{then}. Unraveling the definitions and using symmetry to remove the factor $4$ and the $\pm$-sign, we see that $\widehat{G}_{\xi}(0, 1, 1) = \widehat{F}_{\xi}(0, 1, 1)$, so that by \eqref{final-desmooth} we get
\begin{displaymath}
\begin{split}
\tilde{M}_3 - \mathcal{E}_{\xi}\mathcal{I}_{\xi} \ll 
\| \textbf{X}\|^{\varepsilon} \textbf{X}^{(\frac{1}{2}, \frac{1}{4}, \frac{1}{4})}\big( \Omega^3 \min(Y, W, A)^{-1/10} + \Omega^{-1/2}\big). 
\end{split}
\end{displaymath}
 Combining this with  \eqref{desmooth} and the second part of Proposition \ref{prop4} and choosing $\Omega = \min(Y, W, A)^{1/50}$, we have shown
\begin{equation}\label{final2}
N_{\xi}(\textbf{X})- \mathcal{E}_{\xi} \mathcal{I}_{\xi} \ll \| \textbf{X} \|^{\varepsilon}\sum_{i=1}^3\textbf{X}^{(\bm\zeta^{(i)})}\Big( \min(\textbf{X}_{ij})^{-1/100} + \| \textbf{X} \|^{10\lambda} \frac{Y^{4/50}}{Z}\Big) .
\end{equation}
It remains to combine \eqref{final1} and \eqref{final2}. As in \eqref{combinederror} we conclude
 \begin{equation*}
 \min\Big(   \frac{A^{1/2 +1/50 }}{Y} , \frac{Y^{4/50} }{Z}\Big)  \leq \frac{(YAB)^{\lambda}A^{1/100} Y^{2/50}}{B^{1/4}} \ll  \| \textbf{X} \|^{-1/100-15\lambda}.
 \end{equation*}
This completes the proof of Proposition \ref{asymp-prop}.

\section{Introducing the height conditions}
 
Let $P > 1$ be a large  parameter. Let as before $\textbf{X} = (A, B, C, Y, W, Z)$, where all entries are restricted to powers of $2$. The condition that the entries of $\textbf{X}$ are powers of 2 will remain in force throughout this section. 
Let $\textbf{x} = (a, b, c, y, w, z) \in \mathbb{N}^6$, $\textbf{g} = (\eta, \xi) \in \mathbb{N}^2$.   For given $\textbf{x}$, $\textbf{g}$ let $\mathcal{X} = \mathcal{X}(P, \textbf{g}, \textbf{x})$ be the set of tuples $\textbf{X}$ satisfying
\begin{equation}\label{height}
\begin{split}
 \max(aA, bB, cC)^2 \max(yY, wW) \leq \frac{P}{\eta^2\xi}, \quad \max(yY, wW)^3 (zZ)^2 \leq \frac{P}{\eta^2 \xi^3}.
\end{split}
\end{equation} 
Fix some sufficiently small $\lambda > 0$. We call a pair $(\textbf{X}, \textbf{g} = (\eta, \xi))$ \emph{bad} if \eqref{good} is violated (i.e., one of these inequalities does not hold), otherwise we call it \emph{good}. We call it \emph{very good} if the following stronger version of \eqref{good} holds:
\begin{equation}\label{verygood}
\begin{split}
&\min(A, B, C) \geq \max(A, B, C) (\log P)^{-100}, \\
& \min(Y, W) \geq \max(Y, W) (\log P)^{-100}, \\
&\min(AB, C^2, YWZ^2) \geq \max(AB, C^2, YWZ^2) (\log P)^{-100}, \\
& |\xi| \leq (\log P)^{100},\\
& P(\log P)^{-100} \leq CYWZ  \leq P.
\end{split}
  \end{equation}
Clearly there are at most $O((\log P)^6)$ tuples $\textbf{X}$  (the entries being powers of 2) satisfying \eqref{height}. Moreover, it is easy to see that there are at most 
\begin{equation}\label{atmost}
\ll \log H (\log\log P)^6
\end{equation}
such tuples satisfying \eqref{height}, \eqref{verygood} and 
\begin{equation}\label{H}
\min(A, B, C, Y, W, Z) \leq H.
\end{equation}
Let $\mathcal{X}_{\text{bad}}(P, \textbf{g}, \textbf{x})$ be the set of $\textbf{X} \in \mathcal{X}(P, \textbf{g}, \textbf{x})$ such that $(\textbf{X}, \textbf{g})$ is bad. Let $\mathcal{X}_{H}(P, \textbf{g}, \textbf{x})$ be the set of $\textbf{X} \in \mathcal{X}(P, \textbf{g}, \textbf{x})$ such that \eqref{H} holds. Finally let 
$$\mathcal{X}^{\ast}(P, \textbf{g}, \textbf{x}) =  \mathcal{X}_{\text{bad}}(P, \textbf{g}, \textbf{x}) \cup \mathcal{X}_{H}(P, \textbf{g}, \textbf{x}). $$
We write $$N(P, \textbf{g}, \textbf{x}) = \sum_{\textbf{X} \in \mathcal{X}^{\ast}(P, \textbf{g}, \textbf{x}) } N_{\xi}(\textbf{X}).$$
Both $\mathcal{X}^{\ast}(P, \textbf{g}, \textbf{x})$ and $N(P, \textbf{g}, \textbf{x})$ depend on $H$, but this is not displayed in the notation. Our main result in this section is
\begin{prop}\label{prop-hyp} For $1 \ll H \leq P$ and $0 < \lambda < 1$ we have
\begin{equation}\label{prop7}
N(P, \textbf{g}, \textbf{x}) \ll \frac{P(1 + \log H)(\log P)^{\varepsilon}}{(\xi^{3/2} \eta^{2})^{99/100}(abcywz)^{1/4} }. 
\end{equation}
\end{prop}

\begin{proof} We first claim that 
\begin{equation}\label{claim1}
\sum_{\textbf{X} \in \mathcal{X}_{\text{bad}}(P, \textbf{g}, \textbf{x}) } N_{\xi}(\textbf{X}) \ll P^{1- \lambda/100+2\varepsilon}.
\end{equation}
 Let $\delta = \lambda/100$ and  suppose that 
\begin{equation}\label{suppose}
N_{\xi}(\textbf{X}) \geq P^{1-\delta + \varepsilon}.
\end{equation}
 We show that this implies that \eqref{good} holds, so $\textbf{X}$ is good, and since there are only $O(P^{\varepsilon})$ tuples $\textbf{X}$, this implies the claim. 

From Proposition \ref{elem}  and \eqref{height} we have 
\begin{displaymath}
\begin{split}
N_{\xi}(\textbf{X}) & \ll P^{\varepsilon} (ABC)^{1/2}(YW)^{3/4}Z^{1/2}\\
& \leq  P^{\varepsilon}\max(A, B, C)^{3/2} \max(Y, W)^{3/4} \cdot \max(Y, W)^{3/4} Z^{1/2} \ll 
\frac{P^{1+\varepsilon}}{\eta^2\xi^{3/2}}. 
\end{split}
\end{displaymath}
Contemplating this sequence of inequalities, we conclude from \eqref{suppose} that
\begin{equation}\label{cond1a}
\begin{split}
& \xi \ll P^{5\delta},  \quad  \min(Y, W) \gg \max(Y, W) P^{-5\delta},\\
& \min(A, B, C) \gg \max(A, B, C)P^{-5\delta},\\
& C^2Y  \gg P^{1-5\delta},   \quad Y^3Z^2 \gg  P^{1-5\delta},
\end{split}
\end{equation}
which also implies that all of 
the  three blocks $AB$, $C^2$, $YWZ^2$ must be within $B^{20\delta}$. This  shows our claim that $\textbf{X}$ must satisfy \eqref{good} with $\lambda  = \delta/100$. This establishes \eqref{claim1}. We complement this with a second bound. From Proposition \ref{elem} and \eqref{height} we have \begin{displaymath}
\begin{split}
N_{\xi}(\textbf{X})& \ll P^{\varepsilon}   (ABC)^{1/2}Y^{3/4} \cdot W^{3/4} Z^{1/2}   \ll P^{\varepsilon}   
 \frac{P}{\xi^{3/2}\eta^2 (abcz)^{1/2} (yw)^{3/4}}.  
\end{split}
\end{displaymath}
Combining this with \eqref{claim1}, we conclude
$$\sum_{\textbf{X} \in \mathcal{X}_{\text{bad}}(P, \textbf{g}, \textbf{x}) } N_{\xi}(\textbf{X})  \ll P \min\Big( P^{-\lambda/200}, \frac{P^{\varepsilon}}{\xi^{3/2}\eta^2 (abcywz)^{1/2}}\Big) \ll \frac{P}{(\xi^{3/2} \eta^{2}(abcywz)^{1/2})^{99/100}}.$$
This is acceptable for   \eqref{prop7}.  

For the contribution of $\textbf{X} \in  \mathcal{X}_{H}(P, \textbf{g}, \textbf{x}) \setminus  \mathcal{X}_{\text{bad}}(P, \textbf{g}, \textbf{x}) $ we observe that Proposition \ref{prop2} is available. We note that 
$$\textbf{X}^{(\frac{3}{4}, 0, \frac{1}{4})} +  \textbf{X}^{(\frac{1}{2}, \frac{1}{4}, \frac{1}{4})} = (AB)^{1/4} C (YW)^{3/4} Z^{1/2} + (ABC)^{1/2}  (YW)^{3/4} Z^{1/2} \ll \frac{P}{\eta^2 \xi^{3/2} (abcywz)^{1/4}}$$
upon using \eqref{height}. By \eqref{atmost}, this is acceptable for the very good tuples. For tuples that are good, but not very good, the previous inequality along with the same argument as leading to \eqref{cond1a}  shows 
$$N_{\xi}(\textbf{X}) \ll \frac{P (\log P)^{-10}}{\eta^2 \xi^{3/2} (abcywz)^{1/4}}$$
for such $\textbf{X}$. Since there are at most $O((\log P)^6)$ such tuples, the proof is complete. \end{proof}

\part{Proof of Theorem \ref{thm1}}

\section{Geometry}\label{sec:geometry}

Table~\ref{tab:classification_spherical} contains the
nine smooth spherical Fano threefolds over $\Qbar$ that are not
horospherical (since the horospherical smooth Fano threefolds are all
either toric or flag varieties; see \cite[\S 6.3]{hofscheier}). The
notation $T$ and $N$ in \cite[Table~6.5]{hofscheier} and in our
Table~\ref{tab:classification_spherical} refers to the cases described
at the beginning of the introduction (Section~\ref{sec:intro}) and in
\cite[\S 10.2]{BBDG}.

\begin{table}[ht]
  \centering
  \begin{tabular}[ht]{ccccccc}
    \hline
    rk Pic & Hofscheier & Mori--Mukai & torsor equation & remark\\
    \hline\hline
    2 & $T_1 12$ & II.31 &  {$x_{11}x_{12}-x_{21}x_{22}-x_{31}x_{32}^2$} & eq. $\Gd_\mathrm{a}^3$-cpct.\\
    2 & $N_1 6, N_1 7$ & II.30 & {$x_{11}x_{12}-x_{21}^2-x_{31}x_{32}$} & eq. $\Gd_\mathrm{a}^3$-cpct.\\
    2 & $N_1 8$ & II.29 &  {$x_{11}x_{12}-x_{21}^2-x_{31}x_{32}x_{33}^2$} & variety $X_1$\\
    \hline
    3 & $T_1 18$ & III.24 & {$x_{11}x_{12}-x_{21}x_{22}-x_{31}x_{32}$} & \cite{BBDG} \\
    3 & $T_1 21$ & III.20  &
                             {$x_{11}x_{12}-x_{21}x_{22}-x_{31}x_{32}x_{33}^2$} &  \cite{BBDG} \\
    3 & $N_0 3$ & III.22  &  $x_{11}x_{12}-x_{21}^2-x_{31}x_{32}$ & variety $X_2$ \\
    3 & $N_1 9$ & III.19 &  $x_{11}x_{12}-x_{21}^2-x_{31}x_{32}$ & variety $X_3$\\
    \hline
    4 & $T_0 3$ &  IV.8  &  $x_{11}x_{12}-x_{21}x_{22}-x_{31}x_{32}$ & \cite{BBDG} \\
    4 & $T_1 22$ & IV.7  & $x_{11}x_{12}-x_{21}x_{22}-x_{31}x_{32}$ & \cite{BBDG} \\
    \hline
  \end{tabular}
  \caption{Smooth Fano threefolds that are spherical, but not horospherical}
  \label{tab:classification_spherical}
\end{table}

We proceed to describe the three $N$ cases $X_1, X_2, X_3$ in
Table~\ref{tab:classification_spherical} that are not equivariant
$\Gd_\mathrm{a}^3$-compactifications \cite{arXiv:1802.08090} in more
detail. From the description in the Mori--Mukai classification, we can
construct a split form over $\Qd$ in each case. We then recall from
Hofscheier's list the description using the Luna--Vust theory of spherical
embeddings.

The three varieties will be equipped with an action of
$G=\SL_2\times \Gd_\mathrm{m}$. Let $\varepsilon_1 \in \Xf(B)$ always
be a primitive character of $\Gd_\mathrm{m}$ composed with the natural
inclusion $\Xf(\Gd_\mathrm{m}) \to \Xf(B)$.

\subsection{$X_1$ of type II.29}\label{sec:geometry_X1}

Consider $\Pd^4_{\Qd}$ with coordinates $(z_{11} : z_{12} : z_{21} : z_{31} : z_{32})$
and the hypersurface $Q = \Vd(z_{11}z_{12} - z_{21}^2 - z_{31}z_{32}) \subset \Pd^4_{\Qd}$.
It contains the conic
\begin{align*}
  C_{33} = \Vd(z_{31}, z_{32}).
\end{align*}
Let $X_1$ be the blow-up of $Q$ in $C_{33}$. This is a smooth Fano threefold of type II.29.
 We may define an
  action of $G = \SL_2 \times \Gd_\mathrm{m}$ on $Q$ by
 \begin{equation*}
     (A,t)\cdot \rleft(
      \begin{pmatrix}
        z_{11} & z_{21} \\
        z_{21} & z_{12}
      \end{pmatrix}, z_{31}, z_{32}
      \right) = \left(A\cdot \begin{pmatrix}
        z_{11} & z_{21} \\
        z_{21} & z_{12}
      \end{pmatrix}\cdot A^\top, t\cdot z_{31}, t^{-1}\cdot z_{32}\rright)\text{,}
  \end{equation*}
 which turns $Q$ into a spherical variety.
The following description using
the Luna--Vust theory of spherical embeddings can be easily verified.  The
lattice $\Mm$ has basis
$(\alpha + \varepsilon_1, \alpha - \varepsilon_1)$. We
denote the corresponding dual basis of the lattice $\Nm$ by $(d_1, d_2)$. Then
there is one color with valuation $d = d_1 + d_2$, and the valuation cone
is given by $\Vm = \{v \in \Nm_\Qd : \langle v, \alpha \rangle \le 0\}$. Since $C_{32}$ is $G$-invariant, the variety $X_1$ is a spherical
  $G$-variety and the blow-up morphism
$X_1 \to Q$ can be described by a map of colored fans. The following
figure illustrates this.
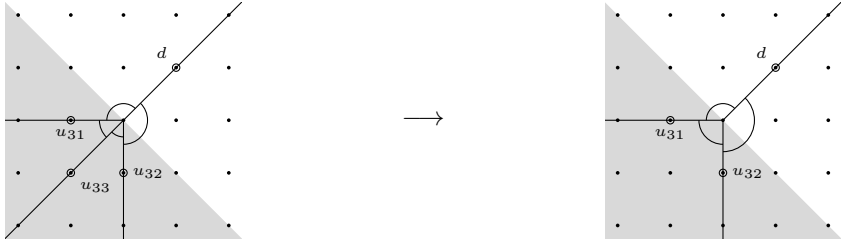
\begin{align*}
  \begin{tikzpicture}[scale=0.7]
    \clip (-2.24, -2.24) -- (2.24, -2.24) -- (2.24, 2.24) -- (-2.24, 2.24) -- cycle;
    \fill[color=gray!30] (-3, 3) -- (3, -3) -- (-3, -3) -- cycle;
    \foreach \x in {-3,...,3} \foreach \y in {-3,...,3} \fill (\x, \y) circle (1pt);
    \draw (1, 1) circle (2pt);
    \draw (-1, 0) circle (2pt);
    \draw (0, -1) circle (2pt);
    \draw (-1, -1) circle (2pt);
    \draw (0,0) -- (3,3);
    \draw (0,0) -- (-3, 0);
    \draw (0,0) -- (0, -3);
    \draw (0,0) -- (-3,-3);
    \path (-1, 0) node[anchor=north] {{\tiny{$u_{31}$}}};
    \path (0, -1) node[anchor=west] {{\tiny{$u_{32}$}}};
    \path (-1, -1) node[anchor=north west] {{\tiny{$u_{33}$}}};
    \path (1, 1) node[anchor=south east] {{\tiny{$d$}}};
    \begin{scope}
      \clip (0,0) -- (1,1) -- (-1,0) -- cycle; \draw (0,0) circle (9pt);
    \end{scope}
    \begin{scope}
      \clip (0,0) -- (-1, 0) -- (-1, -1) -- cycle; \draw (0,0) circle (13pt);
    \end{scope}
    \begin{scope}
      \clip (0,0) -- (-1, -1) -- (0, -1) -- cycle; \draw (0,0) circle (9pt);
    \end{scope}
    \begin{scope}
      \clip (0,0) -- (0,-2) -- (2,2) -- cycle; \draw (0,0) circle (13pt);
    \end{scope}
  \end{tikzpicture}
  &&
  \begin{tikzpicture}[scale=0.7]
    \clip (-0.5, -2.24) -- (0.5, -2.24) -- (0.5, 2.24) -- (-0.5, 2.24) -- cycle;
    \path (0,0) node {{$\longrightarrow$}};
  \end{tikzpicture}
  &&
  \begin{tikzpicture}[scale=0.7]
    \clip (-2.24, -2.24) -- (2.24, -2.24) -- (2.24, 2.24) -- (-2.24, 2.24) -- cycle;
    \fill[color=gray!30] (-3, 3) -- (3, -3) -- (-3, -3) -- cycle;
    \foreach \x in {-3,...,3} \foreach \y in {-3,...,3} \fill (\x, \y) circle (1pt);
    \draw (1, 1) circle (2pt);
    \draw (-1, 0) circle (2pt);
    \draw (0, -1) circle (2pt);
    \draw (0,0) -- (3,3);
    \draw (0,0) -- (-3, 0);
    \draw (0,0) -- (0, -3);
    \path (-1, 0) node[anchor=north] {{\tiny{$u_{31}$}}};
    \path (0, -1) node[anchor=west] {{\tiny{$u_{32}$}}};
    \path (1, 1) node[anchor=south east] {{\tiny{$d$}}};
    \begin{scope}
      \clip (0,0) -- (1,1) -- (-1,0) -- cycle; \draw (0,0) circle (9pt);
    \end{scope}
    \begin{scope}
      \clip (0,0) -- (-1, 0) -- (0, -1) -- cycle; \draw (0,0) circle (13pt);
    \end{scope}
    \begin{scope}
      \clip (0,0) -- (0,-2) -- (2,2) -- cycle; \draw (0,0) circle (17pt);
    \end{scope}
  \end{tikzpicture}
\end{align*}

Here, $u_{31} = -d_1$ and $u_{32} = -d_2$ are the
valuations of the $G$-invariant prime divisors $\Vd(z_{31})$ and
$\Vd(z_{32})$, respectively, and $u_{33} = -d$
is the valuation of the exceptional divisor
$E_{33}$ over $C_{33}$.

We obtain a projective ambient toric variety $Y_1$. From the
description of $\Sigmamax$ in \cite[\S 10.3]{BBDG}, we deduce that
$Y_1$ is smooth and that $-K_{X_1}$ is ample on $Y_1$. Hence
assumption~\cite[(2.3)]{BBDG} holds, and we work with
$Y_1^\circ = Y_1' = Y_1'' = Y_1$.

Now consider $\Pd^5_{\Qd}$ with an additional variable $q$, and let
$Q' = \Vd(z_{11}z_{12} - z_{21}^2 - z_{31}z_{32}, q^2 - z_{31}z_{32})
\subset \Pd^5_{\Qd}$. The covering map $Q' \to Q$ given by forgetting
$q$ induces a covering map of blow-ups $X'_1 \to X_1$. The image of
the last map is the set
\begin{equation*}
  \{(x_{11}:\dots:x_{33}) \in X_1(\Qd) \mid x_{31}x_{32} = -\square\},
\end{equation*}
which is therefore thin; in particular the set $T_1$ from the
introduction is also thin.

\subsection{$X_2$ of type III.22}

   Let $W=\Pd^1_\Qd \times \Pd^2_\Qd$ with
  coordinates $(z_{01}:z_{02})$ and $(z_{11}:z_{12}:z_{21})$.
    Let $C_{32}$ be the curve $\Vd(z_{02},z_{11}z_{12}-z_{21}^2)$
  on $W$. Let $X_2$ be the blow-up of $W$ in
  $C_{32}$. This is a smooth Fano threefold of type III.22.
  We may define an action of $G = \SL_2 \times \Gd_\mathrm{m}$ on $W$ by
  \begin{equation*}
     (A,t)\cdot \rleft(z_{01},z_{02},
      \begin{pmatrix}
        z_{11} & z_{21} \\
        z_{21} & z_{12}
      \end{pmatrix}
      \right) = \left(t\cdot z_{01},z_{02},A\cdot \begin{pmatrix}
        z_{11} & z_{21} \\
        z_{21} & z_{12}
      \end{pmatrix}\cdot A^\top\rright)\text{,}
  \end{equation*}
  which turns $W$ into a spherical variety with the following
  Luna--Vust description.
  The lattice $\Mm$ has basis $(2\alpha, \varepsilon)$. We denote the
  corresponding dual basis of the lattice $\Nm$ by $(d,
  \varepsilon^*)$. Then there is one color with valuation $2d =
  \frac{1}{2}\alpha^\vee$, and the valuation cone is given by $\Vm =
  \{v \in \Nm_\Qd : \langle v, \alpha \rangle \le 0\}$. Since the
  curve $C_{32}$ is $G$-invariant, the variety $X_2$ is a spherical
  $G$-variety, and the blow-up morphism $X_2 \to W$ can be
  described by a map of colored fans.
  The right-hand arrow in the following figure illustrates this.
  
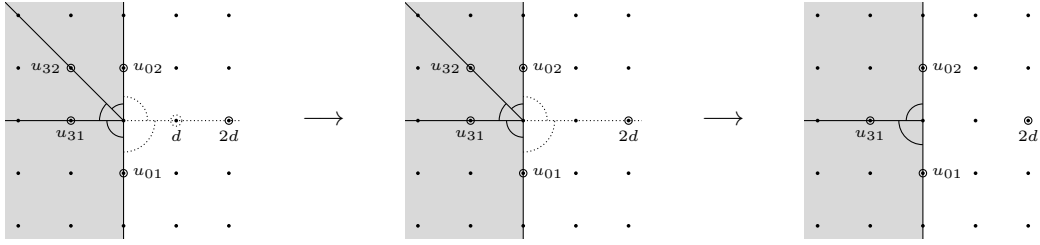
\begin{align*}
  \begin{tikzpicture}[scale=0.7]
    \clip (-2.24, -2.24) -- (2.24, -2.24) -- (2.24, 2.24) -- (-2.24, 2.24) -- cycle;
    \fill[color=gray!30] (0, 3) -- (0, -3) -- (-3, -3) -- (-3, 3) -- cycle;
    \foreach \x in {-3,...,3} \foreach \y in {-3,...,3} \fill (\x, \y) circle (1pt);
    \draw (2, 0) circle (2pt);
    \draw[densely dotted] (1, 0) circle (3pt);
    \draw (0, 1) circle (2pt);
    \draw (-1, 0) circle (2pt);
    \draw (0, -1) circle (2pt);
    \draw (-1, 1) circle (2pt);
    \draw[densely dotted] (0,0) -- (3,0);
    \draw (0,0) -- (0,3);
    \draw (0,0) -- (-4,4);
    \draw (0,0) -- (-3, 0);
    \draw (0,0) -- (0, -3);
    \path (-1, 0) node[anchor=north] {{\tiny{$u_{31}$}}};
    \path (0, -1) node[anchor=west] {{\tiny{$u_{01}$}}};
    \path (0, 1) node[anchor=west] {{\tiny{$u_{02}$}}};
    \path (2, 0) node[anchor=north] {{\tiny{$2d$}}};
    \path (1, 0) node[anchor=north] {{\tiny{$d$}}};
    \path (-1, 1) node[anchor=east] {{\tiny{$u_{32}$}}};
    \begin{scope}
      \clip (0,0) -- (0,1) -- (-1,1) -- cycle; \draw (0,0) circle (9pt);
    \end{scope}
    \begin{scope}
      \clip (0,0) -- (-1,1) -- (-1,0) -- cycle; \draw (0,0) circle (13pt);
    \end{scope}
    \begin{scope}
      \clip (0,0) -- (-1, 0) -- (0, -1) -- cycle; \draw (0,0) circle (9pt);
    \end{scope}
    \begin{scope}
      \clip (0,0) -- (1, 0) -- (0, 1) -- cycle; \draw[densely dotted] (0,0) circle (13pt);
    \end{scope}
    \begin{scope}
      \clip (0,0) -- (1, 0) -- (0, -1) -- cycle; \draw[densely dotted] (0,0) circle (17pt);
    \end{scope}
  \end{tikzpicture}
  &&
  \begin{tikzpicture}[scale=0.7]
    \clip (-0.5, -2.24) -- (0.5, -2.24) -- (0.5, 2.24) -- (-0.5, 2.24) -- cycle;
    \path (0,0) node {{$\longrightarrow$}};
  \end{tikzpicture}
  &&
  \begin{tikzpicture}[scale=0.7]
    \clip (-2.24, -2.24) -- (2.24, -2.24) -- (2.24, 2.24) -- (-2.24, 2.24) -- cycle;
    \fill[color=gray!30] (0, 3) -- (0, -3) -- (-3, -3) -- (-3, 3) -- cycle;
    \foreach \x in {-3,...,3} \foreach \y in {-3,...,3} \fill (\x, \y) circle (1pt);
    \draw (2, 0) circle (2pt);
    \draw (0, 1) circle (2pt);
    \draw (-1, 0) circle (2pt);
    \draw (0, -1) circle (2pt);
    \draw (-1, 1) circle (2pt);
    \draw[densely dotted] (0,0) -- (3,0);
    \draw (0,0) -- (0,3);
    \draw (0,0) -- (-4,4);
    \draw (0,0) -- (-3, 0);
    \draw (0,0) -- (0, -3);
    \path (-1, 0) node[anchor=north] {{\tiny{$u_{31}$}}};
    \path (0, -1) node[anchor=west] {{\tiny{$u_{01}$}}};
    \path (0, 1) node[anchor=west] {{\tiny{$u_{02}$}}};
    \path (2, 0) node[anchor=north] {{\tiny{$2d$}}};
    \path (-1, 1) node[anchor=east] {{\tiny{$u_{32}$}}};
    \begin{scope}
      \clip (0,0) -- (0,1) -- (-1,1) -- cycle; \draw (0,0) circle (9pt);
    \end{scope}
    \begin{scope}
      \clip (0,0) -- (-1,1) -- (-1,0) -- cycle; \draw (0,0) circle (13pt);
    \end{scope}
    \begin{scope}
      \clip (0,0) -- (-1, 0) -- (0, -1) -- cycle; \draw (0,0) circle (9pt);
    \end{scope}
    \begin{scope}
      \clip (0,0) -- (1, 0) -- (0, 1) -- cycle; \draw[densely dotted] (0,0) circle (13pt);
    \end{scope}
    \begin{scope}
      \clip (0,0) -- (1, 0) -- (0, -1) -- cycle; \draw[densely dotted] (0,0) circle (17pt);
    \end{scope}
  \end{tikzpicture}
  &&
  \begin{tikzpicture}[scale=0.7]
    \clip (-0.5, -2.24) -- (0.5, -2.24) -- (0.5, 2.24) -- (-0.5, 2.24) -- cycle;
    \path (0,0) node {{$\longrightarrow$}};
  \end{tikzpicture}
  &&
    \begin{tikzpicture}[scale=0.7]
    \clip (-2.24, -2.24) -- (2.24, -2.24) -- (2.24, 2.24) -- (-2.24, 2.24) -- cycle;
    \fill[color=gray!30] (0, 3) -- (0, -3) -- (-3, -3) -- (-3, 3) -- cycle;
    \foreach \x in {-3,...,3} \foreach \y in {-3,...,3} \fill (\x, \y) circle (1pt);
    \draw (2, 0) circle (2pt);
    \draw (0, 1) circle (2pt);
    \draw (-1, 0) circle (2pt);
    \draw (0, -1) circle (2pt);
    \draw (0,0) -- (0,3);
    \draw (0,0) -- (-3, 0);
    \draw (0,0) -- (0, -3);
    \path (-1, 0) node[anchor=north] {{\tiny{$u_{31}$}}};
    \path (0, -1) node[anchor=west] {{\tiny{$u_{01}$}}};
    \path (0, 1) node[anchor=west] {{\tiny{$u_{02}$}}};
    \path (2, 0) node[anchor=north] {{\tiny{$2d$}}};
    \begin{scope}
      \clip (0,0) -- (0,1) -- (-1,0) -- cycle; \draw (0,0) circle (9pt);
    \end{scope}
    \begin{scope}
      \clip (0,0) -- (-1, 0) -- (0, -1) -- cycle; \draw (0,0) circle (13pt);
    \end{scope}
  \end{tikzpicture}
\end{align*}
Here, $u_{01} = -\varepsilon^*$, $u_{02} = \varepsilon^*$, $u_{31} = -d$ are
the valuations of $\Vd(z_{01})$, $\Vd(z_{02})$, $\Vd(z_{11}z_{12}-z_{21}^2)$,
respectively, and $u_{32} = -d + \varepsilon^*$ is the valuation of the
exceptional divisor $E_{32}$ over $C_{32}$.

The dotted circles in the colored fan of $X_2$ specify
the standard small completion $Y_2$ of the ambient toric variety $Y_2^\circ$.
We note that $Y_2$ is singular and that it is not
possible to obtain a smooth small completion with the
construction from \cite[\S 10.3]{BBDG} in this case. We may, however,
construct a resolution of singularities $Y_2'' \to Y_2$ which
does not affect $X_2$. This is illustrated by the left-hand arrow
in the figure above.

\subsection{$X_3$ of type III.19}

Consider $\Pd^4_{\Qd}$ with coordinates $(z_{11} : z_{12} : z_{21} : z_{31} : z_{32})$
and the hypersurface $Q = \Vd(z_{11}z_{12} - z_{21}^2 - z_{31}z_{32}) \subset \Pd^4_{\Qd}$.
It contains the points
\begin{align*}
  P_{01} = \Vd(z_{11}, z_{12}, z_{21}, z_{31}), && P_{02} = \Vd(z_{11}, z_{12}, z_{21}, z_{32}).
\end{align*}
Let $X_3$ be the blow-up of $Q$ in $P_{01}$ and $P_{02}$. This is a smooth Fano threefold of type III.19.
Since $P_{01}$ and $P_{02}$ are $G$-invariant, the variety $X_3$ is a spherical
  $G$-variety and the blow-up morphism
$X_3 \to Q$ can be described by a map of colored fans. The right-hand arrow in the following
figure illustrates this.
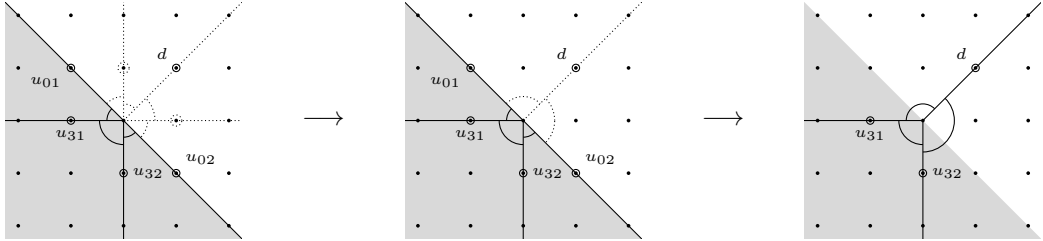
\begin{align*}
    \begin{tikzpicture}[scale=0.7]
    \clip (-2.24, -2.24) -- (2.24, -2.24) -- (2.24, 2.24) -- (-2.24, 2.24) -- cycle;
    \fill[color=gray!30] (-3, 3) -- (3, -3) -- (-3, -3) -- cycle;
    \foreach \x in {-3,...,3} \foreach \y in {-3,...,3} \fill (\x, \y) circle (1pt);
    \draw (1, 1) circle (2pt);
    \draw (-1, 0) circle (2pt);
    \draw (0, -1) circle (2pt);
    \draw (1, -1) circle (2pt);
    \draw (-1, 1) circle (2pt);
    \draw[densely dotted] (1, 0) circle (3pt);
    \draw[densely dotted] (0, 1) circle (3pt);
    \draw[densely dotted] (0,0) -- (3,3);
    \draw (0,0) -- (-3,3);
    \draw (0,0) -- (3,-3);
    \draw (0,0) -- (-3, 0);
    \draw (0,0) -- (0, -3);
    \draw[densely dotted] (0,0) -- (0, 3);
    \draw[densely dotted] (0,0) -- (3, 0);
    \path (-1, 1) node[anchor=north east] {{\tiny{$u_{01}$}}};
    \path (1, -1) node[anchor=south west] {{\tiny{$u_{02}$}}};
    \path (-1, 0) node[anchor=north] {{\tiny{$u_{31}$}}};
    \path (0, -1) node[anchor=west] {{\tiny{$u_{32}$}}};
    \path (1, 1) node[anchor=south east] {{\tiny{$d$}}};
    \begin{scope}
      \clip (0,0) -- (-1,1) -- (-1,0) -- cycle; \draw (0,0) circle (9pt);
    \end{scope}
    \begin{scope}
      \clip (0,0) -- (-1, 0) -- (0, -1) -- cycle; \draw (0,0) circle (13pt);
    \end{scope}
    \begin{scope}
      \clip (0,0) -- (0,-2) -- (2,-2) -- cycle; \draw (0,0) circle (9pt);
    \end{scope}
    \begin{scope}
      \clip (0,0) -- (-2,2) -- (0,2) -- cycle; \draw[densely dotted] (0,0) circle (13pt);
    \end{scope}
    \begin{scope}
      \clip (0,0) -- (0,2) -- (2,2) -- cycle; \draw[densely dotted] (0,0) circle (9pt);
    \end{scope}
    \begin{scope}
      \clip (0,0) -- (2,2) -- (2,0) -- cycle; \draw[densely dotted] (0,0) circle (17pt);
    \end{scope}
    \begin{scope}
      \clip (0,0) -- (2,0) -- (2,-2) -- cycle; \draw[densely dotted] (0,0) circle (13pt);
    \end{scope}
  \end{tikzpicture}
      &&
  \begin{tikzpicture}[scale=0.7]
    \clip (-0.5, -2.24) -- (0.5, -2.24) -- (0.5, 2.24) -- (-0.5, 2.24) -- cycle;
    \path (0,0) node {{$\longrightarrow$}};
  \end{tikzpicture}
  &&
  \begin{tikzpicture}[scale=0.7]
    \clip (-2.24, -2.24) -- (2.24, -2.24) -- (2.24, 2.24) -- (-2.24, 2.24) -- cycle;
    \fill[color=gray!30] (-3, 3) -- (3, -3) -- (-3, -3) -- cycle;
    \foreach \x in {-3,...,3} \foreach \y in {-3,...,3} \fill (\x, \y) circle (1pt);
    \draw (1, 1) circle (2pt);
    \draw (-1, 0) circle (2pt);
    \draw (0, -1) circle (2pt);
    \draw (1, -1) circle (2pt);
    \draw (-1, 1) circle (2pt);
    \draw[densely dotted] (0,0) -- (3,3);
    \draw (0,0) -- (-3,3);
    \draw (0,0) -- (3,-3);
    \draw (0,0) -- (-3, 0);
    \draw (0,0) -- (0, -3);
    \path (-1, 1) node[anchor=north east] {{\tiny{$u_{01}$}}};
    \path (1, -1) node[anchor=south west] {{\tiny{$u_{02}$}}};
    \path (-1, 0) node[anchor=north] {{\tiny{$u_{31}$}}};
    \path (0, -1) node[anchor=west] {{\tiny{$u_{32}$}}};
    \path (1, 1) node[anchor=south east] {{\tiny{$d$}}};
    \begin{scope}
      \clip (0,0) -- (-1,1) -- (-1,0) -- cycle; \draw (0,0) circle (9pt);
    \end{scope}
    \begin{scope}
      \clip (0,0) -- (-1, 0) -- (0, -1) -- cycle; \draw (0,0) circle (13pt);
    \end{scope}
    \begin{scope}
      \clip (0,0) -- (0,-2) -- (2,-2) -- cycle; \draw (0,0) circle (9pt);
    \end{scope}
    \begin{scope}
      \clip (0,0) -- (-2,2) -- (2,2) -- cycle; \draw[densely dotted] (0,0) circle (13pt);
    \end{scope}
    \begin{scope}
      \clip (0,0) -- (2,-2) -- (2,2) -- cycle; \draw[densely dotted] (0,0) circle (17pt);
    \end{scope}
  \end{tikzpicture}
      &&
  \begin{tikzpicture}[scale=0.7]
    \clip (-0.5, -2.24) -- (0.5, -2.24) -- (0.5, 2.24) -- (-0.5, 2.24) -- cycle;
    \path (0,0) node {{$\longrightarrow$}};
  \end{tikzpicture}
  &&
  \begin{tikzpicture}[scale=0.7]
    \clip (-2.24, -2.24) -- (2.24, -2.24) -- (2.24, 2.24) -- (-2.24, 2.24) -- cycle;
    \fill[color=gray!30] (-3, 3) -- (3, -3) -- (-3, -3) -- cycle;
    \foreach \x in {-3,...,3} \foreach \y in {-3,...,3} \fill (\x, \y) circle (1pt);
    \draw (1, 1) circle (2pt);
    \draw (-1, 0) circle (2pt);
    \draw (0, -1) circle (2pt);
    \draw (0,0) -- (3,3);
    \draw (0,0) -- (-3, 0);
    \draw (0,0) -- (0, -3);
    \path (-1, 0) node[anchor=north] {{\tiny{$u_{31}$}}};
    \path (0, -1) node[anchor=west] {{\tiny{$u_{32}$}}};
    \path (1, 1) node[anchor=south east] {{\tiny{$d$}}};
    \begin{scope}
      \clip (0,0) -- (1,1) -- (-1,0) -- cycle; \draw (0,0) circle (9pt);
    \end{scope}
    \begin{scope}
      \clip (0,0) -- (-1, 0) -- (0, -1) -- cycle; \draw (0,0) circle (13pt);
    \end{scope}
    \begin{scope}
      \clip (0,0) -- (0,-2) -- (2,2) -- cycle; \draw (0,0) circle (17pt);
    \end{scope}
  \end{tikzpicture}
\end{align*}

Here, $u_{31} = -d_1$ and $u_{32} = -d_2$ are the valuations of $\Vd(z_{31})$
and $\Vd(z_{32})$, respectively, and $u_{01}$ and $u_{02}$ are the valuations
of the exceptional divisors $E_{01}$ over $P_{01}$ and $E_{02}$ over $P_{02}$,
respectively.

The dotted circles in the colored fan of $X_3$ specify the singular standard
small completion $Y_3$ of the ambient toric variety $Y_3^\circ$. Again, it is
not possible to obtain a smooth small completion with the construction from
\cite[\S 10.3]{BBDG}. The left-hand arrow in the figure above describes a
resolution of singularities $Y_3'' \to Y_3$ which does not affect $X_3$.

\section{Cox rings and torsors}

\subsection{Type II.29}

A Cox ring for $X_1$ is given by
  \begin{align*}
  \Rm(X_1) = \Qd[x_{11},x_{12},x_{21},x_{31},x_{32},x_{33}]/(x_{11}x_{12}-x_{21}^2-x_{31}x_{32}x_{33}^2)
\end{align*}
with $\Pic X \cong \Z^2$ where
\begin{align*}
& \deg(x_{11}) = \deg(x_{12}) = \deg(x_{22}) = (1,0),\\
&\deg(x_{31}) = \deg(x_{32}) = (0,1),\quad \deg(x_{33}) = (1,-1).
\end{align*}
The anticanonical class is $-K_{X} = (2,1)$.
A universal torsor over $X_1$ is
\begin{equation*}
  \Tm_1 = \Spec\Rm(X_1) \setminus Z_{X_1}\text{,}
\end{equation*}
where
\begin{align*}
  Z_{X_1} = \Vd(x_{11},x_{12},x_{21},x_{33}) \cup \Vd(x_{31},x_{32})\text{.}
\end{align*}
  
\subsection{Type III.22}

A Cox ring for $X_2$ is given by
\begin{align*}
  \Rm(X_2) = \Qd[x_{01},x_{02},x_{11},x_{12},x_{21},x_{31},x_{32}]/(x_{11}x_{12}-x_{21}^2-x_{31}x_{32})
\end{align*}
with $\Pic X_2 \cong \Z^3$ where
\begin{align*}
 & \deg(x_{01})=(1,0,0), \quad  \deg(x_{02})=(1,0,-1)\text{,}\\
 & \deg(x_{11})=\deg(x_{12})=\deg(x_{21})=(0,1,0)\text{,}\\
 & \deg(x_{31})=(0,2,-1), \quad \deg(x_{32})=(0,0,1)\text{.}
\end{align*}
The anticanonical class is $-K_{X_2}=(2,3,-1)$. A universal torsor over $X_2$ is
\begin{equation*}
  \Tm_2 = \Spec\Rm(X_2) \setminus Z_{X_2}\text{,}
\end{equation*}
where
\begin{align*}
  Z_{X_2} &= \Vd(x_{11},x_{12},x_{21}) \cup \Vd(x_{01},x_{32}) \cup
  \Vd(x_{02},x_{31}) \cup \Vd(x_{01},x_{02})\text{.}
\end{align*}

Let $Y_2^\circ$ be the ambient toric variety of $X_2$; its rays are
\begin{equation*}
  (0, 0, 1, -1),\ 
  (0, 0, -1, 1),\ 
  (-1, -1, -2, 0),\ 
  (1, 0, 0, 0),\ 
  (0, 1, 0, 0),\ 
  (0, 0, 1, 0),\ 
  (0, 0, 0, 1),
\end{equation*}
corresponding to $x_{01},\dots,x_{32}$, respectively, and each of its nine
maximal cones is generated by the four rays corresponding to two of
$x_{11},x_{12},x_{21}$ and one of the pairs $x_{02},x_{32}$ or $x_{31},x_{32}$
or $x_{01},x_{31}$ from the maximal cones of the spherical fan. We have
\begin{equation*}
  Z_{Y_2^\circ}=\Vd(x_{11},x_{12},x_{21}) \cup \Vd(x_{01},x_{32}) \cup
  \Vd(x_{02},x_{31}) \cup \Vd(x_{01},x_{02})\text{.}
\end{equation*}

Let $Y_2$ be the standard small completion of $Y_2^\circ$. It has the same
rays as $Y_2^\circ$, and the two additional maximal cones generated by the
four rays corresponding to $x_{11},x_{12},x_{21}$ and $x_{01}$ or $x_{02}$
(see the spherical fan), which are both singular; the singular locus of
$Y_2^\circ$ corresponds to their intersection, the cone generated by
$(-1, -1, -2, 0),\ (1, 0, 0, 0),\ (0, 1, 0, 0)$. We have
\begin{equation*}
  Z_{Y_2} = \Vd(x_{11},x_{12},x_{21},x_{31}) \cup  \Vd(x_{11},x_{12},x_{21},x_{32}) \cup \Vd(x_{01}, x_{02}) \cup \Vd(x_{01}, x_{32}) \cup \Vd(x_{02}, x_{31}) \text{.}
\end{equation*}

A toric desingularization $\rho : Y_2'' \to Y_2$ is obtained by adding the ray
$(0, 0, -1, 0)$ (the primitive multiple of the sum of the rays of the singular
cone). Its Cox ring is
  \begin{equation*}
    \Rm(Y_2'')=\Qd[x_{01},\dots,x_{32},z_1]
  \end{equation*}
  with
  \begin{align*}
    & \deg(x_{01})=(0,0,1,1), \quad  \deg(x_{02})=(0,0,0,1)\text{,}\\
    & \deg(x_{11})=\deg(x_{12})=\deg(x_{21})=(0,1,0,0)\text{,}\\
    & \deg(x_{31})=(1,2,0,0), \quad \deg(x_{32})=(0,0,1,0), \quad \deg(z_1)=(1,0,1,0)
  \end{align*}
  and the irrelevant ideal
  \begin{equation*}
    Z_{Y_2''}=\Vd(x_{01},x_{02}) \cup \Vd(x_{01},x_{32}) \cup \Vd(x_{02},x_{31}) \cup \Vd(x_{31},z_1) \cup \Vd(x_{32},z_1)
    \cup \Vd(x_{11},x_{12},x_{21}).
  \end{equation*}

  Now let $Y_2' = Y_2''$. Then $X_2 \subset Y_2'$ is given in Cox coordinates
  by the homogeneous equation
  \begin{equation*}
    x_{11}x_{12}z_1-x_{21}^2z_1-x_{31}x_{32}=0,
  \end{equation*}
  and we have $\rho^*(-K_{X_2})=(\frac 3 2,3,\frac 5 2,2)$.

\subsection{Type III.19}

A Cox ring for $X_3$ is given by
\begin{align*}
  \Rm(X_3) = \Qd[x_{01}, x_{02}, x_{11}, x_{12}, x_{21}, x_{31}, x_{32}]/
  (x_{11}x_{12}-x_{21}^2-x_{31}x_{32})
\end{align*}
with $\Pic X_3 \cong \Z^3$ where
\begin{align*}
  &\deg(x_{01}) = (0,1,0), \quad \deg(x_{02}) = (0,0,1),\\
  &\deg(x_{11}) = \deg(x_{12}) = \deg(x_{21}) = (1,0,0),\\
  &\deg(x_{31}) = (1,-1,1), \quad \deg(x_{32}) = (1,1,-1).
\end{align*}
The anticanonical class is $-K_{X_3}=(3,1,1)$. A universal torsor over $X_3$ is
  \begin{equation*}
  \Tm_3 = \Spec\Rm(X_3) \setminus Z_{X_3}\text{,}
  \end{equation*}
where
\begin{align*}
  Z_{X_3} &= \Vd(x_{11},x_{12},x_{21}) \cup \Vd(x_{01}, x_{02}) \cup \Vd(x_{01}, x_{32}) \cup \Vd(x_{02}, x_{31}) \text{.}
\end{align*}
Its ambient toric variety $Y_3^\circ$ has the rays
\begin{equation*}
  (1, -1, 0, 0),\
  (-1, 1, 0, 0),\
  (-1, -1, -1, -1),\
  (0, 0, 0, 1),\
  (0, 0, 1, 0),\
  (1, 0, 0, 0),\
  (0, 1, 0, 0);
\end{equation*}
we have
\begin{equation*}
  Z_{Y_3^\circ} = \Vd(x_{11},x_{12},x_{21}) \cup \Vd(x_{01}, x_{02}) \cup \Vd(x_{01}, x_{32}) \cup \Vd(x_{02}, x_{31}) \text{.}
\end{equation*}

Its standard small completion $Y_3$ has two singular maximal cones
corresponding to $x_{0i},x_{11},x_{12},x_{21}$, but their intersection is
smooth, hence $Y_3$ has two isolated singularities; we have
\begin{equation*}
  Z_{Y_3} = \Vd(x_{11},x_{12},x_{21},x_{31}) \cup  \Vd(x_{11},x_{12},x_{21},x_{32}) \cup \Vd(x_{01}, x_{02}) \cup \Vd(x_{01}, x_{32}) \cup \Vd(x_{02}, x_{31}) \text{.}
\end{equation*}
Blowing up the singularities adds the
rays $(-1,0,0,0)$, $(0,-1,0,0)$ (half of the sum of the corresponding four
rays), giving $Y_3''$.  The Cox ring of $Y_3''$ has two extra generators
$z_1,z_2$, with degrees
\begin{align*}
  &\deg(x_{01}) = (0,0,0,0,1), \quad \deg(x_{02}) = (0,0,0,1,0),\\
  &\deg(x_{11}) = \deg(x_{12}) = \deg(x_{21}) = (0,0,1,0,0),\\
  &\deg(x_{31}) = (0,1,1,1,0), \quad \deg(x_{32}) = (1,0,1,0,1),\\
  &\deg(z_1)=(0,1,0,0,1), \quad \deg(z_2) = (1,0,0,1,0).
\end{align*}
We have
\begin{align*}
  Z_{Y_3''}={}
  &\Vd(z_2, z_1)\cup \Vd(z_2, x_{32}) \cup \Vd(z_1, x_{32}) \cup \Vd(z_2, x_{31})
    \cup \Vd(z_1, x_{31}) \cup \Vd(x_{32}, x_{21}, x_{12}, x_{11})\\
  & \cup \Vd(x_{31}, x_{21}, x_{12}, x_{11})
    \cup \Vd(z_2, x_{02}) \cup \Vd(x_{31}, x_{02})
    \cup \Vd(x_{21}, x_{12}, x_{11}, x_{02}) \\
  &  \cup \Vd(z_1, x_{01}) \cup \Vd(x_{32}, x_{01})
    \cup \Vd(x_{21}, x_{12}, x_{11}, x_{01}) \cup \Vd(x_{02}, x_{01}).
\end{align*}
Let $Y_3'=Y_3''$. Then
the equation for $X_3 \subset Y_3'$ is
\begin{equation*}
  x_{11}x_{12}z_1z_2-x_{21}^2z_1z_2-x_{31}x_{32}=0,
\end{equation*}
and we have $\rho^*(-K_{X_3}) = (2,2,3,3,3)$.

\section{Counting problems}\label{sec:counting_problems}

Applying the first part of this paper, we obtain the following counting problems, in
which $T_j$ is always the subset of $X_j(\Qd)$ where all Cox coordinates
are nonzero and, in case of $X_1$, where $x_{31}x_{33} \ne -\square$. For
simplicity, we write $N_j(B)$ for $N_{X_j(\Qd) \setminus T_j, H_j}(B)$ as
in the introduction, and we write $\{x, y\}$ to mean $x$ or $y$.

\begin{cor}\label{cor:counting_problems}
    {\rm (a)} We have
  \begin{align*}
    4 N_1(B) = \#\left\{\xv \in \Zd^6_{\ne 0} :
    \begin{aligned}
      &x_{11}x_{12}-x_{21}^2-x_{31}x_{32}x_{33}^3=0, \quad \max|\Pm_1(\xv)| \le B\\
      &\gcd(x_{11},x_{12},x_{21},x_{33})=\gcd(x_{31},x_{32})=1,\quad x_{31}x_{32}
      \ne -\square\\
    \end{aligned}
        \right\}\text{,}
  \end{align*}
  with
  \begin{equation*}
    \Pm_1(\xv) =\left\{
      \{x_{11},x_{12},x_{21}\}^2\{x_{31}, x_{32}\},
      \{x_{31}, x_{32}\}^3x_{33}^2
    \right\}\text{.}
  \end{equation*}
  \\
  {\rm (b)} We have
  \begin{align*}
    8 N_2(B) = \#\left\{\xv \in \Zd^7_{\ne 0} :
    \begin{aligned}
      &x_{11}x_{12}-x_{21}^2-x_{31}x_{32}=0, \quad \max|\Pm_2(\xv)| \le B\\
      &\gcd(x_{11},x_{12},x_{21})=\gcd(x_{01},x_{32})=\gcd(x_{02},x_{31})=\gcd(x_{01},x_{02})=1\\
    \end{aligned}
        \right\}\text{,}
  \end{align*}
  with
  \begin{equation*}
    \Pm_2(\xv) =\left\{
      \begin{aligned}
        &x_{01}^2\{x_{11},x_{12},x_{21}\}x_{31},\ 
        x_{02}^2\{x_{11},x_{12},x_{21}\}^3x_{32}, 
        \\
        &x_{01}x_{02}\{x_{11},x_{12},x_{21}\}^3,\ 
        x_{02}^2x_{31}^{3/2}x_{32}^{5/2},\ 
        x_{01}^2x_{31}^{3/2}x_{32}^{1/2}
      \end{aligned}
    \right\}\text{.}
  \end{equation*}
  \\
  {\rm (c)} We have
  \begin{align*}
    8 N_3(B) = \#\left\{\xv \in \Zd^7_{\ne 0} :
    \begin{aligned}
      &x_{11}x_{12}-x_{21}^2-x_{31}x_{32}=0, \quad \max|\Pm_3(\xv)| \le B\\
      &\gcd(x_{11},x_{12},x_{21})=\gcd(x_{01},x_{32})=\gcd(x_{02},x_{31})=\gcd(x_{01},x_{02})=1\\
    \end{aligned}
    \right\}\text{,}
  \end{align*}
  with
  \begin{equation*}
    \Pm_3(\xv) =\left\{
      \begin{aligned}
        &x_{01}^2\{x_{11},x_{12},x_{21}\}^2x_{31},\
        x_{02}^2\{x_{11},x_{12},x_{21}\}^2x_{32},\\
        &x_{01}x_{02}\{x_{11},x_{12},x_{21}\}^3,\
        x_{01}^2x_{31}^2x_{32},\
        x_{02}^2x_{31}x_{32}^2
      \end{aligned}
    \right\}\text{.}
  \end{equation*}
\end{cor}

\begin{proof}
  For $X_1$, we argue as in \cite{BBDG} since the ambient toric variety $Y_1$
  is regular.

  For $X_2$, we apply Proposition~\ref{prop:countingproblem_abstract}
  and obtain the counting problem
 \begin{align*}
    16 N_2(B) = \#\left\{(\xv,z_1) \in \Zd^8_{\ne 0} :
    \begin{aligned}
      &x_{11}x_{12}z_1-x_{21}^2z_1-x_{31}x_{32}=0, \quad \max|\Pm_2'(\xv,z_1)| \le B\\
      &\gcd(x_{11},x_{12},x_{21})=\gcd(x_{01},x_{32})=\gcd(x_{02},x_{31})=1\\
      &\gcd(x_{01},x_{02})=\gcd(x_{31},z_1)=\gcd(x_{32},z_1)=1\\
    \end{aligned}
    \right\}\text{}
  \end{align*}
  with
  \begin{equation*}
    \Pm_2'(\xv,z_1) =\left\{
      \begin{aligned}
        &x_{01}^2\{x_{11},x_{12},x_{21}\}x_{31}z_1^{1/2}, 
        x_{02}^2\{x_{11},x_{12},x_{21}\}^3x_{32}z_1^{3/2}, 
        x_{01}x_{02}\{x_{11},x_{12},x_{21}\}^3z_1^{3/2}, 
        \\
        &x_{02}^2 x_{31}^{3/2} x_{32}^{5/2},
        x_{01}^2 x_{31}^{3/2} x_{32}^{1/2}
      \end{aligned}
    \right\}\text{.}
  \end{equation*}
  But the equation together with $\gcd(x_{31}x_{32},z_1)=1$ implies
  $z_1= \pm 1$. After multiplying the equation with $z_1$, the substitution of
  $(x_{11}z_1, x_{12}z_1, x_{21}z_1, x_{31}z_1)$ by
  $(x_{11},x_{12},x_{21},x_{31})$ leads to our counting problem.

  For $X_3$, we similarly obtain
  \begin{align*}
    16 N_3(B) = \#\left\{(\xv,z_1,z_2) \in \Zd^9_{\ne 0} :
    \begin{aligned}
      &x_{11}x_{12}z_1z_2-x_{21}^2z_1z_2-x_{31}x_{32}=0, \quad \max|\Pm_3'(\xv,z_1)| \le B\\
      &\gcd(z_2, z_1)= \gcd(z_2, x_{32})=\gcd(z_1, x_{32})=\gcd(z_2,
      x_{31})=1\\  
      & \gcd(z_1, x_{31})=\gcd(x_{32}, x_{21}, x_{12}, x_{11})=1\\
      &\gcd(x_{31}, x_{21}, x_{12}, x_{11})=\gcd(z_2, x_{02})=\gcd(x_{31},
      x_{02})=1\\
      &\gcd(x_{21}, x_{12}, x_{11}, x_{02})=\gcd(z_1, x_{01})=\gcd(x_{32},
      x_{01})=1\\
      &\gcd(x_{21}, x_{12}, x_{11}, x_{01})=\gcd(x_{02}, x_{01})=1
    \end{aligned}
    \right\}\text{,}
  \end{align*}
  The height condition
  is given by the monomials
 \begin{equation*}
    \Pm_3'(\xv,z_1) =\left\{
      \begin{aligned}
        &x_{01}^2 \{x_{11},x_{12},x_{21}\}^2 x_{31} z_1 z_2^{2},\
        x_{02}^2 \{x_{11},x_{12},x_{21}\}^2 x_{32} z_1^{2} z_2 ,\\
        &x_{01} x_{02} \{x_{11},x_{12},x_{21}\}^3 z_1^{2} z_2^{2} ,\
        x_{02}^{2} x_{31} x_{32}^{2} z_1 ,\
        x_{01}^{2} x_{31}^{2} x_{32} z_2  \\
        &x_{01}^2 x_{31}^{2} x_{32} z_2 ,\
        x_{02}^2 x_{31} x_{32}^{2} z_1
      \end{aligned}
    \right\}\text{.}
  \end{equation*}
  In this counting problem, we observe that this equation
  together with $\gcd(z_1z_2,x_{31}x_{32})=1$ implies $z_1=\pm 1$ and $z_2=\pm
  1$. The torsor equation also allows us to simplify the coprimality
  conditions.
\end{proof}

\begin{remark}
  \label{rem:smfano2}
  The varieties $X_1$, $X_2$ and $X_3$ are as in
  Remark~\ref{rem:smfano1}, and in each case we have chosen $L$ as in
  the proof of Lemma~\ref{lemma:exL}. After having eliminated the
  additional variables in the proof of
  Corollary~\ref{cor:counting_problems}, we have obtained the same
  monomials $\Pm_1(\xv)$, $\Pm_2(\xv)$, $\Pm_3(\xv)$ as if we had
  directly applied \cite{BBDG} (disregarding that $Y$ is singular). In
  this case, \cite[Lemmas~3.8, 3.9, and 3.10]{BBDG} still apply, with
  the difference that the vertices of the polytopes are not
  necessarily integral. Moreover, \cite[Lemma~4.7]{BBDG} is also
  valid without the assumption that $Y$ is smooth.
\end{remark}

\section{Application: Proof of Theorem~\ref{thm1}}

\subsection{The analytic machinery}\label{sec:machinery}

Theorems 8.4, 9.2 and 10.1 in \cite{BBDG} provide an asymptotic formula for counting problems as in Corollary \ref{cor:counting_problems} under various assumptions and show in addition that the shape of the asymptotic formula  agrees with the Manin--Peyre prediction. We need a small variation of these results that we state in full detail for the reader's convenience.

Suppose that we are given a diophantine equation
 \begin{equation}\label{torsor}
 \sum_{i=1}^k  \prod_{j=1}^{J_i}   x_{ij}^{h_{ij}} = 0
 \end{equation}
 with certain $h_{ij} \in
\mathbb{N}$ 
  and height conditions 
  \begin{equation}\label{height1}
   \prod_{i=0}^{k} \prod_{j=1}^{J_i} |x_{ij}|^{\alpha^\nu_{ij}} \leq B \quad ( 1\le \nu \le M)
   \end{equation}
 for certain nonnegative exponents 
$\alpha^\nu_{ij}$ 
 whose variables are
  restricted by coprimality conditions 
  \begin{equation}\label{gcd}
  {\rm gcd}\{x_{ij}: (i,j)\in S_\rho\}  = 1 \quad (1\le \rho\le r)
  \end{equation}
for certain $S_{\rho} \subseteq \{(i,j) : i = 0, \ldots, k, j = 1, \ldots, J_i\}$, cf.\ \cite[(1.2) -- (1.4)]{BBDG}.  We write $J = J_0 + \ldots + J_k$. Let $N(B)$ denote the number of integral solutions to \eqref{torsor} with nonzero variables $\textbf{x}_j$ subject to \eqref{height1} and \eqref{gcd}.

 With these data, we define the following quantities. For $\textbf{g} \in \mathbb{N}^r$ write ${\bm \gamma} = (\gamma_{ij}) \in \mathbb{N}^J$, $\gamma_{ij} = \text{lcm}\{g_{\rho} \mid (i, j) \in S_{\rho}\}$ and $${\bm \gamma}^{\ast} = \Big( \prod_{j=1}^{J_i} \gamma_{ij}^{h_{ij}} \Big)_{1 \leq i \leq k} \in \mathbb{N}^k$$
 as in \cite[(8.11), (8.14)]{BBDG}.  
 
As in \cite[(5.1) -- (5.6)]{BBDG}, for $\textbf{b} \in \mathbb{N}^k$  define the (formal) singular series
$$  {\mathscr E}_{\mathbf b} = \sum_{q=1}^{\infty} \underset{a
    \bmod{q}}{\left.\sum \right.^{\ast}} \prod_{i=1}^k E_i(q,ab_i, \textbf{h}_i), \quad  E(q,a;{\mathbf h}) = q^{-n} \sum_{\substack{1\le x_j \le q\\ 1\le j\le n} }
  e\Big(\frac{ax_1^{h_1}x_2^{h_2}\cdots x_n^{h_n}}{q}\Big),
  $$
the (absolutely convergent) singular integral
$$ \mathscr{I}_{\mathbf{b}}(\mathbf{X}) = \langle\mathbf
  X_0\rangle\int_{-\infty}^\infty \prod_{i=1}^kI(b_i\beta,\mathbf X_i; \textbf{h}_i) \,\mathrm d\beta,\quad   I(\beta, {\mathbf Y};\mathbf h) = \int_{\mathscr Y} e(\beta
  y_1^{h_1}y_2^{h_2}\cdots y_n^{h_n})\,\mathrm d\mathbf y$$
 and the number   $\mathscr{N}_{\mathbf{b}}(\mathbf{X})$  of solutions
$\mathbf{x} \in \mathbb{Z}^J$ to \eqref{torsor} satisfying
$\frac{1}{2}X_{ij} \leq |x_{ij}| \leq X_{ij}$.

 As in \cite[(3.6), (7.1) -- (7.3)]{BBDG} define the block matrix
 \begin{equation*}
  \mathscr{A} = \left(
    \begin{matrix}
      \mathscr{A}_1 & \mathscr{A}_2\\
      \mathscr{A}_3 & \mathscr{A}_4
    \end{matrix}
  \right) \in \mathbb{R}^{(J+1) \times (M+k)}
\end{equation*}
where $  \mathscr{A}_1 = (\alpha_{ij}^{\nu}) \in \mathbb{R}_{\ge 0}^{J\times M}$ 
with $0 \leq i \leq k$, $1 \leq j \leq J_i$, $1 \leq \nu \leq M$, 
\begin{equation*}
  \mathscr{A}_2 = (e_{ij}^\mu) \in \mathbb{R}^{J\times k}
  \text{ with $e_{ij}^\mu=\begin{cases}\delta_{\mu=i}h_{ij}&\text{$i<k$, $\mu < k$,}
      \\-h_{kj}&\text{$i=k$, $\mu < k$,}\\
      -1&\text{$i<k$, $\mu = k$,}\\
      h_{kj}-1
      &\text{$i=k$, $\mu = k$,}
    \end{cases}$}
\end{equation*}
$\mathcal{A}_3 = (1, \ldots, 1) \in \mathbb{R}^{1 \times M}$ and $\mathcal{A}_4 = (0, \ldots, 0, -1) \in \mathbb{R}^{1\times k}$.  
Let $R = \text{rk}(\mathcal{A}_1)$ and $c_2 = J-R$ as in \cite[(7.5)]{BBDG}. Let $\mathscr{J}$ be a set of subsets of pairs $(i, j)$ with $0 \leq i \leq k$, $1 \leq j \leq J_i.$ 
For $H \geq 1$, some small fixed constant $0 < \lambda < 1$ and $\mathbf{b}, \mathbf{y} \in \mathbb{N}^{J}$  let
$N_{\mathbf{b}, \mathbf{y}}(B, H, \mathscr{J}
)$ be the number of solutions
$\mathbf{x} \in (\mathbb{Z}\setminus \{0\})^J$ satisfying the conditions
\begin{equation}\label{equation}
  \begin{split}
    &  \sum_{i=1}^k  \prod_{j=1}^{J_i}  (b_{ij} x_{ij})^{h_{ij}} = 0,\quad
    \prod_{i=0}^k \prod_{j=1}^{J_i} | y_{ij} x_{ij}|^{\alpha^\nu_{ij}} \leq B
    \quad (1 \leq \nu \leq N),
  \end{split}
\end{equation}
and at least one of the conditions
\begin{equation}\label{Hlambda}
  \begin{split}
    & \min_{ij} |x_{ij}| \leq H, \\
    & \min_{1 \leq i \leq k} \prod_{j =    1}^{J_i} |x_{ij}|^{h_{ij}} < \Bigl(\max_{1 \leq i \leq k} \prod_{j =   1}^{J_i} |2x_{ij}|^{h_{ij}}\Bigr)^{1-\lambda},\\
    & \min_{(ij) \in J}(|x_{ij}|) \leq \max_{(ij) \in J}(|2x_{ij}|) \max(|x_{ij}|)^{-\lambda}, \quad J 
    \in \mathscr{J}.
  \end{split}
\end{equation}
  Let $\mathscr{S}_{\textbf{y}}(B, H, \mathscr{J}
  )$ denote the set
of all $\mathbf{x} \in [1, \infty)^J$ that satisfy \eqref{Hlambda} and the $M$
inequalities in the second part of \eqref{equation}.

We choose a maximal linearly
independent set of $R$ rows $Z_{1}, \ldots, Z_{R}$ of the matrix
$(\mathscr{A}_1\, \mathscr{A}_2)$.  Let $Z_{R+1}, \ldots, Z_{J}$ be the
remaining rows of $(\mathscr{A}_1\, \mathscr{A}_2)$. As in \cite[(8.23), (8.24)]{BBDG} let
$\mathscr{B} = (b_{kl}) \in \mathbb{R}^{(J-R) \times R}$ be the unique matrix
with
\begin{equation*}
  \mathscr{B}  \left(\begin{smallmatrix}  Z_{1}  \\ \vdots \\   Z_{R} 
    \end{smallmatrix}\right) = \left(\begin{smallmatrix} Z_{R+1} \\ \vdots \\
      Z_{J} \end{smallmatrix}\right).
\end{equation*}
Under Hypothesis 2 below (cf.\ \eqref{rk}), the
last row $(\mathscr{A}_3\, \mathscr{A}_4)$ of $\mathscr{A}$ can be written as a linear
combination of $Z_1, \ldots, Z_R$, say
\begin{equation*}
 \sum_{\ell = 1}^R   b_{\ell} Z_{\ell} =  (\mathscr{A}_3\, \mathscr{A}_4).
\end{equation*}

Suppose these $R$ rows are indexed by
a set $I$ of pairs $(i, j)$ with $0 \leq i \leq k$, $1 \leq j \leq J_i$ with
$|I| = R$. As in \cite[(9.1)]{BBDG} let
\begin{equation*}
\Phi^{\ast}(\textbf{t}) =  \sum_{i = 1}^k\prod_{(i, j) \in I}  t_{ij}^{h_{ij}},
\end{equation*}
and let $\mathscr{F}$ be the affine $(R-1)$-dimensional hypersurface
$\Phi^{\ast}(\textbf{t}) = 0$ over $\mathbb{R}$. Let $\chi_I$ be the characteristic
function on the set
$$\prod_{(i, j) \in I} |t_{ij}|^{\alpha^\mu_{ij}} \leq 1, \quad 1 \leq \mu \leq N.$$
 and define the surface integral
  \begin{equation*}
    c_{\infty}    = 2^{J-R} \int_{\mathscr{F}} \frac{ \chi_I(\textbf{t}) }{ \|
      \nabla \Phi^{\ast}(\textbf{t})\|}\, {\rm d}\mathscr{F}\textbf{t}.
  \end{equation*}
Let
\begin{equation*}
  c_{\text{fin}}=   \prod_p \lim_{L \rightarrow \infty}\frac{1}{p^{L (J-1)}}
  \#\Bigg\{\textbf{x}  \bmod{ p^{L}} :
  \begin{array}{l}
    \displaystyle \sum_{i=1}^k \prod_{j=1}^{J_i}
    x_{ij}^{h_{ij}} \equiv 0 \bmod{p^{L}},\\
    (\{x_{ij} : (i, j) \in S_{\rho}\}, p) = 1 \text{ for } 1 \leq \rho \leq r
  \end{array}\Bigg\}
\end{equation*}
and
\begin{equation*}
  c^{\ast}  =   \text{vol}\Big\{ \textbf{r} \in [0, \infty]^{J-R}  :
  b_{\ell}  -\sum_{\iota=1}^{J-R}  r_{\iota}  b_{\iota \ell}   \geq 0
  \text{ for all }  1 \leq \ell \leq R \Big\} ,
  \end{equation*}
cf.\ \cite[(9.3), (8.36), (8.34)]{BBDG}.

Suppose that the following three hypotheses hold:\\

1)  \cite[Hypothesis 1]{BBDG} The singular series  $\mathscr{E}_{{\bm \gamma}^{\ast}}$ is absolutely
convergent and moreover
\begin{equation}\label{E}
    \mathscr{E}_{{\bm \gamma}^{\ast}} \ll  \gamma_1^{\beta_1} \gamma_2^{\beta_2}\cdots \gamma_k^{\beta_k}
  \end{equation}
  for some $\beta_1,\dots,\beta_k \le 1$. 
  Further, there exist finitely many $\bm\zeta_1 = (\zeta_{11}, \ldots, \zeta_{1k}), \ldots, \bm \zeta_t = (\zeta_{t1}, \ldots \zeta_{tk}) \in \mathbb R^{k}$ satisfying
  \begin{equation}\label{zeta1}
    \zeta_{\nu i} > 0  \text{ for all }  1 \leq i \leq k, \quad
    h_{ij} \zeta_{\nu i} < 1\text{ for all }i, j, \quad
    \sum_{i=1}^k \zeta_{\nu i} = 1
  \end{equation}
  for all $1 \leq \nu \leq t$,  real numbers $0 < \lambda \leq 1$, $\delta_1>0$ and $C\ge 0$ and a set $\mathscr{J}$  of subsets of pairs $(i, j)$ with $0 \leq i \leq k$, $1 \leq j \leq J_i,$   such that
  whenever $\mathbf{X} \in [1, \infty)^{J}$ obeys the conditions
  \begin{equation}\label{samesize}
  \begin{split}
   & \min_{0 \leq i \leq k} \mathbf X_i^{\mathbf h_i} \geq \bigl(\max_{1 \leq i
      \leq k} \mathbf X_i^{\mathbf h_i}\big)^{1-\lambda}, \\
      &   \min_{(ij) \in J}X_{ij} \geq \max_{(ij) \in J}X_{ij} \cdot \max(X_{ij})^{-\lambda}, \quad J 
    \in \mathscr{J},
      \end{split}
  \end{equation}
  then uniformly in $\mathbf b\in(\mathbb Z\setminus\{0\})^k$, one has 
  \begin{equation}\label{errorterm1}
    \mathscr{N}_{{\bm \gamma}^{\ast}}(\mathbf{X}) - \mathscr{E}_{{\bm \gamma}^{\ast}}
    \mathscr{I}_{{\bm \gamma}^{\ast}}(\mathbf{X})
    \ll  (\gamma_1 \cdots \gamma_k)^C (\min_{ij}X_{ij})^{-\delta_1} \sum_{\nu=1}^t  \prod_{i=0}^k
    \prod_{j=1}^{J_i} X_{ij}^{1 - h_{ij}\zeta_{\nu i} + \varepsilon}.
  \end{equation}

2) \cite[(7.4), (7.6), (8.5), (8.6), Hypothesis 2]{BBDG} Suppose that 
\begin{equation}\label{rk}
    \text{rk}(\mathscr{A}_1) = \text{rk}(\mathscr{A}).
  \end{equation}  
     For $\lambda$ and $\mathscr{J}$ as above suppose that there
  exist
  ${\bm \eta} = (\eta_{ij}) \in \mathbb{R}_{> 0}^J$, ${\bm \zeta}$ as\footnote{Note that in \cite[(7.11)]{BBDG}, \emph{any} $\bm \zeta$ does the job, contrary to the statement this has nothing to do with the ${\bm \zeta}$ in Hypothesis 5.1.} in \eqref{zeta1} and $ \delta_2,
  \delta_2^{\ast} > 0$ with the following properties:
 \begin{equation}\label{3}
   C_1({\bm \eta}) \colon  \quad 
   \sum_{(i, j) \in S_{\rho}}\eta_{ij} \geq 1+\delta_2 \quad \text{for all} \quad 1 \leq \rho \leq r, 
 \end{equation}
 \begin{equation}\label{2a}
   N_{\bm \gamma, \bm{\gamma} \cdot \mathbf{y}}(B, H,  \mathscr{J})
   \ll  B (\log B)^{c_2  -1+\varepsilon} (1+\log H)  {\bm \gamma}^{-{\bm  \eta}}
   \langle\mathbf{y}\rangle^{-\delta_2^{\ast}}
 \end{equation}
 and
 \begin{equation}\label{continuous}
\sum_{\nu=1}^t   \int_{\mathscr{S}_{\textbf{y}}(B, H, \mathscr{J})} \prod_{ij}
   x_{ij}^{-h_{ij}\zeta_{\nu i}}
   {\rm d}\mathbf{x} \ll B (\log B)^{c_2  -1+\varepsilon} (1+\log H)
   \langle \mathbf{y}\rangle^{-\delta_2^{\ast}}
 \end{equation} 
 for any $\varepsilon > 0$. In addition, suppose that there is some $\delta_4 >0$ with
\begin{equation}\label{1b}
  \sum_{(i,j) \in S_{\rho}} (1 - \beta_i h_{ij}) \geq  1 + \delta_4
  \,\,\, (1 \leq \rho \leq r)
  \quad \text{ and } \quad \beta_ih_{ij} \leq 1 \,\, (1 \leq i \leq k, 1 \leq
  j \leq J_i),
\end{equation}
 and
 \begin{equation}\label{J-cond}
  J_i \geq 2 \quad \text{whenever}\quad  \zeta_i \geq 1/2.
\end{equation}
For any vector ${\bm \zeta}$ satisfying \eqref{zeta1}, where we allow more generally
$\zeta_i \geq 0$, and for arbitrary $\zeta_0 > 0$, we also assume that the
system of $J+1$ linear equations
\begin{equation}\label{1a}
  \begin{split}
    \left(\begin{matrix} \mathscr{A}_1\\ \mathscr{A}_3\end{matrix}\right) {\bm
      \sigma} = \Big(1 - h_{01}\zeta_0, \ldots, 1 - h_{kJ_k}\zeta_k,
    1\Big)^{\top}
  \end{split}
\end{equation}
in $M$ variables has a solution ${\bm \sigma} \in \mathbb{R}_{>0}^{M}$.  \\

 3)  \cite[(9.2)]{BBDG} Assume that 
  one of the $k$ monomials in $\Phi^{\ast}$ consists of only one
    variable,  which has   exponent 1. \\
  
 Then 
 \begin{equation}\label{finalformula}
    N(B) \sim c^{\ast} c_{\text{fin}} c_{\infty} B (\log B)^{c_2}, \quad B \rightarrow \infty.
    \end{equation}
    
    \medskip
 
\textbf{Proof.}  This is   \cite[Theorem 8.4 \& Lemma 9.1]{BBDG}. The only changes are
\begin{itemize}
\item the bound \eqref{E} is only needed for $\textbf{b} =  {\bm \gamma}^{\ast}$, see the displays before (8.19) and (8.35), and the display (8.30) in \cite{BBDG};
\item the bound \eqref{2a} is only needed for $\textbf{b} = {\bm \gamma}$, see (8.12), and the last display in Section 8.2 in \cite{BBDG};
\item  in \eqref{errorterm1} we can afford a sum over finitely many values of $\bm \zeta$. Obviously this has no influence on the argument in \cite[Section 8.3]{BBDG};
\item we inserted some additional inequalities in \eqref{Hlambda} and the corresponding inequalities in \eqref{samesize} (in \cite[(5.11), (7.8)]{BBDG} we had $\mathscr{J} = \emptyset$). This has no impact on the argument in \cite[Section 8]{BBDG}. Only the set $\mathscr{R}_{\delta, \lambda}$ in \cite[Section 8.2]{BBDG} needs to be redefined as 
\[\mathscr{R}_{\delta, \lambda} = \Biggl\{\mathbf{X} = (\textbf{X}_1, \ldots,
\textbf{X}_{k}) \in [1, \infty)^J : \begin{array}{l} \min_{i, j} X_{ij} \geq \max
X_{ij}^{\delta}, \\ \min_{1 \leq i \leq k}   \textbf{X}_i^{\textbf{h}_i}  \geq
\big(\max_{1 \leq i \leq k}
\textbf{X}_i^{\textbf{h}_i}\big)^{1-\lambda}, \\  \min_{(ij) \in J}X_{ij} \geq \max_{(ij) \in J}X_{ij} \cdot \max(X_{ij})^{-\lambda}, \, J 
    \in \mathscr{J}\end{array}\Biggr\}.\]\\
\end{itemize}

The conditions under which \eqref{finalformula} holds, look very complicated, but much of this will be automatic in our applications. All counting problems given in Corollary \ref{cor:counting_problems} are of the form \eqref{torsor} -- \eqref{gcd}. In particular, for $X_1$ we have $r=2$ and 
$$S_1 = \{(1,1), (1,2), (2,1), (3,3))\}, \quad S_2 = \{(3,1), (3,2)\},$$
so for $\textbf{g} = (\eta, \xi) \in \mathbb{N}^2$ we have $$\bm \gamma = (\eta, \eta, \eta, \xi, \xi, \eta) \in \mathbb{N}^6, \quad \bm \gamma^{\ast} = (\eta^2, \eta^2, \eta^2 \xi^2) \in \mathbb{N}^3,$$
and hence $\mathscr{E}_{\bm \gamma^{\ast}} = \mathscr{E}_{\xi}$ in the notation of \eqref{defsing}. 

Hypothesis 1 with $\mathscr{J} = \emptyset$ along with \eqref{1b} and \eqref{J-cond} follows from \cite[Proposition 5.1]{BBDG} exactly as in the proof of \cite[Theorem 10.1]{BBDG} for $X_2, X_3$. For $X_1$ with its special torsor equation, we choose 
\begin{equation}\label{mathscrJ}
   \mathscr{J} = \big\{\{(1,1), (1,2), (2,1)\}, \{(3,1), (3,2)\}\big\}
   \end{equation}
    reflecting the first two inequalities in \eqref{good}. Then Hypothesis 1 follows from Proposition \ref{asymp-prop} and \eqref{sing-bound}, and 
 \eqref{1b} and \eqref{J-cond} are obvious.

 In all cases, \eqref{rk} and \eqref{1a} follow from Remark \ref{rem:smfano2}
 and Hypothesis 3 is satisfied by \eqref{eq:assumption_real_density}, which we
 verify by inspection of $\Sigma'(1) \supset \Sigma(1)$ in each case.

 Thus in order to prove Theorem \ref{thm1} it only remains to check \eqref{3}
 -- \eqref{continuous} and that the constant
 $c^{\ast}c_{\text{fin}} c_{\infty}$ in \eqref{finalformula} agrees with
 Peyre's prediction. For the latter, we apply
 Proposition~\ref{prop:measure_torsor_mod_p^l} (where transformations similar
 to those in Corollary~\ref{cor:counting_problems} will be necessary for
 $X_2$, $X_3$ since $Y_2$, $Y_3$ are singular) and Proposition~\ref{prop:real_density} with
 \eqref{eq:assumption_real_density} since we are in the situation of
 Remark~\ref{rem:smfano1}.

\subsection{The variety $X_1$}

We are given the equation \eqref{eq0} with $J = 6$ variables,  with $r=2$ coprimality conditions
$$(x_{11}, x_{12}, x_{21}, x_{33}) = (x_{31}, x_{32}) = 1$$
and with $N = 8$ height conditions given by the exponent matrix
$$\mathcal{A}_1 = (\alpha_{ij, \nu}) = \left(\begin{smallmatrix} 2& & & 2 & & & & &\\ & 2 & &  & 2& & & &\\   & &2 & & & 2& & &\\ 1& 1& 1&  & & &3 & &\\ &  & & 1 &1 & 1& & 3&\\   & & & & & & 2& 2& 
 \end{smallmatrix}\right) \in \mathbb{R}^{6\times 8}, \quad \mathcal{A}_2 = \left(\begin{smallmatrix} 1& &  -1 \\ 1&  &  -1 \\   & 2& -1  \\  -1& -1& &   \\ -1& -1 & &  \\  -2 & -2& 1&  
 \end{smallmatrix}\right) \in \mathbb{R}^{6\times 3}.$$
We are going to check \eqref{3} -- \eqref{continuous}. With $\mathscr{J}$ as
in \eqref{mathscrJ}, \eqref{2a} follows from Proposition \ref{prop-hyp} with
$\delta_2^{\ast} = 1/4 > 0$ and
${\bm \eta} = \frac{99}{100}(\frac 1 2,\frac 1 2,\frac 1 2,\frac 3 4,\frac 3
4,\frac 1 2)$
satisfying
\eqref{3}.  The continuous version \eqref{continuous} of \eqref{2a} for the
three $\bm \zeta$ values given in \eqref{zeta} proceeds now in exactly in the
same way as the proof of Proposition \ref{prop-hyp}.

Let $\sigma \in \Sigma_{\max}$ be the cone generated by the rays
corresponding to $x_{11}, x_{21}, x_{32}, x_{33}$, let $\rho_0$ be the
ray corresponding to $x_{11}$, and let $\rho_1$ be the ray
corresponding to $x_{31}$; then conditions
\eqref{eq:assumption_real_density} are satisfied. The resulting
leading constant is Peyre's constant by
Proposition~\ref{prop:measure_torsor_mod_p^l} and
Proposition~\ref{prop:real_density}.

\subsection{The variety $X_2$}\label{app1}

For $X_2$ and $X_3$, we choose $\mathscr{J} =  \emptyset$ as in \cite{BBDG} and apply \cite[Proposition 7.6]{BBDG} (as in \cite[Sections 11.4 and 12.4]{BBDG}) to check  \eqref{3} -- \eqref{continuous}. 

By  Corollary~\ref{cor:counting_problems}(b), we have $J=7$ torsor variables $x_{ij}$
with $0 \leq i \leq 3$
satisfying the equation
\begin{equation}\label{particular2}
x_{11} x_{12} + x_{21}^2 + x_{31}x_{32} = 0,
\end{equation}
after changing signs. We have
$r=4$ coprimality conditions  with
\begin{equation}\label{gcd2}
  S_1 = \{(1, 1), (1, 2), (2, 1)\}, \ S_2 = \{(0, 1), (3, 2)\}, \
  S_3 = \{(0, 2), (3, 1)\}, \ S_4 = \{(0, 1), (0, 2)\},
\end{equation}
We have $N=11$ height conditions with corresponding
exponent matrix
\begin{equation*}
\mathcal{A}_1 = \left(\begin{smallmatrix}
2&2 &2 & & & &1 &1 &1 & &2 \\
& & &2 &2 &2 &1 &1 &1 & 2& \\
1& & & 3& & & 3& & & & \\
& 1& & & 3& & & 3& & & \\
& & 1& & & 3& & & 3& & \\
1& 1& 1& &  &  & & & &3/2 &3/2 \\
& & & 1& 1& 1& & & &5/2 &1/2
\end{smallmatrix}\right) \in \mathbb{R}_{\geq 0}^{7
\times 11}, \quad
\mathcal{A}_2 =  \left(\begin{smallmatrix}
    & & -1\\
    & & -1\\
    1 & &-1\\
    1 & &-1\\
    -2& -2 & 1\\
     & 1 &-1\\
     & 1 &-1
  \end{smallmatrix}\right) \in \mathbb{R}^{7 \times 3}.
\end{equation*}

We choose ${\bm \zeta} = (\frac 1 3, \frac 1 3, \frac 1 3)$ satisfying \cite[(5.10),
(8.6)]{BBDG}, $\lambda = 1/25200$ as in \cite[(5.13)]{BBDG}, and
${\bm\tau}^{(2)} = (1,1,\frac 1 2,\frac 1 2,1,\frac 1 2,\frac 1 2)$ satisfying
\cite[(7.18) using the notation (7.13)]{BBDG}. Using a computer algebra system and the notation \cite[(7.30) -- (7.32)]{BBDG}, we  confirm the conditions  
$C_2({\bm\tau}^{(2)}), C_2((1-h_{ij}/3)_{ij})$ from \cite[Proposition 7.6]{BBDG}. With $c_2=2$ we obtain \begin{displaymath}
\begin{split}
\dim(\mathscr{H} \cap \mathscr{P}) &= 2, \\
\dim(\mathscr{H} \cap \mathscr{P}_{ij}) &=
\begin{cases}
  1, &(i,j)=(0,2),(3,1),(3,2),\\
  0, &\text{otherwise},
\end{cases}
\\
\dim(\mathscr{H} \cap\mathscr{P}(1/25200, \pi)) &= 0
\end{split}
\end{displaymath}
for both vectors $(1 - h_{ij}/3)_{ij}$ and ${\bm\tau}^{(2)}$, confirming the two remaining conditions $C_3({\bm\tau}^{(2)}), C_3((1-h_{ij}/3)_{ij})$. Thus \cite[Proposition 7.6]{BBDG} is available and yields the conditions \eqref{3} -- \eqref{continuous}. 

Let $\sigma \in \Sigma_{\max}$ be the cone generated by the rays
corresponding to $x_{11}, x_{21}, x_{31}, x_{32}$, let $\rho_0$ be the
ray corresponding to $x_{11}$, and let $\rho_1$ be the ray
corresponding to $x_{02}$; then conditions
\eqref{eq:assumption_real_density} are satisfied. The correct shape of
the leading constant now follows from
Proposition~\ref{prop:measure_torsor_mod_p^l} and
Proposition~\ref{prop:real_density}.

More precisely, for the $p$-adic densities,
we have $c_{\text{fin}} = \prod_p c_p$ as in
Section~\ref{sec:machinery} with
\begin{equation*}
  c_p = \lim_{L \to \infty} \frac{1}{p^{6L}} \#\Bigg\{\textbf{x}  \bmod{ p^{L}} :
  \begin{array}{l}
    x_{11}x_{12}+x_{21}^2+x_{31}x_{32} \equiv 0 \bmod{p^{L}},\\
    (\{x_{ij} : (i, j) \in S_{\rho}\}, p) = 1 \text{ for } 1 \leq \rho \leq 4
  \end{array}\Bigg\}
\end{equation*}
We note that
\begin{equation*}
  c_p = \left(1-\frac 1 p\right)^{-1}\lim_{L \to \infty} \frac{1}{p^{7L}} \#\Bigg\{(\textbf{x},z_1)  \bmod{ p^{L}} :
  \begin{array}{l}
    x_{11}x_{12}z_1-x_{21}^2z_1-x_{31}x_{32} \equiv 0 \bmod{p^{L}},\\
    (\{x_{ij} : (i, j) \in S_{\rho}\}, p) = 1 \text{ for } 1 \leq \rho \leq 4,\\ (x_{31}x_{32},z_1,p)=1
  \end{array}\Bigg\}
\end{equation*}
because we have the $(\#(\Zd/p^L\Zd)^\times : 1)$-surjection
\begin{equation*}
  (x_{01},\dots,x_{31},x_{32},z_1)\mapsto
  (x_{01},x_{02},x_{11},-x_{12},x_{21},x_{31},x_{32}z_1^{-1})
\end{equation*}
between the two sets since the final coprimality condition and the congruence
imply $z_1 \in (\Zd/p^L\Zd)^\times$. By
Proposition~\ref{prop:measure_torsor_mod_p^l}, this shows that
$c_p=(1-p^{-1})^{\rank \Pic X} \mu_p(X(\Qd_p))$, as expected.

\subsection{The variety $X_3$}

This is very similar to the treatment of $X_2$. We have the same torsor
variables, the same torsor equation \eqref{particular2}, and the same
coprimality conditions \eqref{gcd2}.
We have $N=11$ height conditions   with corresponding exponent matrix
\begin{equation*}
\mathcal{A}_1 = \left(\begin{smallmatrix}
 & & & &1&1&1&2&2&2&2\\
2&2&2&2&1&1&1& & & & \\
 & & &2& & &3& & & &2\\
 & &2& & &3& & & &2& \\
 &2& & &3& & & &2& & \\
1& & & & & & &2&1&1&1\\
2&1&1&1& & & &1& & & 
\end{smallmatrix}\right) \in \mathbb{R}_{\geq 0}^{7 \times 11}
\end{equation*}
and the same $\mathcal{A}_2$ as before.

We work with the same
${\bm \zeta}, \lambda, {\bm\tau}^{(2)},{\bm\beta},\delta_4$. Everything is
identical except \cite[(7.35)]{BBDG} because of the different
$\mathcal{A}_1$. We confirm $C_2({\bm\tau}^{(2)}), C_2((1-h_{ij}/3)_{ij})$ and
compute $c_2=2$ and
\begin{displaymath}
\begin{split}
\dim(\mathscr{H} \cap \mathscr{P}) &= 2, \\
\dim(\mathscr{H} \cap \mathscr{P}_{ij}) &=
\begin{cases}
  1, &(i,j)=(0,1),(0,2),(3,1),(3,2),\\
  0, &\text{otherwise},
\end{cases}
\\
\dim(\mathscr{H} \cap\mathscr{P}(1/25200, \pi)) &= 0
\end{split}
\end{displaymath}
for the vector $(1 - h_{ij}/3)_{ij}$, and the same for the vector
${\bm\tau}^{(2)}$.

For the leading constant, let $\sigma \in \Sigma_{\max}$ be the cone
generated by the rays corresponding to
$x_{11}, x_{21}, x_{31}, x_{32}$, let $\rho_0$ be the ray
corresponding to $x_{11}$, and let $\rho_1$ be the ray corresponding
to $x_{02}$; then conditions \eqref{eq:assumption_real_density} are
satisfied.

\end{document}